\documentclass[11pt]{article}
\usepackage[T1]{fontenc}
\usepackage[utf8]{inputenc}
\usepackage{lmodern,amsmath,amssymb,amsthm,mathtools,mathrsfs,microtype}
\usepackage[margin=1in,headheight=14pt]{geometry}
\usepackage{xcolor,tikz,float}
\usepackage[colorlinks=true,allcolors=black]{hyperref}
\usepackage{bookmark,enumitem,needspace,booktabs,tabularx}

\allowdisplaybreaks[2]
\setlist{topsep=4pt,itemsep=3pt,leftmargin=*}
\newtheorem{theorem}{Theorem}[section]
\newtheorem{proposition}[theorem]{Proposition}
\newtheorem{lemma}[theorem]{Lemma}
\newtheorem{corollary}[theorem]{Corollary}
\newtheorem{conjecture}[theorem]{Conjecture}
\theoremstyle{remark}
\newtheorem{remark}[theorem]{Remark}
\newcommand{\Ric}{\operatorname{Ric}}
\newcommand{\Rm}{\operatorname{Rm}}
\newcommand{\Hess}{\operatorname{Hess}}
\newcommand{\diver}{\operatorname{div}}
\newcommand{\Vol}{\operatorname{Vol}}
\newcommand{\Area}{\operatorname{Area}}
\numberwithin{equation}{section}
\hypersetup{pdftitle={Positive Einstein metrics on infinitely many four-manifolds},
pdfsubject={Conic surface analysis, angle deformation, and geometric compactness}}

\title{Positive Einstein metrics on infinitely many four-manifolds}
\author{Tristan Ozuch}
\makeatletter
\renewcommand{\maketitle}{%
  \begin{center}
    {\Large\bfseries\@title\par}\medskip
    {\large\@author\par}%
  \end{center}\medskip}
\makeatother

\begin{document}
\maketitle
\begin{abstract}
We construct smooth Einstein metrics with positive scalar curvature on \(\#_k({\mathbb S^2}\times {\mathbb S^2})\) and on \(k\mathbb{CP}^2\#(2k+1)\overline{\mathbb{CP}}{}^{2}\), for \(k\geqslant2\). In particular, we show that positive Einstein metrics exist on infinitely many homeomorphism types of closed $4$-manifolds.

The first family shows that any simply connected spin $4$-manifold
carrying metrics with positive scalar curvature is homeomorphic to an Einstein $4$-manifold. Wick rotation gives associated Lorentzian static axisymmetric Einstein metrics with positive cosmological constant and arbitrarily many horizons with the same surface gravity.

The metrics are obtained by varying the cone angle $2\pi\beta$ of a prescribed singularity along a
torus, starting from the product metric on $\mathbb S^2\times\mathbb S^2$ and
from the symmetric K\"ahler-Einstein metric on
$\mathbb{CP}^2\#3\overline{\mathbb{CP}}{}^2$.
At $\beta=\frac{2}{k+1}$, changing
the period of the angle around the torus removes the cone singularity and changes the topology.

For both families, as $k\to+\infty$, the metrics collapse along the torus orbits to an explicit $2$-dimensional disk, whose boundary has Hausdorff dimension $\frac{4}{3}$. At the scale of curvature, they converge to complete periodic Ricci-flat metrics with Kasner asymptotics: the Euclidean Myers-Korotkin-Nicolai metric for the first family, and a new one for the second.

\end{abstract}

\section{Introduction}\label{sec:introduction}

A Riemannian metric $g$ is Einstein if, for a constant $\lambda$, it solves
\begin{equation}\label{eq:einstein}
 \Ric(g)=\lambda g.
\end{equation}
Einstein metrics are the natural generalizations of constant curvature metrics in higher dimensions and are candidates for optimal metrics on a given topology \cite{Bes87}.

In dimension four, the Chern-Gauss-Bonnet formula reads, for a closed four-manifold $M$,
\begin{equation}\label{eq:CGB}
 \int_M|\Rm_g|^2\,dV_g=32\pi^2\chi(M)+4\int_M|\Ric_g^\circ|^2\,dV_g,
\end{equation}
where $\Ric_g^\circ$ is the trace-free Ricci tensor, so Einstein metrics minimize $\int_M|\Rm_g|^2\,dV_g$. Their existence is strongly constrained by the topology through the classical Hitchin-Thorpe inequality  \cite{Hit74,Tho69} and LeBrun's Seiberg-Witten refinements \cite{LeB96,LeB01}.
A central question is \cite[Introduction, Section~I, p.~18]{Bes87}:
\begin{quote}
\emph{``Are manifolds admitting an Einstein metric rather scarce or numerous?''}
\end{quote}
In dimension four, Einstein metrics with positive scalar curvature have so far been rare, and a concrete question is:
\emph{Do infinitely many pairwise nonhomeomorphic closed four-manifolds
admit Einstein metrics with positive scalar curvature?}
Theorem~\ref{thm:main} answers this question affirmatively.

\subsection{Known examples of Einstein $4$-manifolds and main result}

Our limited understanding of Einstein $4$-metrics is often attributed to the lack of examples. Constructing Einstein metrics to uncover the landscape of these metrics has a long history, whose successes can be summarized as follows.
\begin{itemize}
    \item Tools from complex geometry and the simplifications induced by a (conformally) K\"ahler structure \cite{Aub78,Yau78,Tia90,CLW08}. This heavily biases our intuition about Einstein metrics: for instance, all known compact Ricci-flat $4$-manifolds are flat or locally hyperk\"ahler \cite{And10,LeB26}.
    \item Cohomogeneity one symmetries reduce the Einstein equation to ordinary differential equations \cite{Boh98,PP87,LPP04}, but in dimension $4$ they only reach a few simply connected topologies \cite{Par86}, on which the known examples are conformally flat or conformally K\"ahler.
    \item Gluing constructions have had some success when gluing together rigid models with negative curvature: the Dehn fillings of cusped hyperbolic manifolds \cite{And06,Biq08,Bam12} and the smoothings of branched covers of hyperbolic manifolds \cite{FP20,HJ24,GH25}. Outside negative curvature, this almost never happens: stringent obstructions to gluing exist in general \cite{Biq13,Ozu21,Ozu22a,Ozu22b,Ozu24,HO}, even when gluing together K\"ahler pieces \cite{LO26}.
\end{itemize}
In the present article, we introduce a different, purely Riemannian construction: we deform the cone angle of a prescribed singularity along a torus.
\medskip 

Until now, the known simply connected examples with positive scalar curvature were the round ${\mathbb S^4}$, the
Fubini-Study metric on $\mathbb{CP}^2$, the product on ${\mathbb S^2}\times {\mathbb S^2}$,
Page's metric on
$\mathbb{CP}^2\#\overline{\mathbb{CP}}{}^{2}$ \cite{Pag78}, Tian's construction of
K\"ahler-Einstein metrics on the del Pezzo surfaces
$\mathbb{CP}^2\#r\overline{\mathbb{CP}}{}^{2}$ for $3\leqslant r\leqslant8$
\cite{Tia90}, and the Chen-LeBrun-Weber metric on
$\mathbb{CP}^2\#2\overline{\mathbb{CP}}{}^{2}$ \cite{CLW08}. Apart from the conformally flat round sphere, all of them are K\"ahler or conformally K\"ahler, and LeBrun classified  conformally Kähler Einstein $4$-metrics \cite{LeB12}.
The examples with positive scalar curvature above only reach finitely many homeomorphism types, with $b_2\leqslant9$. In higher dimensions, many positive Einstein metrics are found with symmetries \cite{WZ90,Boh98,PP87,LPP04}.

We construct smooth Einstein metrics with positive
scalar curvature and arbitrarily large second Betti number on both spin and non-spin $4$-manifolds. 
\begin{theorem}\label{thm:main}
For every $k\geqslant 2$, the smooth manifolds
\[
 M_k:=\#_k({\mathbb S^2}\times {\mathbb S^2}),\qquad \text{ and }\qquad
 N_k:=k\mathbb{CP}^{2}\#(2k+1)\overline{\mathbb{CP}}{}^{2}
\]
admit Einstein metrics with positive scalar curvature.
\end{theorem}
By Lichnerowicz \cite{Lic63}, a closed simply connected spin $4$-manifold with positive scalar curvature has signature zero, hence it is homeomorphic to ${\mathbb S^4}$ or to some $M_k$ by Freedman \cite{Fre82}. All of these carry metrics with positive scalar curvature \cite{GL80}, so by Theorem~\ref{thm:main} each is homeomorphic to a manifold carrying a positive Einstein metric.

The metrics on $M_k$ have an effective isometric ${\mathbb T^2}$-action with orthogonal and hypersurface-orthogonal generators, and they have \textit{pure curvature operator}: in some orthonormal frame $(e_i)$, the curvature operator is diagonal in the basis $(e_i\wedge e_j)_{i<j}$ of $\Lambda^2$. We also show that they are all linearly unstable in Proposition \ref{prop:square-instability}.
\medskip 

\noindent \textit{Lorentzian examples. }By Wick rotation, the metrics on $M_k$ also yield static Lorentzian Einstein metrics with positive cosmological constant. They have $k+1$ nondegenerate Killing horizons with equal surface gravities, hence matching Hawking temperatures (Proposition~\ref{prop:wick}). To the best of our knowledge, the closest Lorentzian examples are the static axisymmetric metrics found numerically by Dias-Gibbons-Santos-Way \cite{DGSW23}. Note, however, that the Lorentzian
setting allows much weaker regularity conditions, which would correspond
to Riemannian metrics with conical singularities along surfaces.
\medskip

\noindent 
\emph{Collapse.} As $k\to+\infty$, the metrics of Theorem~\ref{thm:main} collapse to an explicit disk $B_\infty$, the orbit space of the Kottler metric with negative mass and cosmological constant $3$, whose boundary circle has Hausdorff dimension $\frac{4}{3}$ (Theorem~\ref{thm:collapsed-limit}).
\medskip

\noindent 
\emph{Finite fundamental groups.}
For every integer $k\geqslant1$, a quotient of the metrics of Theorem \ref{thm:main} by a free $\mathbb Z_{k+1}$-action gives a closed oriented Einstein four-manifold $(Z^{(k)},g_Z^{(k)})$ with
\[
 \Ric(g_Z^{(k)})=3g_Z^{(k)},\qquad \pi_1(Z^{(k)})=\mathbb Z_{k+1},\qquad\text{and}\qquad  \Vol(Z^{(k)},g_Z^{(k)})=o(k^{-1}).
\]
In the square family, we have $H_*(Z^{(k)};\mathbb Q)\approx H_*({\mathbb S^4};\mathbb Q)$ and $\int_{Z^{(k)}}|\Rm|^2\,dV=64\pi^2$, the minimum for closed oriented positive Einstein four-manifolds. In the hexagon family, we find $H_*(Z^{(k)};\mathbb Q)\approx H_*(\overline{\mathbb{CP}}{}^2;\mathbb Q)$ and $\int_{Z^{(k)}}|\Rm|^2\,dV=96\pi^2$ (Corollary~\ref{cor:finite-fundamental-groups}). In particular, Einstein metrics with $\Ric=3g$ and fixed curvature energy can have arbitrarily large finite fundamental groups. For $k=1$, both quotients $Z^{(1)}$ occur in LeBrun's classification \cite[Propositions~2 and~3]{LeB21}.

As $k\to+\infty$, the quotients $(Z^{(k)},g_Z^{(k)})$ converge in the Gromov-Hausdorff sense to a segment of length $\frac{\pi}{3}$, with $|\Rm|\to+\infty$ at one end only (Remark~\ref{rem:segment-limit}). Blow-ups at the corners converge to complete periodic Ricci-flat limits: an equal-rod Myers-Korotkin-Nicolai quotient in the square family, and a non-spin metric on $\overline{\mathbb{CP}}{}^2\setminus {\mathbb T^2}$ which is not locally Hermitian in either orientation in the hexagon family (Theorem~\ref{thm:periodic-models}). The corresponding blow-ups of the original metrics are their periodic extensions (Remark~\ref{rem:covers}).

\subsection{The two initial metrics and the angle deformation}
Our construction starts from one of the following two smooth metrics, both denoted by $g_1$.
\medskip

\noindent 
\emph{The square.}  On
$M_1={\mathbb S^2}\times {\mathbb S^2}$, let $g_{{\mathbb S^2}}$ be the round unit metric and take $g_1:=\tfrac13(g_{{\mathbb S^2}}+g_{{\mathbb S^2}})$.
Rotation around an axis in each sphere gives a ${\mathbb T^2}$-action with square quotient.
\medskip

\noindent 
\emph{The hexagon.} On $N_1=\mathbb{CP}^{2}\#3\overline{\mathbb{CP}}{}^{2}$,
take the toric K\"ahler-Einstein metric with the dihedral symmetries of
its hexagonal ``fan'' toric quotient
\cite[Theorem~1.1 and Example~4.1]{BS99}.

Let $M$ be $M_1$ or $N_1$,  let $\pi:M\to\overline\Omega:=M/{\mathbb T^2}$ be the quotient map, write $o$ for the center of $\overline\Omega$, and set $\Sigma:=\pi^{-1}(o)$.  A cone angle $2\pi\beta$ along $\Sigma$ means
that, to leading order in normal polar coordinates, the metric $g$ is
$dr^2+\beta^2r^2d\theta^2+g_\Sigma$, with $\theta$ of period $2\pi$
and $g_\Sigma$ a smooth metric on $\Sigma$.
\begin{theorem}\label{thm:angle}
For either initial metric $g_1$, there is a family
$(g_\beta)_{0<\beta\leqslant1}$ starting at $g_1$, with
\[
 \Ric(g_\beta)=3g_\beta\quad\text{on }M\setminus\Sigma,\quad \text{ and }\quad\text{ conical singularity of angle $2\pi\beta$ along $\Sigma$}.
\]
They are smooth on $M\setminus\Sigma$ and invariant under the torus action and under the finite  group $\mathcal G$ of Section~\ref{subsec:symmetries}. They depend smoothly on $\beta$ in the weighted spaces of Section~\ref{sec:spaces}, in the sense of Corollary~\ref{cor:local}.
\end{theorem}

The finite symmetries \eqref{eq:group-actions} and \eqref{eq:hexagon-lifts} make the metric coefficients $h$ and $G$ of \eqref{eq:intro-matrix-ansatz} $\pi$-periodic in $\theta$. We may therefore define $g_\beta^{(k)}$ by the same metric formula with
$\theta\in\mathbb R/(k+1)\pi\mathbb Z$.
Its cone angle along $\Sigma$ is $(k+1)\pi\beta$, which equals $2\pi$ at $\beta=\frac{2}{k+1}$.
Section~\ref{sec:proof-main} shows that the resulting metrics $g^{(k)}_{2/(k+1)}$ are smooth Einstein metrics on $M_k$ and $N_k$. The orbit space then has $2(k+1)$ sides in the square construction and $3(k+1)$ sides in the hexagon construction, as in Figure~\ref{fig:covering}.
\begin{figure}[htbp]
\centering
\begin{tikzpicture}[every node/.style={font=\small},
  first/.style={thick},second/.style={thick,dashed},
  third/.style={thick,densely dotted}]
\foreach \d/\xx/\kk/\angle in {2/-3.8/1/1,3/0/2/{2/3},4/3.8/3/{1/2}}{
 \begin{scope}[xshift=\xx cm]
  \node at (0,1.5) {$k=\kk$, $\beta=\angle$};
  \foreach \row/\mult/\offset in {0/2/45,1/3/30}{
   \begin{scope}[yshift=-3.1*\row cm]
    \pgfmathtruncatemacro{\n}{\mult*\d}
    \pgfmathtruncatemacro{\last}{\n-1}
    \pgfmathsetmacro{\handle}{.5*sin(180/\n)}
    \foreach \j in {0,...,\last}{
     \pgfmathsetmacro{\ang}{360*\j/\n+\offset}
     \pgfmathsetmacro{\nang}{360*(\j+1)/\n+\offset}
     \ifnum\row=0\relax
      \pgfmathtruncatemacro{\kind}{mod(\j,2)*2}
     \else
      \pgfmathtruncatemacro{\kind}{mod(\j,3)}
     \fi
     \ifcase\kind\relax\def\sty{first}\or\def\sty{second}\or\def\sty{third}\fi
     \draw[\sty]
      ({.9*cos(\ang)},{.9*sin(\ang)})
      .. controls
       ({.9*cos(\ang)+\handle*cos(\ang+135)},
        {.9*sin(\ang)+\handle*sin(\ang+135)}) and
       ({.9*cos(\nang)-\handle*cos(\nang+45)},
        {.9*sin(\nang)-\handle*sin(\nang+45)})
      .. ({.9*cos(\nang)},{.9*sin(\nang)});
     \draw[thin]
      ({.9*cos(\ang)+.055*cos(\ang+135)},
       {.9*sin(\ang)+.055*sin(\ang+135)})
      -- ({.9*cos(\ang)+.055*(cos(\ang+135)-cos(\ang+45))},
          {.9*sin(\ang)+.055*(sin(\ang+135)-sin(\ang+45))})
      -- ({.9*cos(\ang)-.055*cos(\ang+45)},
          {.9*sin(\ang)-.055*sin(\ang+45)});
    }
    \fill (0,0) circle (1.5pt);
    \ifnum\row=0\relax
     \node at (0,-1.25) {$M_{\kk}$};
    \else
     \node at (0,-1.25) {$N_{\kk}$};
    \fi
   \end{scope}
  }
 \end{scope}}
\end{tikzpicture}
\caption{Orbit spaces for the square
construction (top) and hexagon construction (bottom).
The dot represents $\Sigma$.
Solid, dashed and dotted sides have rod vectors $(1,0)$, $(1,1)$ and
$(0,1)$, respectively.}
\label{fig:covering}
\end{figure}
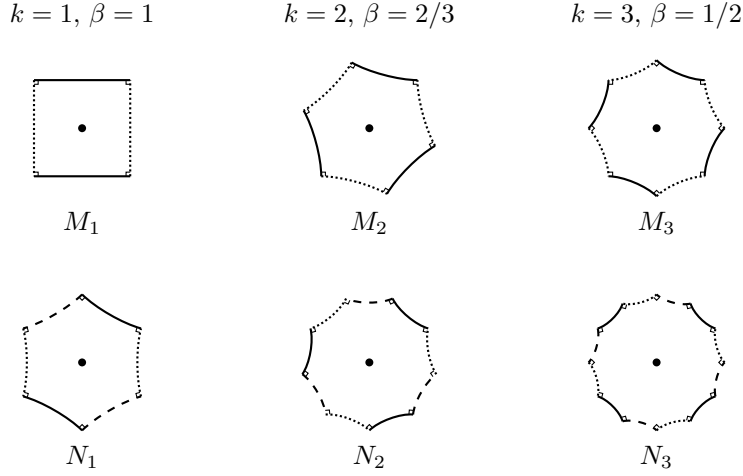

This can be compared with Grove-Ziller \cite[Theorem~3.3 and Examples~4.2-4.4]{GZ12} who
construct polar metrics from Riemannian orbit polygons, including on connected sums of
$\mathbb S^2\times\mathbb S^2$ or with
Bazaikin-Matvienko \cite{BM07} who use warped products over
right-angled polygons and torus quotients to construct invariant metrics with positive Ricci
curvature.

\medskip 
We choose $\Sigma$ to be a principal torus orbit for the following reason. By Atiyah-LeBrun \cite[Theorems~2.1 and~2.2]{AL13}, as in the smooth case, an Einstein metric with a singularity as in Theorem \ref{thm:angle} satisfies the edge-cone Hitchin-Thorpe inequality
\begin{equation}\label{eq:edge-HT}
 2\chi(M)\pm3\tau(M)\geqslant(1-\beta)\bigl(2\chi(\Sigma)\pm(1+\beta)[\Sigma]^2\bigr).
\end{equation}
A principal torus orbit has $\chi(\Sigma)=0$ and $[\Sigma]^2=0$. So the right-hand side of \eqref{eq:edge-HT} vanishes for every $\beta$, and the constraint reduces to the Hitchin-Thorpe inequality of $M$.

Using the cone angle as a continuity parameter is a classical idea which was applied successfully to K\"ahler-Einstein metrics with cone singularities along a divisor  \cite{Don12,JMR16,CDS15a,CDS15b,CDS15c,Tia15,BG22}.
In the Ricci-flat case, Li-Sun \cite{LS25} recently constructed complete Ricci-flat metrics on infinitely many topologies. There, the angles $2\pi$ are obtained by a degree argument.

\subsection[The reduced equations]{The reduced equations}\label{subsec:intro-reduced}

On the set of two-dimensional torus orbits, write the metric $g$ as
\begin{equation}\label{eq:intro-matrix-ansatz}
 \begin{gathered}
 g=h+G_{ij}\,d\phi_i\,d\phi_j,
 \qquad G=G^{\mathsf T}>0,\\
 \rho:=\sqrt{\det G},\qquad \Phi:=\rho^{-1}G:\Omega\longrightarrow\mathbb H^2,\\
 \mathbb H^2:=\{S=S^{\mathsf T}>0:\det S=1\},\qquad\text{and}\qquad
 g_{\mathbb H^2}:=\tfrac12\operatorname{tr}(S^{-1}dS\,S^{-1}dS).
 \end{gathered}
\end{equation}
Here $h$ is the quotient metric, the angles $\phi_1,\phi_2$ have period $2\pi$, repeated indices are summed, $4\pi^2\rho$ is the area of the torus orbits and $\Phi_{ij}\,d\phi_i\,d\phi_j$ is the flat metric of the torus orbit rescaled to area $4\pi^2$. In the square case,
$G=\operatorname{diag}(u^2,v^2)$ and
$\Phi=\operatorname{diag}(\frac{u}v,\frac{v}u)$ lies on a geodesic of $\mathbb H^2$.

We solve the following \textit{reduced Einstein equations}, consequences of the Einstein equation \eqref{eq:einstein} with $\lambda=3$:
 \begin{equation}\label{eq:intro-reduced}
 \Delta_h\rho=-6\rho,\qquad
 \tau_h(\Phi)+d\Phi(\nabla_h\log\rho)=0,\qquad\text{and}\qquad
 4K_h=|d\Phi|_h^2-|d\log\rho|_h^2,
\end{equation}
where $K_h$ is the Gaussian curvature of $h$ and $\tau_h(\Phi)$ is the tension field \eqref{eq:matrix-tension} and we seek solutions of \eqref{eq:intro-reduced} for which the quotient metric $h=h_\beta$ has cone angle $2\pi\beta$ at $o$.

For the square family, whose metrics are static, the reduced equations \eqref{eq:intro-reduced} are the Weyl-Papapetrou reduction with cosmological constant for static axisymmetric metrics \cite[Section~3.4]{CLSZ07}.

\subsection{Strategy and analytic ingredients}\label{sec:strategy}
The proof is a continuity method parametrized by the cone angle along $\Sigma$: starting from $g_1$ at $\beta=1$, we decrease $\beta$ and show that the set of $\beta$ reached by Einstein metrics $g_\beta$ with cone angle $2\pi\beta$ along $\Sigma$ is both open and closed in $(0,1]$, hence equal to $(0,1]$. In particular, it contains all of the values $\beta=\frac{2}{k+1}$ for $k\geqslant 1$.

Throughout the deformation, we impose the finite symmetries of Section~\ref{sec:geometry}, which fix the central point $o$ and the value $\Phi(o)$. They restrict the Fourier frequencies in the polar angle $\theta$ near $o$ (Lemma~\ref{lem:central-modes}) and dictate our choice of weighted spaces.
\medskip

\noindent
\emph{Analytic setup. }
The prescribed singularity along $\Sigma = \pi^{-1}(\{o\})$ is a cone point of $h$ \textit{only}: $G$ extends to $o$ with $G(o)>0$. Near $o$, we use the conic surface analysis of Mazzeo-Weiss \cite[Section~4]{MW17}.  Deforming a cone angle along a surface is an edge problem, whose linearization can have infinite-dimensional kernel or cokernel \cite{Maz91}. Torus invariance removes the dependence along the singular torus and reduces it to a conic problem on the orbit surface. Along the rods and at the corners, the function spaces encode the smoothness conditions, and the Dirichlet problem at the corners is solved with an estimate of Grisvard (Lemma~\ref{lem:grisvard}).
\medskip

\noindent
\emph{Recovery of the Einstein equation.} For a solution of \eqref{eq:intro-reduced}, the tensor $T:=\rho(\Ric_g-3g)|_{T\Omega}$ may a priori not vanish but $T$ is trace-free and divergence-free, and we show that this, together with the rod angle conditions \eqref{eq:angle-equalities}, forces $T=0$ (Proposition~\ref{prop:recovery}).
\medskip

\noindent
\emph{Rigidity at fixed angle.} Local continuation by the implicit function theorem with parameter $\beta$ requires the linearization at fixed angle to be invertible (Section~\ref{sec:kernel}). Its kernel vanishes by an integral identity whose terms are all nonnegative: a singular Lichnerowicz-Obata inequality (Proposition~\ref{prop:spectrum}) controls the variation of $\rho$, and the negative curvature of $\mathbb H^2$ controls the variation of $\Phi$. The boundary terms vanish thanks to the rod angle conditions \eqref{eq:angle-equalities}. Comparison with an invertible operator shows that the linearization is Fredholm of index zero, hence invertible (Proposition~\ref{prop:invertibility}).
\medskip

\noindent
\emph{Compactness.} Section~\ref{sec:compactness} proves compactness on compact subintervals of $(0,1)$. Concavity of $\sqrt\rho$ controls the level sets of $\rho$ (Corollary~\ref{cor:concavity}), and a Schl\"afli formula with an area-volume inequality bounds the volume from below on compact angle intervals (Proposition~\ref{prop:Schlafli} and Corollary~\ref{cor:volume-branch}). Near $\Sigma$, the reduced equations \eqref{eq:intro-reduced} alone  bound $|\Rm|$ on a tube of fixed size (Lemma~\ref{lem:curv-central}). At the rods and corners, we follow \cite[Section~3.1]{Don08}: the fixed-point set of the isometry $\mathsf c$ completes to a \textit{closed} surface $\bar S$, made of four copies of $\overline\Omega$ glued along the rods, on which the points of $\partial\Omega$ are interior points. This lets us bound the conformal factor (Section~\ref{sec:reflected-section} and Lemma~\ref{lem:coord-conformal}) without dealing with boundary effects. A curvature blow-up elsewhere would give a nonflat Ricci-flat ALE space, excluded by the rods of our families (Corollary~\ref{cor:ale-exclusion}).
\medskip

\noindent
\emph{Degeneration.} As $k\to+\infty$, $\Vol(g^{(k)})\to0$ and $\sup_q\operatorname{diam}\pi^{-1}(q)\leqslant C\sqrt{\rho_k(o)}\to0$ by Lemma~\ref{lem:scalar}(c).
The angle of the rotation $\mathsf r$ by one rod tends to zero, so the limiting reduced data are rotationally symmetric, and the reduced equations \eqref{eq:intro-reduced} identify the limit disk $B_\infty$ (Theorem~\ref{thm:collapsed-limit}).

In Section~\ref{sec:questions}, we discuss possible extensions of the construction starting from the Fubini-Study, Page and Chen-LeBrun-Weber metrics, see Figure~\ref{fig:reachable-topologies}. Appendix~\ref{sec:appendix-auxiliary} collects technical proofs.

\subsection*{Acknowledgements}
The author would like to thank Rafe Mazzeo for explaining to him, years ago, the (many) difficulties of deforming Einstein metrics with conical singularities, as well as his works \cite{Maz91} and \cite{MW17}.
This work is supported by the National Science Foundation under Grant No.~DMS-2405328.

\subsection*{Use of AI tools}
In the preparation of this article, the author used large language models for the following purposes:
\begin{enumerate}
    \item running numerical tests to confirm the existence of the metrics and to test specific steps of the proof, 
    \item locating references in the literature, some of which were used to shorten and streamline arguments,
    \item verification of logic, constant dependencies and computations, which also led to occasional shortcuts in arguments, 
    \item formatting and language editing, and
    \item extending the construction for $\#_k({\mathbb S^2}\times {\mathbb S^2})$ to $k\mathbb{CP}^2\#(2k+1)\overline{\mathbb{CP}}{}^{2}$, and merging the two constructions into a single proof. The author deliberately chose to combine these results into one article, so as not to contribute to a flood of incremental publications in differential geometry.
\end{enumerate}

\section{Geometric setup}
\label{sec:geometry}
\subsection{The metrics and their symmetries}\label{subsec:symmetries}

Recall $\pi:M\to\overline\Omega=M/{\mathbb T^2}$ and define
$\Omega:=\overline\Omega\setminus\partial\overline\Omega$.
Topologically, $\overline\Omega$ is a closed disk and as a manifold with right-angle corners, it has four sides in the square case and six in the hexagon case.

The torus angles $\phi_1,\phi_2$ of \eqref{eq:intro-matrix-ansatz} and their lattice remain fixed. We write $\boldsymbol\phi:=(\phi_1,\phi_2)^{\mathsf T}$ and $X_i:=\partial_{\phi_i}$ for the fixed generators. Our metrics have the form
\begin{equation}\label{eq:matrix-ansatz}
 g=h+G_{ij}\,d\phi_i\,d\phi_j,
 \qquad G_{ij}=g(X_i,X_j),\qquad\text{and}\qquad G=G^{\mathsf T}>0.
\end{equation}
The coefficients are functions on $\Omega$, and $d\phi_i\,d\phi_j$ denotes the symmetric product $\tfrac12(d\phi_i\otimes d\phi_j+d\phi_j\otimes d\phi_i)$.
For either initial metric $g_1$, let $h_1$ denote its quotient metric. We write $\rho_1$ and $\Phi_1$ for its coefficients, so that its Gram matrix is $G_1=\rho_1\Phi_1$.

The fibers of $\pi$ over $\Omega$ are the \emph{principal} orbits, those with trivial stabilizer. An open side of $\overline\Omega$ is called a \emph{rod} \cite{Wis03,ER02,CT10}: over it, the stabilizer is the circle generated by $n_1X_1+n_2X_2$ for a primitive pair $(n_1,n_2)\in\mathbb Z^2$, the (up to sign) \emph{rod vector} of that side. Over a corner, the stabilizer is $\mathbb T^2$, and the rod vectors of the two adjacent rods form a lattice basis, as in \cite[Definition~3.5]{LS25}. We write $e_1:=(1,0)$ and $e_2:=(0,1)$ for the standard basis of $\mathbb Z^2$.

Throughout the deformation, we moreover impose a group of symmetries $\mathcal G$, namely $\mathcal G_4$ for the square and $\mathcal G_6$ for the hexagon, described in \eqref{eq:group-actions} and \eqref{eq:hexagon-lifts} in the next subsubsections. We also note that for both families, any two rod vectors $v,v'$ with $v\neq\pm v'$ satisfy
\begin{equation}\label{eq:rod-determinant}
 |\det(v,v')|=1.
\end{equation}

\subsubsection{The square}\label{subsec:square}
Here $\Omega=(-\frac{\pi}{2},\frac{\pi}{2})^2$, with $x,y$ the latitude coordinates
on the two spheres. The initial metric $g_1$ is
\begin{equation}\label{eq:square-initial-coordinates}
 g_1=\tfrac13\bigl(dx^2+dy^2+\cos^2x\,d\phi_1^2+\cos^2y\,d\phi_2^2\bigr).
\end{equation}
Hence the rods $\{x=\pm\frac{\pi}{2}\}$ have rod vector $(1,0)$, the rods $\{y=\pm\frac{\pi}{2}\}$ have rod vector $(0,1)$, and the corners are the four points $(\pm\frac{\pi}{2},\pm\frac{\pi}{2})$.
The initial quotient metric is $h_1=\frac{1}{3}(dx^2+dy^2)$. We fix its conformal class and consider the diagonal case $G=\operatorname{diag}(u^2,v^2)$ of \eqref{eq:matrix-ansatz}:
\begin{equation}\label{eq:ansatz}
 g=h+u^2d\phi_1^2+v^2d\phi_2^2,\qquad h\in[h_1],\qquad\text{and}\qquad u,v>0.
\end{equation}

Besides the torus action, we impose
$\mathcal G_4:=\langle\mathsf r_x,\mathsf r_y,\mathsf e,\mathsf c_1,\mathsf c_2\rangle$, where
\begin{equation}\label{eq:group-actions}
\begin{aligned}
\mathsf r_x(x,y,\phi_1,\phi_2)&:=(-x,y,\phi_1,\phi_2),\qquad
\mathsf r_y(x,y,\phi_1,\phi_2):=(x,-y,\phi_1,\phi_2),\\
\mathsf e(x,y,\phi_1,\phi_2)&:=(y,x,\phi_2,\phi_1),\\
\mathsf c_1(x,y,\phi_1,\phi_2)&:=(x,y,-\phi_1,\phi_2),\qquad\text{and}\qquad
\mathsf c_2(x,y,\phi_1,\phi_2):=(x,y,\phi_1,-\phi_2).
\end{aligned}
\end{equation}
We put $\mathsf c:=\mathsf c_1\mathsf c_2$, which reverses both angles.
In terms of the coefficients, invariance under $\mathsf c_1$, or $\mathsf c_2$, is equivalent to $G_{12}=0$, that is, to the diagonal form \eqref{eq:ansatz}, and the other generators mean that
\[
 \begin{gathered}
\mathsf r_x^*h=\mathsf r_y^*h=\mathsf e^*h=h,\\
 u(-x,y)=u(x,-y)=u(x,y),\qquad\text{and}\qquad v(x,y)=u(y,x).
 \end{gathered}
\]

\subsubsection{The hexagon}
Here the off-diagonal entry $G_{12}$ is allowed in \eqref{eq:intro-matrix-ansatz}. The K\"ahler-Einstein metric $g_1$ has the
isometries below by \cite[Theorem~1.1, Section 3 and Example~4.1]{BS99}. Set
\begin{equation}\label{eq:hexagon-symmetries}
 \mathsf A:=\begin{pmatrix}1&-1\\1&0\end{pmatrix},\qquad
 \mathsf B:=\begin{pmatrix}0&1\\1&0\end{pmatrix},\qquad\text{so that}\qquad
 \mathsf A^6=\mathsf B^2=I\quad\text{and}\quad
 \mathsf B\mathsf A\mathsf B=\mathsf A^{-1}.
\end{equation}

Let $\mathsf r,\mathsf b$ be the order-six rotation and a reflection of $\overline\Omega$, $\mathsf b$ mapping the two rods with rod vector $\pm(1,0)$ to the two rods with rod vector $\pm(0,1)$. We use the same symbols for their lifts, defined by
\begin{equation}\label{eq:hexagon-lifts}
 \mathsf r(p,\boldsymbol\phi):=(\mathsf r p,\mathsf A\boldsymbol\phi),
 \qquad
 \mathsf b(p,\boldsymbol\phi):=(\mathsf b p,\mathsf B\boldsymbol\phi),
 \qquad\text{and}\qquad
 \mathsf c(p,\boldsymbol\phi):=(p,-\boldsymbol\phi).
\end{equation}
We require invariance under $\mathcal G_6:=\langle\mathsf r,\mathsf b,\mathsf c\rangle$, that is, in terms of the coefficients,
\begin{equation}\label{eq:hexagon-equivariance}
 \mathsf r^*h=\mathsf b^*h=h,\qquad
 G\circ\mathsf r=\mathsf A^{-\mathsf T}G\mathsf A^{-1},\qquad\text{and}\qquad
 G\circ\mathsf b=\mathsf B^{-\mathsf T}G\mathsf B^{-1}.
\end{equation}

The six rods have rod vectors
$v_j:=\mathsf A^j(1,0)^{\mathsf T}$, $0\leqslant j\leqslant5$, namely
\[
 (1,0),\ (1,1),\ (0,1),\ (-1,0),\ (-1,-1),\ (0,-1).
\]

\subsubsection{The central torus}
In the square case, set $m:=4$ and $\mathsf r(x,y):=(-y,x)$, lifted as $\mathsf r_x\mathsf e$, and in the hexagon case, set $m:=6$ and let $\mathsf r$ be the order-six rotation of $\overline\Omega$ in \eqref{eq:hexagon-lifts}.
For every metric \eqref{eq:matrix-ansatz}, $G$ extends to $o$ with $G(o)>0$, as the singularity is in $h$ only (Section~\ref{subsec:intro-reduced}), and the restriction $g_\Sigma$ of $g$ to
$\Sigma$ is a flat metric. The symmetries in $\mathcal G$ fix its determinant-one matrix $\Phi(o)$, so that
\begin{equation}\label{eq:central-torus-metric}
 \begin{gathered}
 g_\Sigma=\rho(o)\,(\Phi_o)_{ij}d\phi_i\,d\phi_j,\qquad \rho(o)>0,\\
 \Phi_o:=\Phi_1(o)=
 \begin{cases}
 I,&m=4,\\[1mm]
 \dfrac1{\sqrt3}\begin{pmatrix}2&-1\\-1&2\end{pmatrix},&m=6
 \end{cases},\qquad\text{and}\qquad \Phi(o)=\Phi_o.
 \end{gathered}
\end{equation}
Indeed, in the hexagon case, we have
\[
 \Phi(o)=\mathsf A^{-\mathsf T}\Phi(o)\mathsf A^{-1},\qquad
 \det\Phi(o)=1\qquad\text{and}\qquad\Phi(o)>0
 \quad\Longrightarrow\quad
 \Phi(o)=\Phi_o.
\]
In the square case, invariance under $\mathsf c_1$ and $\mathsf e$ of \eqref{eq:group-actions} implies
$\Phi(o)=I$, or equivalently $u(o)=v(o)$.

\subsubsection{Conformal normalization and fixed coordinates}\label{sec:common-coordinates}

We construct background metrics $\gamma_\beta$ with cone angle $2\pi\beta$ at $o$, to be perturbed as $h=e^{2a}\gamma_\beta$.
Fix a corner $c_0$ and let $c_j:=\mathsf r^jc_0$, $0\leqslant j<m$, be the ordered corners.

The \emph{central chart} is a conformal coordinate $w=x+iy=re^{i\theta}$ near $o$, with $w(o)=0$, normalized so that $\mathsf r$ acts by $w\mapsto e^{2\pi i/m}w$. In the square case, the latitude coordinates of \eqref{eq:square-initial-coordinates} give the central chart and, after translations and changes of sign, the adapted corner charts.
\medskip

\noindent\emph{Background metrics.} Let $\chi(r)$ be a smooth radial cutoff, equal to one near $o$ and supported in the central chart. Fix a smooth invariant $\gamma_1\in[h_1]$, even across the rods and at the corners, with $\gamma_1=\frac{1}{3}|dw|^2$ near $o$: in the square, $\gamma_1:=h_1$. In the hexagon, for a smooth function $\kappa_1$ near $o$, we have $h_1=\frac{1}{3}e^{2\kappa_1}|dw|^2$ near $o$ and $\gamma_1:=e^{-2\chi\kappa_1}h_1$.

We write $a_1$ for the conformal factor of $g_1$, so that $h_1=e^{2a_1}\gamma_1$, with $a_1=0$ in the square and $a_1=\chi\kappa_1$ in the hexagon.
\medskip

\noindent\emph{Fixed coordinates.} For $0<\beta<2$, we choose homeomorphisms $\Psi_\beta$ from a fixed copy of $\overline\Omega$, the \emph{fixed coordinates}, onto $\overline\Omega$, smooth diffeomorphisms away from $o$, depending smoothly on $\beta$ and commuting with the finite symmetries. They satisfy $\Psi_1=\operatorname{id}$ globally on $\overline\Omega$ and $\Psi_\beta=\operatorname{id}$ near $\partial\Omega$. In polar coordinates $(R,\theta)$ on the fixed copy and $(r,\theta)$ on the central chart, and for some $R_0>0$ independent of $\beta\in (0,2) $ with $\chi=1$ on $\Psi_\beta(\{R<R_0\})$, they are given by
\begin{equation}\label{eq:fixed-cone-identification}
\begin{aligned}
 \Psi_\beta(R,\theta)&:=((\beta R)^{1/\beta},\theta),
 &&\text{on }\{R<R_0\},\\
 \gamma_\beta&:=\Psi_\beta^*
       \bigl(e^{2(\beta-1)\chi(r)\log r}\gamma_1\bigr)
       =\tfrac13(dR^2+\beta^2R^2d\theta^2)
       &&\text{on $\{R<R_0\}$}.
\end{aligned}
\end{equation}
Thus $r=(\beta R)^{1/\beta}$ there, as for the conic models of \cite[Section~2.1]{MW17}, and $\gamma_\beta$ has cone angle $2\pi\beta$ at $o$. 

Pulling back all coefficients by $\Psi_\beta$ makes the unknowns functions on a set independent of $\beta$ and we keep the same letters for the pulled-back coefficients. In the fixed coordinates, the conformal class is $\Psi_\beta^*[h_1]$, and we write $h=e^{2a}\gamma_\beta$: the conformal factor $a$ is the unknown.

In a conformal chart, write $\gamma_1=e^{2c_\gamma}|dw|^2$. Before pullback, we have the expansion
\begin{equation}\label{eq:conformal-coefficients}
 h=e^{2a_w}|dw|^2,\qquad\text{and}\qquad
 a_w:=a+c_\gamma+(\beta-1)\chi(r)\log r,
\end{equation}
where $a$ stands for $a\circ\Psi_\beta^{-1}$ and the logarithmic term carries the angle.

\subsection[Ricci curvature and the reduced equations]{Ricci curvature and the reduced equations}
\label{sec:curvature}\label{sec:reduction}

For a map $\Phi$ from a Riemannian surface $(\Omega,\gamma)$ to $(\mathbb H^2,g_{\mathbb H^2})$ of \eqref{eq:intro-matrix-ansatz}, which has curvature $-1$,
the tension field is
\begin{equation}\label{eq:matrix-tension}
 \tau_\gamma(\Phi):=\operatorname{tr}_\gamma\nabla d\Phi=\Delta_\gamma\Phi-|d\Phi|_{\gamma}^2\Phi,
 \qquad\text{where}\qquad |d\Phi|^2_\gamma:=\operatorname{tr}_\gamma\Phi^*g_{\mathbb H^2}.
\end{equation}
The matrix formula is Lemma~\ref{lem:matrix-calculus}(iii). We use the sign convention $\Delta_\gamma:=\operatorname{div}_\gamma\nabla_\gamma$.

In the fixed coordinates \eqref{eq:fixed-cone-identification}, our notation is
\begin{equation}\label{eq:conic-chart}
 \gamma:=\gamma_\beta,\qquad h=e^{2a}\gamma,\qquad
 G=\rho\Phi,\qquad\text{and}\qquad d\mu:=\rho\,dA_h.
\end{equation}
For a $2\times2$ matrix $A$, put $A^\circ:=A-\tfrac12\operatorname{tr}(A)I$, and similarly $S^\circ:=S-\tfrac12(\operatorname{tr}_hS)h$ for a symmetric $2$-tensor $S$ on $(\Omega,h)$.

\begin{proposition}\label{prop:Ricci}
Define
\[
 \Ric_g^{\rm v}:=(\Ric_g(X_i,X_j))_{i,j=1}^2.
\]
For \eqref{eq:matrix-ansatz}, with $\ell:=\log\rho$, the mixed components $\Ric_g(Y,X_i)$, $Y\in T\Omega$, vanish, and we have
\begin{equation}\label{eq:matrix-ricci}
 \begin{aligned}
 \Ric_g|_{T\Omega}
 &=K_hh-\Hess_h\ell-\tfrac12d\ell^2
                         -\tfrac12\Phi^*g_{\mathbb H^2},\\
 G^{-1}\Ric_g^{\rm v}
 &=-\tfrac12\rho^{-1}\diver_h(\rho G^{-1}dG).
 \end{aligned}
\end{equation}
For a matrix of differential forms, $\operatorname{div}_h$ and $d$ mean entry by entry.

The scalar curvature and volume are
\begin{equation}\label{eq:scalar-volume}
 \begin{aligned}
 S_g&=2K_h-2\rho^{-1}\Delta_h\rho
                  +\tfrac12(|d\ell|_h^2-|d\Phi|_h^2),\\
 dV_g&=\rho\,dA_h\,d\phi_1\,d\phi_2,\qquad\text{and}\qquad
  \Vol(g)=4\pi^2\int_\Omega d\mu.
 \end{aligned}
\end{equation}
For a torus-invariant function $f$, we also have
\begin{equation}\label{eq:geometric-laplacian}
 \Delta_gf=\rho^{-1}\diver_h(\rho\nabla_hf)
 =\Delta_hf+\langle d\ell,df\rangle_h.
\end{equation}
\end{proposition}

\begin{proof}
The Ricci curvature is computed in \cite[Section~2.1, (2.6), (2.10), and~(2.17)]{LS25}:
with $J:=G^{-1}dG$ it gives
\[
 \Ric_g|_{T\Omega}=\Ric_h-\Hess_h\log\rho
                    -\tfrac14\operatorname{tr}(J\otimes J),
 \qquad\text{and}\qquad G^{-1}\Ric_g^{\rm v}=-\tfrac1{2\rho}\diver_h(\rho J).
\]
Our normalization satisfies
\[
 \begin{aligned}
 J&=d\ell\,I+\Phi^{-1}d\Phi,\qquad
 \operatorname{tr}(\Phi^{-1}d\Phi)=0,\\
 \tfrac14\operatorname{tr}(J\otimes J)
 &=\tfrac12d\ell^2+\tfrac12\Phi^*g_{\mathbb H^2},\qquad\text{and}\qquad
 \sqrt{\det g}=\rho\sqrt{\det h}.
 \end{aligned}
\]
These identities give \eqref{eq:matrix-ricci}. Then $S_g=\operatorname{tr}_h(\Ric_g|_{T\Omega})+\operatorname{tr}(G^{-1}\Ric_g^{\rm v})$ and $dV_g=\rho\,dA_h\,d\phi_1\,d\phi_2$ give \eqref{eq:scalar-volume} and
\eqref{eq:geometric-laplacian}.
\end{proof}

\begin{proposition}\label{prop:reduced}
For \eqref{eq:matrix-ansatz}, with $h=e^{2a}\gamma$ for any metric $\gamma$ conformal to $h$, such as $\gamma=\gamma_\beta$ of \eqref{eq:conic-chart}, the equations
$\Ric_g^{\rm v}=3G$ and $S_g=12$ are equivalent to
\begin{equation}\label{eq:matrix-reduced}
 \begin{aligned}
 \rho^{-1}\Delta_\gamma\rho+6e^{2a}&=0,\\
 \tau_\gamma(\Phi)+d\Phi(\nabla_\gamma\ell)&=0,\\
 \Delta_\gamma a-K_\gamma+\tfrac14(|d\Phi|_\gamma^2-|d\ell|_\gamma^2)&=0.
 \end{aligned}
\end{equation}
\end{proposition}
\begin{proof}
Using the decomposition in the proof of Proposition~\ref{prop:Ricci}, the vertical equation and \eqref{eq:matrix-ricci} give
\[
 \begin{aligned}
 6=\operatorname{tr}(G^{-1}\Ric_g^{\rm v})
 &=-\rho^{-1}\diver_h(\rho\nabla_h\ell)
 =-\rho^{-1}\Delta_h\rho,\\
 0=(G^{-1}\Ric_g^{\rm v})^\circ
 &=-\frac{1}{2\rho}\diver_h(\rho \Phi^{-1}d\Phi).
 \end{aligned}
\]
The equation $(G^{-1}\Ric_g^{\rm v})^\circ=0$ becomes the equation for $\Phi$ in \eqref{eq:matrix-reduced} after multiplication by $e^{2a}$, since  $\rho^{-1}\diver_h(\rho\Phi^{-1}d\Phi)=\Phi^{-1}\bigl(\tau_h(\Phi)+d\Phi(\nabla_h\ell)\bigr)$, by Lemma~\ref{lem:matrix-calculus}(iii).
For the scalar equation $S_g=12$, use the scalar-curvature identity \eqref{eq:scalar-volume}:
\[
 12=2K_h+12+\tfrac12(|d\ell|_h^2-|d\Phi|_h^2),
 \qquad\text{that is,}\qquad 4K_h=|d\Phi|_h^2-|d\ell|_h^2.
\]
The equation for $a$ in \eqref{eq:matrix-reduced} is obtained using the 2D curvature formula
\[
 e^{2a}K_h=K_\gamma-\Delta_\gamma a
 \qquad \text{ with }h=e^{2a}\gamma.
\]
\end{proof}

\begin{remark}\label{rem:remaining-component}
In local coordinates on $\mathbb H^2$, the principal parts of the system \eqref{eq:matrix-reduced} are the scalar Laplacians $\Delta_\gamma$ of $\ell$, of the two coordinates of $\Phi$ and of $a$, so it is elliptic. By \eqref{eq:matrix-ricci}, for a solution of \eqref{eq:matrix-reduced}, the only component of $\Ric_g-3g$ which may not vanish is the trace-free horizontal part
\begin{equation}\label{eq:remaining-component}
 (\Ric_g-3g)|_{T\Omega}
 =-\bigl(\Hess_h\ell+\tfrac12d\ell^2+\tfrac12\Phi^*g_{\mathbb H^2}\bigr)^\circ.
\end{equation}
Imposing its two scalar equations, with second-order part $(\partial_x^2\ell-\partial_y^2\ell,\,2\partial_x\partial_y\ell)$, would make the system overdetermined. Proposition~\ref{prop:recovery} shows that \eqref{eq:remaining-component} vanishes for the solutions of \eqref{eq:matrix-reduced} satisfying the angle conditions \eqref{eq:angle-equalities}, so they solve \eqref{eq:einstein} with $\lambda=3$.
\end{remark}

\subsection{Smoothness and prescribed angles}\label{sec:smoothness}

We use $w=x+iy$ for local conformal coordinates of $[h_1]$. Near a corner the chart identifies $\overline\Omega$ with a quadrant $\{x\geqslant0,y\geqslant0\}$, and near a rod, it is identified with $\{x\geqslant0\}$. A function on a rod chart $\{x\geqslant 0\}$ is called even or odd if its extension across the rod has that parity in $x$. At a corner $\{x\geqslant0,y\geqslant0\}$, these parities are in $x$ and $y$.

For a rod choose its rod vector $v\in\mathbb Z^2$ as the first vector of a lattice basis, and choose a conformal chart $(x,y)$ with the rod $\{x=0\}$, $x\geqslant0$. Write $h=e^{2\bar a}|dw|^2$. At a corner choose the rod vectors of the two adjacent rods as the torus basis, with the first circle collapsing on $\{x=0\}$ and the second on $\{y=0\}$, and write $h=e^{2\widehat a}|dw|^2$. In an adapted integral basis, we continue to write $G$ for the Gram matrix, and the target coordinates and all barred or hatted coefficients in that chart refer to this basis. 
In rod and corner charts, hats and bars designate the corner and rod quantities and rescalings, respectively, on the entire chart: for instance rescaled Gram matrices are
\begin{equation}\label{eq:normalized-gram}
\begin{aligned}
 \bar G&:=\operatorname{diag}(x^{-1},1)\,G\,
                  \operatorname{diag}(x^{-1},1),
 &\bar\rho&:=\sqrt{\det\bar G}=\frac{\rho}{x},\\
 \widehat G&:=\operatorname{diag}(x^{-1},y^{-1})\,G\,
                  \operatorname{diag}(x^{-1},y^{-1}),
 &\widehat\rho&:=\sqrt{\det\widehat G}=\frac{\rho}{xy}.
\end{aligned}
\end{equation}

The following criteria are standard, compare \cite{Wis03,ER02,CT10,LS25}. Smooth extension of the metric $g$ at a rod is equivalent to
\begin{equation}\label{eq:matrix-smooth-side}
\begin{gathered}
 G=\begin{pmatrix}
       x^2\bar G_{11}&x\bar G_{12}\\
       x\bar G_{12}&\bar G_{22}
    \end{pmatrix},\qquad
 \bar G_{11}|_{x=0}=e^{2\bar a}|_{x=0},\qquad\bar G>0\ \text{up to }\{x=0\},\\
 \bar a,\bar G_{11},\bar G_{22}\ \text{are smooth and even in }x,
 \qquad\text{and}\qquad \bar G_{12}\ \text{is smooth and odd in }x.
\end{gathered}
\end{equation}
At a corner, smoothness is equivalent to
\begin{equation}\label{eq:matrix-smooth-corner}
\begin{gathered}
 G=\begin{pmatrix}
       x^2\widehat G_{11}&xy\widehat G_{12}\\
       xy\widehat G_{12}&y^2\widehat G_{22}
    \end{pmatrix},\qquad\widehat G>0\ \text{up to the axes},\\
 \widehat G_{11}|_{x=0}=e^{2\widehat a}|_{x=0},\qquad
 \widehat G_{22}|_{y=0}=e^{2\widehat a}|_{y=0},\\
 \widehat a,\widehat G_{11},\widehat G_{22}\ \text{are smooth and even in }x,y,
 \qquad\text{and}\qquad \widehat G_{12}\ \text{is smooth and odd in each variable}.
\end{gathered}
\end{equation}

They follow by writing $g$ in the Cartesian coordinates
\begin{equation}\label{eq:cartesian}
 (\mathsf X_a,\mathsf Y_a):=(x_a\cos\phi_a,\,x_a\sin\phi_a),\qquad\text{where}\qquad(x_1,x_2):=(x,y),
\end{equation}
near a corner, using $dx_a^2+x_a^2d\phi_a^2=d\mathsf X_a^2+d\mathsf Y_a^2$, $x_a^2d\phi_a=\mathsf X_a\,d\mathsf Y_a-\mathsf Y_a\,d\mathsf X_a$, and that smooth even functions of $x_a$ are smooth functions of $x_a^2$. At a rod, only the first plane is used.

The angle conditions in \eqref{eq:matrix-smooth-side} and \eqref{eq:matrix-smooth-corner} say that the circumference-to-distance ratio of the orbit circles tends to $2\pi$ at the rods, as in the angle normalization of \cite[Definition~4.19]{LS25}: at a rod $\{x=0\}$, for instance, we have
\begin{equation}\label{eq:rod-angle-ratio}
 \lim_{x\downarrow0}
 \frac{2\pi\sqrt{G_{11}(x,y)}}{\operatorname{dist}_g((x,y,\boldsymbol\phi),\{x=0\})}
 =2\pi\sqrt{\bar G_{11}(0,y)}\,e^{-\bar a(0,y)}.
\end{equation}

\paragraph{Regularized coordinates.}

In computations, it is convenient to use the following global coordinates $(s,\omega)\in\mathbb{R}^2$ on $\mathbb H^2$ to parametrize $\Phi= \rho^{-1}G$, in which:
\begin{equation}\label{eq:target-halfplane}
 \begin{gathered}
 s:=\log\frac{\rho}{G_{11}},\qquad \omega:=\frac{G_{12}}{G_{11}},\\
 \Phi=\begin{pmatrix}e^{-s}&e^{-s}\omega\\
                  e^{-s}\omega&e^s+e^{-s}\omega^2\end{pmatrix},
 \qquad\text{and}\qquad g_{\mathbb H^2}=ds^2+e^{-2s}d\omega^2.
 \end{gathered}
\end{equation}
 These are the half-plane coordinates $z:=\omega+ie^s$ of $\mathbb H^2$. The rescalings \eqref{eq:normalized-gram} act on $z$ by the dilations $z\mapsto xz$ at a rod and $z\mapsto\frac{x}yz$ at a corner, and on the logarithmic height $s$ by the translations $s\mapsto s+\log x$, respectively $s\mapsto s+\log\frac{x}{y}$. So removing these singular factors is additive and the angle conditions \eqref{eq:angle-equalities} are affine, which simplifies the analysis of the singular derivative terms of the regularized equations in Lemma~\ref{lem:regularized-equations}.

 At rods and corners we use  the same coordinates for the rescaled matrices $\bar G$ and $\widehat G$:
\begin{equation}\label{eq:regularized-corner-matrix}
\begin{array}{c|ccc}
 &\text{log area}&\text{longitudinal coordinate}&\text{transverse coordinate}\\ \hline
 \text{rod}&\bar\ell:=\ell-\log x&
       \bar s:=s+\log x&\bar\omega:=x\omega\\
 \text{corner}&\widehat\ell:=\ell-\log x-\log y&
       \widehat s:=s-\log(y/x)&\widehat\omega:=(x/y)\omega.
\end{array}
\end{equation}
With hats at a corner, and identically with bars at a rod, we have
\begin{equation}\label{eq:gram-regular-coordinates}
 \widehat G=e^{\widehat\ell}
 \begin{pmatrix}
  e^{-\widehat s}&e^{-\widehat s}\widehat\omega\\
  e^{-\widehat s}\widehat\omega&
       e^{\widehat s}+e^{-\widehat s}\widehat\omega^2
 \end{pmatrix}.
\end{equation}
These algebraic changes preserve $C^{j,\alpha}$ regularity on sets where $\widehat G$, respectively $\bar G$, is uniformly positive definite, and there $\widehat G$ has even diagonal and odd off-diagonal entries if and only if $\widehat\ell,\widehat s$ are even and $\widehat\omega$ is odd in each normal variable. Note that these are compatible: on a corner chart with $y>0$,
$\bar\ell=\widehat\ell+\log y$, $\bar s=\widehat s+\log y$ and $\bar\omega=y\widehat\omega$.

Since $\bar G_{11}=e^{\bar\ell-\bar s}$, $\widehat G_{11}=e^{\widehat\ell-\widehat s}$, and $\widehat G_{22}=e^{\widehat\ell+\widehat s}$ on $\{y=0\}$, where the odd $\widehat\omega$ vanishes, the angle equalities of \eqref{eq:matrix-smooth-side} and \eqref{eq:matrix-smooth-corner} become
\begin{equation}\label{eq:angle-equalities}
\begin{aligned}
 2\bar a&=\bar\ell-\bar s &&\text{on a rod }\{x=0\},\\
 2\widehat a&=\widehat\ell-\widehat s &&\text{on }\{x=0\} \text{ close to a corner $\{x=y=0\}$},\\
 2\widehat a&=\widehat\ell+\widehat s &&\text{on }\{y=0\} \text{ close to a corner $\{x=y=0\}$}.
\end{aligned}
\end{equation}
At a corner $c=\{x=y=0\}$ they give $\widehat s(c)=0$ and $\widehat\ell(c)=2\widehat a(c)$.
In summary, where $\widehat G>0$, the metric $g$ is smooth at a corner if and only if $\widehat a,\widehat\ell,\widehat s$ are smooth and even, $\widehat\omega$ is smooth and odd in each variable, and \eqref{eq:angle-equalities} holds and similarly with bars at a rod.

\begin{remark}
Li and Sun only impose bounded hyperbolic distance from the rod and corner models, and \textit{recover} smoothness from the Ricci-flat equations after adjusting periods and coordinates \cite[Section~4.2]{LS25}.
Here regularity is \textit{prescribed} in fixed charts also for variations that do not satisfy the Einstein equation. These prescribed conditions dictate some of our choices of function spaces.
\end{remark}

\section{The reduced equations and regularity}\label{sec:spaces}

Fix an Einstein metric $g_*$ with cone angle $2\pi\beta_*$ along $\Sigma$, where $0<\beta_*\leqslant1$, in either symmetry class of Section~\ref{sec:geometry}: $\Ric(g_*)=3g_*$ on $M\setminus\Sigma$, and $g_*$ extends smoothly across the rods and corners. We also assume that $g_*$ satisfies the central condition \eqref{eq:background-central}, stated in Section~\ref{subsec:local-center} once the weighted spaces are defined. 

We first define the spaces $\mathcal X$ and $\mathcal Y$ of Section~\ref{subsec:local-global} on which the left-hand sides of \eqref{eq:matrix-reduced} give a smooth map. We then prove that a solution of \eqref{eq:matrix-reduced} satisfies $\Ric_g=3g$ and has the required regularity at the rods, corners and central torus. Section~\ref{sec:kernel} then proves invertibility and local angle deformation. Except for Section~\ref{subsec:central-regularity}, the arguments below also apply at $\beta_*=1$.

Only the central conical point requires a weight. Fix $0<\alpha<1$ and a compact interval $J:=[\beta_-,\beta_+]\subset(0,2)$ containing $\beta_*$ in its interior. Since $\beta_+<2$, we may fix a weight $\nu$ with
\begin{equation}\label{eq:weight-choice}
 1<\nu<\min\Bigl\{2,\frac{2}{\beta_+}\Bigr\}.
\end{equation}
Constants $C$ depend only on $g_*$, $J$, $\alpha$, $\nu$ and $\mathfrak U$, where $\mathfrak U$ is the fixed finite collection of coordinate charts and bundle trivializations of Section~\ref{subsec:local-global}. Higher-order estimates may also depend on the number of derivatives.

The function spaces below encode three conditions required by the geometry and the reduced equations. At the rods, the angle conditions \eqref{eq:angle-equalities} prescribe the boundary values of the conformal variation $b$, and its normal derivative is recovered later from the remaining Einstein equation (Proposition~\ref{prop:recovery}). At a corner, the third equation must satisfy the compatibility condition \eqref{eq:corner-compatibility} of the $C^{2,\alpha}$ Dirichlet problem (Remark~\ref{rem:corner-compatibility}). At the center, the symmetries remove the constant shape mode and place the nonconstant harmonic degrees above the weight $\nu$ of \eqref{eq:weight-choice}.

\subsection{Coordinates, boundary conditions, and function spaces}\label{subsec:coordinates-spaces}

\subsubsection{Parametrizing metric variations}
Use the fixed coordinates and reference metrics $\gamma_\beta$ of \eqref{eq:fixed-cone-identification}: all coefficients below are functions of these fixed coordinates. Write the background as $g_*=h_*+(G_*)_{ij}\,d\phi_i\,d\phi_j$, with $h_*=e^{2a_*}\gamma_{\beta_*}$ and $G_*=\rho_*\Phi_*=e^{\ell_*}\Phi_*$.

The unknown is $q=(f,\xi,b)$, where $f$ is the (logarithmic) area variation, $b$ the conformal variation, and $\xi$ the shape variation, a vector field along $\Phi_*$, that is, $\xi(p)\in T_{\Phi_*(p)}\mathbb H^2=\bigl(\Phi_*^*T\mathbb H^2\bigr)_p$. We use the hyperbolic exponential, in the matrix form of Lemma~\ref{lem:matrix-calculus}:
\begin{equation}\label{eq:local-chart}
\begin{gathered}
 \rho=\rho_*e^f,\qquad a=a_*+b,\\
 \Phi:=\operatorname{Exp}_{\Phi_*}\xi=\cosh|\xi|_*\,\Phi_*+\frac{\sinh|\xi|_*}{|\xi|_*}\,\xi,\qquad\text{and}\qquad
 |\xi|_*^2:=\tfrac12\operatorname{tr}\bigl((\Phi_*^{-1}\xi)^2\bigr).
\end{gathered}
\end{equation}
Formula \eqref{eq:local-chart} parametrizes all matrices in $\mathbb H^2$ and commutes with the prescribed symmetries. We call $q=(f,\xi,b)$ a \textit{variation} of $g_*$, and we put
\begin{equation}\label{eq:perturbed-metric}
 g_{\beta,q}:=e^{2(a_*+b)}\gamma_\beta+\rho_*e^f\,\Phi_{ij}\,d\phi_i\,d\phi_j,
 \qquad\text{with }\Phi\text{ given by \eqref{eq:local-chart}}.
\end{equation}
At fixed $\beta$, \eqref{eq:scalar-volume} gives $dV_{g_{\beta,q}}=e^{f+2b}\,dV_{g_{\beta,0}}$.

At $(\beta,q)=(\beta_*,0)$, the derivative of $q\mapsto g_{\beta,q}$ is the tensor variation
\[
 \dot g(q)=2b\,h_*+\bigl(fG_*+\rho_*\xi\bigr)_{ij}\,d\phi_i\,d\phi_j.
\]
For symmetric matrices $W$ and $S\in\mathbb H^2$, define the projection $\Pi_S$ onto $T_S\mathbb H^2$ along $\mathbb R S$ by
\begin{equation}\label{eq:matrix-projection}
 \Pi_SW:=W-\tfrac12\operatorname{tr}(S^{-1}W)S.
\end{equation}
For the metric $g_{\beta,q}$ of \eqref{eq:perturbed-metric}, the reduced Einstein equations \eqref{eq:matrix-reduced} read $\mathcal E(\beta,q)=0$, where, with $\tau_\gamma(\Phi)=\operatorname{tr}_\gamma\nabla d\Phi$,
\begin{equation}\label{eq:conic-map}
 \mathcal E(\beta,q):=
 \begin{pmatrix}\mathcal E_1\\ \mathcal E_2\\ \mathcal E_3\end{pmatrix}
 =\begin{pmatrix}
 \rho^{-1}\Delta_{\gamma_\beta}\rho+6e^{2a}\\
 \Pi_{\Phi_*}\mathcal H\\
 \Delta_{\gamma_\beta}a-K_{\gamma_\beta}+\frac{1}{4}\bigl(|d\Phi|_{\gamma_\beta}^2-|d\ell|_{\gamma_\beta}^2\bigr)
 \end{pmatrix},
 \qquad\text{and}\qquad \mathcal H:=\tau_{\gamma_\beta}(\Phi)+d\Phi(\nabla_{\gamma_\beta}\ell).
\end{equation}
The projection \eqref{eq:matrix-projection} identifies $T_\Phi\mathbb H^2$ with the fixed bundle $T_{\Phi_*}\mathbb H^2$, so $\mathcal E_2=0$ is equivalent to $\mathcal H=0$. Since $g_*$ is Einstein, $\mathcal E(\beta_*,0)=0$.

\subsubsection{Local norms at the central point}\label{subsec:local-center}
Recall that $\Phi_o=\Phi_1(o)=\Phi_*(o)$ by the symmetries, see \eqref{eq:central-torus-metric}. 
Near the center we use three target variables:
\begin{center}\small
\begin{tabular}{@{}ll@{}}
$\xi(p)\in T_{\Phi_*(p)}\mathbb H^2$ & the variation in the exponential chart \eqref{eq:local-chart}\\
$\xi^o:=\Pi_{\Phi_o}\xi\in T_{\Phi_o}\mathbb H^2$ & the same variation in the fixed vector space $T_{\Phi_o}\mathbb H^2$\\
$U:=\Pi_{\Phi_o}\Phi\in T_{\Phi_o}\mathbb H^2$ & the nonlinear shape map in the same fixed space
\end{tabular}
\end{center}
By Lemma~\ref{lem:matrix-calculus}, $\Pi_{\Phi_o}$ restricts to an isomorphism $T_{\Phi_*(p)}\mathbb H^2\to T_{\Phi_o}\mathbb H^2$, and $\Phi$ is recovered from $U$ by
\begin{equation}\label{eq:central-matrix-coordinate}
 \Phi=\sqrt{1+|U|_o^2}\,\Phi_o+U,\qquad\text{and}\qquad
 |U|_o^2:=\tfrac12\operatorname{tr}\bigl((\Phi_o^{-1}U)^2\bigr).
\end{equation}
Both constructions commute with the central symmetries, which fix $\Phi_o$.

Let $B_o:=\{0\leqslant R<R_0\}$ be the central chart, with $R_0$ as in \eqref{eq:fixed-cone-identification}. For $A_o:=[1,2]\times\mathbb R/(2\pi\mathbb Z)$, $\delta\in\mathbb R$, $k\geqslant0$, and $F$ a function or a $T_{\Phi_o}\mathbb H^2$-valued map on $B_o\setminus\{o\}$, define
\begin{equation}\label{eq:b-holder-norm}
 \|F\|_{R^\delta C^{k,\alpha}_{\mathrm b}}:=\sup_{0<\lambda<\frac{R_0}{2}}\lambda^{-\delta}\|F(\lambda R',\theta)\|_{C^{k,\alpha}(A_o)}.
\end{equation}
The annular norm uses $dR'^2+R'^2d\theta^2$ here for simplicity since using $dR'^2+\beta^2R'^2d\theta^2$ gives an equivalent norm, uniformly for $\beta\in J$. Ordinary H\"older estimates on the fixed annulus $A_o$, applied to $u(\lambda R',\theta)$ and $v(\lambda R',\theta)$, give, 
\begin{equation}\label{eq:weighted-product}
 \|uv\|_{R^{\delta_1+\delta_2}C^{j,\alpha}_{\mathrm b}}\leqslant C\|u\|_{R^{\delta_1}C^{j,\alpha}_{\mathrm b}}\|v\|_{R^{\delta_2}C^{j,\alpha}_{\mathrm b}},
\end{equation}
and the continuous inclusions $R^{\delta_+}C^{j,\alpha}_{\mathrm b}\subset R^{\delta_-}C^{j,\alpha}_{\mathrm b}$ for $\delta_+\geqslant\delta_-$ \cite[Section~4.1]{MW17}, \cite[Proposition~2.1(iii)]{PR00}. For $\delta\in\mathbb R$, we write $F=O_{\mathrm b}(R^\delta)$ if $F\in R^\delta C^{k,\alpha}_{\mathrm b}$ near $o$ for every integer $k\geqslant0$, with the norms \eqref{eq:b-holder-norm}, or equivalently, if all derivatives $(R\partial_R)^i\partial_\theta^jF$ are $O(R^\delta)$. For vector-valued maps this is meant componentwise, and for tensors it refers to the components in the coframe $(dR,R\,d\theta,d\phi_1,d\phi_2)$. We assume that the background $g_*$ satisfies
\begin{equation}\label{eq:background-central}
 \rho_*-\rho_*(o),\quad a_*-a_*(o)\quad\text{and the entries of }\Phi_*-\Phi_o\quad\text{belong to }R^\nu C^{2,\alpha}_{\mathrm b},
\end{equation}
where $\Phi_*(o)=\Phi_o$ by \eqref{eq:central-torus-metric}. This is satisfied by $g_{1}$ at $\beta_*=1$, by \eqref{eq:seed-bounds} below, so the curve of metrics can start, and these properties continue to hold for the other metrics $g_\beta$ of the curve, by \eqref{eq:central-domain}, and for their limits, by \eqref{eq:coord-central-derivatives}. On $B_o$, we require the following conditions, with constants $f(o),b(o)\in\mathbb R$, and define the associated norm:
\begin{equation}\label{eq:central-domain}
\begin{gathered}
 f-f(o),\ b-b(o),\ \xi^o\in R^\nu C^{2,\alpha}_{\mathrm b},\\
 \|q\|_{\mathcal X(B_o)}:=|f(o)|+|b(o)|+\|f-f(o)\|_{R^\nu C^{2,\alpha}_{\mathrm b}}+\|b-b(o)\|_{R^\nu C^{2,\alpha}_{\mathrm b}}+\|\xi^o\|_{R^\nu C^{2,\alpha}_{\mathrm b}}.
\end{gathered}
\end{equation}
Since $\nu>0$, $f(o)$ and $b(o)$ are the limits of $f$ and $b$ at $o$, the decomposition in \eqref{eq:central-domain} is unique, and $\xi(o)=0$.

The central conditions \eqref{eq:central-domain} are motivated by the angular Fourier modes in $\theta$ allowed by the symmetries in $\mathcal G$, as in the conic theory of \cite{MW17}. For fixed $R$, a function on $B_o\setminus\{o\}$ is expanded in the angle variable $\theta$ as
$F(R,\theta)=\sum_{j\in\mathbb Z}F_j(R)\,e^{ij\theta}$
and $T_{\Phi_o}\mathbb H^2$-valued maps are expanded componentwise. By \cite[Section~4]{MW17}, the homogeneous solutions of $\Delta_{\gamma_\beta}$ on the punctured disk are read, mode by mode, from
\begin{equation}\label{eq:central-indicial-equation}
 \Delta_{\gamma_\beta}(R^\lambda e^{ij\theta})
 =3\Bigl(\lambda^2-\frac{j^2}{\beta^2}\Bigr)R^{\lambda-2}e^{ij\theta}.
\end{equation}
Hence a nonzero Fourier index $j$ has the $\Delta_{\gamma_\beta}$-harmonic modes $R^{\pm|j|/\beta}e^{ij\theta}$ and the index $0$ gives $1$ and $\log R$. The finite symmetries determine which indices occur.

\begin{lemma}\label{lem:central-modes}
Let $F$ be a $\mathcal G$-invariant function on $B_o\setminus\{o\}$, such as $f$, $b$, $\rho$ or $a$, or one of the equivariant maps $\xi^o$ and $U=\Pi_{\Phi_o}\Phi$. Then $F_j\neq0$ only for $j\in2\mathbb Z$, and $\xi^o_0=U_0=0$. In particular, every harmonic mode $R^{|j|/\beta}e^{ij\theta}$ of $F$ with $j\neq0$ has degree at least $2/\beta$.
\end{lemma}

\begin{proof}
The rotation $\mathsf r$ acts by $\theta\mapsto\theta+2\pi/m$ and on the values of $F$ by a linear map $C$ commuting with $\Pi_{\Phi_o}$, so that $F(R,\theta+2\pi/m)=C\,F(R,\theta)$. Taking Fourier coefficients, we obtain
\[
 \bigl(e^{2\pi ij/m}I-C\bigr)F_j=0,
\]
so $F_j\neq0$ only if $e^{2\pi ij/m}$ is an eigenvalue of $C$.
\begin{itemize}
\item For invariant functions $C=1$, so $j\in m\mathbb Z$.
\item In the square, $U$ and $\xi^o$ lie on the diagonal line of $T_{\Phi_o}\mathbb H^2$, where $C=-1$ because $\mathsf r=\mathsf r_x\mathsf e$ exchanges the two circles, so $j\in4\mathbb Z+2$.
\item In the hexagon, $C(V)=\mathsf A^{-\mathsf T}V\mathsf A^{-1}$ by \eqref{eq:hexagon-equivariance}. In the coordinates $V=\bigl(\begin{smallmatrix}v_1&v_2\\v_2&-v_1-v_2\end{smallmatrix}\bigr)$ of $T_{\Phi_o}\mathbb H^2=\{\operatorname{tr}(\Phi_o^{-1}V)=0\}$, it reads $C(v_1,v_2)=(-v_1-v_2,v_1)$ with eigenvalues $e^{\pm2\pi i/3}$, hence $j\in6\mathbb Z\pm2$.
\end{itemize}
Together with \eqref{eq:central-indicial-equation}, this gives the stated degrees.
\end{proof}

Smoothness of $g_1$ and Lemma~\ref{lem:central-modes} give, since a smooth function whose Fourier modes all have even index has no linear term at $o$,
\begin{equation}\label{eq:seed-bounds}
 \rho_1-\rho_1(o)=O_{\mathrm b}(R^2),\qquad a_1-a_1(o)=O_{\mathrm b}(R^2),\qquad \text{and} \qquad
 U_1=\Pi_{\Phi_o}\Phi_1=O_{\mathrm b}(R^2),
\end{equation}so $g_1$ satisfies \eqref{eq:background-central} at $\beta_*=1$, by \eqref{eq:central-matrix-coordinate} and $\nu<2$.

\begin{remark}\label{rem:weight-restrictions}
The three restrictions in \eqref{eq:weight-choice} are used as follows.
\begin{enumerate}
\item $\nu>1$: since $\int_0^{R_0}(R^{\nu-2})^2R\,dR<+\infty$, every $F\in R^{\nu-2}C^{0,\alpha}_{\mathrm b}$ is square integrable for $dA_{\gamma_\beta}$ near $o$, so the forcing terms lie in the energy spaces of Section~\ref{sec:kernel}. Moreover $|dF|_{\gamma_\beta}=O(R^{\nu-1})=o(1)$ for $F\in R^\nu C^{2,\alpha}_{\mathrm b}$.
\item $\nu<2$: $R^2\in R^\nu C^{2,\alpha}_{\mathrm b}$, so the quadratic remainders, such as the $O_{\mathrm b}(R^2)$ terms of \eqref{eq:center-regularity}, lie in $R^\nu C^{2,\alpha}_{\mathrm b}$, and the embedding \eqref{cmp:compact-weighted-embedding} is compact.
\item $\nu<2/\beta_+$: for $\beta\in J$ and every index $j\neq0$ allowed by Lemma~\ref{lem:central-modes}, $\nu<2/\beta\leqslant|j|/\beta$, so the weight stays below the degree of every harmonic mode $R^{|j|/\beta}e^{ij\theta}$ with $j\neq0$.
\end{enumerate}
\end{remark}

\begin{remark}\label{rem:friedrichs}
As in \cite[Section~5.4]{MW17}, we use the Friedrichs extension of $\Delta_{\gamma_\beta}$: the conditions \eqref{eq:central-domain} exclude $\log R$ and the negative powers in the expansion of harmonic modes, which would make the energy $\int_{\{R<R_0\}}|du|^2_{\gamma_\beta}\,dA_{\gamma_\beta}$ infinite.
\end{remark}

\subsubsection{Local norms at rods and corners}\label{subsec:local-rods}

Three descriptions of the metric are used near the axes. The Gram matrices $G$, $\bar G$ and $\widehat G$ express positivity and smooth extension. The variation $q=(f,\xi,b)$ gives the global variations, compatible with the symmetries and with the covariant linearization. The regularized coordinates \eqref{eq:regularized-corner-matrix} give the local equations and norms at rods and corners, and Lemma~\ref{lem:regularized-chart} below relates them to $q$.

In an adapted rod chart $U_\partial$ of \eqref{eq:matrix-smooth-side}, the first torus generator collapses on $\{x=0\}$, with $x\geqslant0$ normal to the rod and $y$ tangential. Use the target coordinates $(s,\omega)$ of \eqref{eq:target-halfplane} in this basis, and the orthonormal frame
\begin{equation}\label{eq:target-rod-frame}
 e_s:=\partial_s,\qquad\text{and}\qquad e_\perp:=e^s\partial_\omega,\qquad\text{in which }g_{\mathbb H^2}=ds^2+e^{-2s}d\omega^2.
\end{equation}
We decompose
\[
 \xi(p)=\xi_s(p)\,e_s\bigl(\Phi_*(p)\bigr)+\xi_\perp(p)\,e_\perp\bigl(\Phi_*(p)\bigr),\qquad\text{and}\qquad |\xi|_*^2=\xi_s^2+\xi_\perp^2.
\]

Write $C^{j,\alpha}_{\mathrm{ev}}$ and $C^{j,\alpha}_{\mathrm{odd}}$ for functions whose even, respectively odd, extensions in the normal variable belong to $C^{j,\alpha}$, with the norms of those extensions. Here, $F(|x|,y)$ is the even extension, and the odd extension is $\operatorname{sgn}(x)F(|x|,y)$, with value zero at $x=0$. The regularity conditions and norm are
\begin{equation}\label{eq:side-domain-norm}
\begin{gathered}
 \xi=\xi_se_s+\xi_\perp e_\perp,\qquad f,\xi_s\in C^{2,\alpha}_{\mathrm{ev}},\qquad \xi_\perp\in C^{2,\alpha}_{\mathrm{odd}},\qquad\text{and}\qquad b\in C^{2,\alpha},\\
 \|q\|_{\mathcal X(U_\partial)}:=\|f\|_{C^{2,\alpha}_{\mathrm{ev}}}+\|\xi_s\|_{C^{2,\alpha}_{\mathrm{ev}}}+\|\xi_\perp\|_{C^{2,\alpha}_{\mathrm{odd}}}+\|b\|_{C^{2,\alpha}}.
\end{gathered}
\end{equation}
The fixed-angle boundary condition is
\begin{equation}\label{eq:conic-boundary}
 2b=f-\xi_s\qquad\text{on }\{x=0\}.
\end{equation}
 The regularized Gram coefficients have the parities of Section~\ref{sec:smoothness}, and Lemma~\ref{lem:regularized-chart} identifies \eqref{eq:conic-boundary} with the exact angle equality $2\bar a=\bar\ell-\bar s$.

At a corner $c$, use the quadrant chart $U_c$ of \eqref{eq:matrix-smooth-corner}, whose two adapted circle generators collapse on $x=0$ and $y=0$. Here $C^{j,\alpha}_{\mathrm{ev}}$ and $C^{j,\alpha}_{\mathrm{odd}}$ mean even, respectively odd, in \emph{each} signed variable. In particular, $F(|x|,|y|)$ is the even extension. Require
\begin{equation}\label{eq:corner-domain}
\begin{gathered}
 f,\xi_s\in C^{2,\alpha}_{\mathrm{ev}},\qquad \xi_\perp\in C^{2,\alpha}_{\mathrm{odd}},\qquad b\in C^{2,\alpha},\\
 2b=f-\xi_s\quad\text{on }\{x=0\},\qquad\text{and}\qquad 2b=f+\xi_s\quad\text{on }\{y=0\},\\
 \|q\|_{\mathcal X(U_c)}:=\|f\|_{C^{2,\alpha}_{\mathrm{ev}}}+\|\xi_s\|_{C^{2,\alpha}_{\mathrm{ev}}}+\|\xi_\perp\|_{C^{2,\alpha}_{\mathrm{odd}}}+\|b\|_{C^{2,\alpha}(U_c)}.
\end{gathered}
\end{equation}

 \begin{lemma}\label{lem:regularized-chart}
At a corner, put  $\widehat a:=a+c_\gamma$ and $\widehat\Phi:=\widehat\rho^{-1}\widehat G$. The exponential chart \eqref{eq:local-chart} reads
\begin{equation}\label{eq:matrix-chart-differential}
 \widehat\ell=\widehat\ell_*+f,\qquad
 \widehat\Phi=\operatorname{Exp}_{\widehat\Phi_*}\bigl(\xi_s\,e_s(\widehat\Phi_*)+\xi_\perp\,e_\perp(\widehat\Phi_*)\bigr),\qquad\text{and}\qquad
 \widehat a=\widehat a_*+b.
\end{equation}
The map $q\mapsto(\widehat\ell-\widehat\ell_*,\widehat s-\widehat s_*,\widehat\omega-\widehat\omega_*,\widehat a-\widehat a_*)$ is a smooth diffeomorphism between neighborhoods of $0$ in $C^{2,\alpha}_{\mathrm{ev}}\times C^{2,\alpha}_{\mathrm{ev}}\times C^{2,\alpha}_{\mathrm{odd}}\times C^{2,\alpha}$, with differential $(f,\xi_s,e^{\widehat s_*}\xi_\perp,b)$ at $q=0$. In particular, for $q$ near $0$, $\|q\|_{\mathcal X(U_c)}$ is equivalent to
\[
 \|\widehat\ell-\widehat\ell_*\|_{C^{2,\alpha}_{\mathrm{ev}}}+\|\widehat s-\widehat s_*\|_{C^{2,\alpha}_{\mathrm{ev}}}+\|\widehat\omega-\widehat\omega_*\|_{C^{2,\alpha}_{\mathrm{odd}}}+\|\widehat a-\widehat a_*\|_{C^{2,\alpha}}.
\]
On the axes, $\widehat s-\widehat s_*=\xi_s$ and $\widehat\omega=\widehat\omega_*$, so the boundary conditions of \eqref{eq:corner-domain} are exactly the angle equalities \eqref{eq:angle-equalities}. The same holds with bars at a rod, with \eqref{eq:conic-boundary}.
\end{lemma}
\begin{proof}
Put $\mathsf D:=(xy)^{1/2}\operatorname{diag}(x^{-1},y^{-1})$, so that $\det\mathsf D=1$ and  $\widehat\Phi=\mathsf D\Phi\mathsf D$. By Lemma~\ref{lem:matrix-calculus}(iii), the congruence $S\mapsto\mathsf DS\mathsf D$ is an isometry of $\mathbb H^2$  commuting with \eqref{eq:matrix-exp}. It is the dilation $z\mapsto\frac{x}yz$  of Section~\ref{sec:smoothness}, which preserves the frame \eqref{eq:target-rod-frame}, and this gives \eqref{eq:matrix-chart-differential} and the differential. The reflections  across the axes act on the regularized coordinates by $\widehat\omega\mapsto-\widehat\omega$, which fixes $e_s$ and reverses $e_\perp$. Since $\widehat s_*$ is even, $\widehat\omega_*$ is odd  and $\operatorname{Exp}$ depends smoothly on its base point, the map preserves the parities, and the inverse function theorem gives the diffeomorphism and the equivalence of norms. On an axis, $\xi_\perp=0$ by oddness, and $\tau\mapsto\operatorname{Exp}_{\widehat\Phi_*}(\tau e_s)$ is the vertical line $\{\widehat\omega=\widehat\omega_*\}$, parametrized by arclength $\widehat s$. So $\widehat s-\widehat s_*=\xi_s$ and $\widehat\omega=\widehat\omega_*$  there, and subtracting the angle equalities  of $g_*$ from those of $g_{\beta,q}$ gives the last claim. At a rod, use $\mathsf D:=x^{1/2}\operatorname{diag}(x^{-1},1)$.
\end{proof}

\emph{The corner value of $\xi_s$.} At a corner $c$, even without the boundary conditions of \eqref{eq:corner-domain}, every $\mathcal G$-equivariant $\xi$ with $\xi_s$ even and $\xi_\perp$ odd satisfies
\begin{equation}\label{eq:corner-symmetry}
 \xi_s(c)=0.
\end{equation}
Indeed, a conjugate of $\mathsf e$, respectively of $\mathsf{br}$, fixes $c$, exchanges $x$ and $y$, and conjugates $\widehat G$ by $\mathsf B$, so invariance and the oddness of $\widehat G_{12}$ give $\widehat s(c)=\widehat s_*(c)=0$, and  Lemma~\ref{lem:regularized-chart} gives $\xi_s(c)=\widehat s(c)-\widehat s_*(c)$.

\paragraph{Leading operators at rods and corners.}
In these hat and bar variables, the  diagonal singular parts of the reduced equations \eqref{eq:matrix-reduced}, which the analysis below has to invert, are
\begin{equation}\label{eq:dictionary-operators}
 D_j:=\partial_x^2+\frac{j}x\partial_x+\partial_y^2+\frac{j}y\partial_y\quad\text{for }j\in\{1,2\},\qquad\text{and}\qquad
 D_\perp:=D_1-\frac{1}{x^2}-\frac{1}{y^2}.
\end{equation}
The equations also contain singular coupling terms, such as the divided derivatives of $\widehat\ell$ in Lemma~\ref{lem:regularized-equations}. We only treat the most difficult case of a corner. In a conformal chart $w=x+iy$, the subscript $0$ denotes the Euclidean metric $|dw|^2$: thus $\Delta_0=\partial_x^2+\partial_y^2$, with the corresponding $\Hess_0$, $\langle\cdot,\cdot\rangle_0$ and $|\cdot|_0$.

\begin{remark}\label{rem:conformal-parity}
The matrix coefficients have the parities indicated in the table below. The conformal variation $b$, however, is an ordinary $C^{2,\alpha}$ function with the prescribed boundary values of \eqref{eq:corner-domain}.
\begin{center}\small
\begin{tabular}{@{}l>{\raggedright\arraybackslash}p{1.55cm}>{\raggedright\arraybackslash}p{3.3cm}>{\raggedright\arraybackslash}p{2.5cm}>{\raggedright\arraybackslash}p{4.4cm}@{}}
\toprule
component&coefficient&regularity condition&model operator&Euclidean model\\ \midrule
$f$&$\widehat\ell$&$C^{2,\alpha}$, even &$D_2$&radial Laplacian on $\mathbb R^3\times\mathbb R^3$\\
$\xi_s$&$\widehat s$&$C^{2,\alpha}$, even&$D_1$&radial Laplacian on $\mathbb R^2\times\mathbb R^2$\\
$\xi_\perp$&$\widehat\omega$&$C^{2,\alpha}$, odd&$D_\perp$ &Laplacian on $\mathbb R^2\times\mathbb R^2$ on radial times first harmonic\\
$b$&$\widehat a$&$C^{2,\alpha}$, no parity imposed, prescribed boundary values&$\Delta_0$&Dirichlet problem on the quadrant\\
\bottomrule
\end{tabular}
\end{center}
\end{remark}

Lemma~\ref{lem:division-lifts} identifies the model operators for $f$, $\xi_s$ and $\xi_\perp$, acting on our parity spaces, with Euclidean Laplacians acting on ordinary H\"older spaces, so standard Schauder estimates apply. For $b$, we use boundary Schauder estimates \cite[Chapter~6]{GT01} along the rods, the Dirichlet estimate at corners of Lemma~\ref{lem:grisvard}, based on \cite[Section~6.4.2]{Gri85}, and the conic estimate of Mazzeo-Weiss, Lemma~\ref{lem:cone-Schauder}, at the center.

\subsubsection{The spaces \texorpdfstring{$\mathcal X$}{X} and \texorpdfstring{$\mathcal Y$}{Y}}\label{subsec:local-global}
Let $\mathcal C$ be the set of corners of $\Omega$. The collection of charts $\mathfrak U$, invariant under the finite symmetries we impose, consists of the central chart $B_o$, one quadrant chart $U_c$ for each $c\in\mathcal C$, finitely many rod charts $U_\partial$, and finitely many interior charts away from $o$ and $\partial\Omega$ and over which $\Phi_*^*T\mathbb H^2$ is trivialized. Define
\begin{equation}\label{eq:deformation-domain}
 \mathcal X=\mathcal X^{2,\alpha}_\nu:=\left\{(f,\xi,b):\
 \begin{array}{l}
 f,b\ \mathcal G\text{-invariant},\ \xi\ \mathcal G\text{-equivariant},\\
 \eqref{eq:central-domain},\ \eqref{eq:side-domain-norm},\ \eqref{eq:conic-boundary},\ \eqref{eq:corner-domain}\text{ hold}
 \end{array}\right\},
\end{equation}
with ordinary $C^{2,\alpha}$ regularity in the remaining interior charts. Its norm is the sum $\|q\|_{\mathcal X}:=\sum_{U\in\mathfrak U}\|q|_U\|_{\mathcal X(U)}$ of the local norms \eqref{eq:central-domain}, \eqref{eq:side-domain-norm} and \eqref{eq:corner-domain} over the fixed finite collection $\mathfrak U$, with the $C^{2,\alpha}$ norm on the interior charts.

The values of the reduced map $\mathcal E$ of \eqref{eq:conic-map} are triples $F=(F_1,F_2,F_3)$ of $\mathcal G$-invariant functions $F_1,F_3$ and a $\mathcal G$-equivariant section $F_2$ of $\Phi_*^*T\mathbb H^2$ on $\Omega\setminus\{o\}$. Decompose $F_2$ in the frame \eqref{eq:target-rod-frame} and impose the following local norms:
\begin{equation}\label{eq:conic-spaces}
\begin{array}{c|c|c}
 \text{chart}&\text{components}&\text{regularity}\\ \hline
 o&(F_1,F_2^o,F_3)&R^{\nu-2}C^{0,\alpha}_{\mathrm b}\\
 \text{rod or corner}&(F_1,F_{2,s},F_3)&C^{0,\alpha}\\
 \text{rod or corner}&F_{2,\perp}&C^{0,\alpha}_{\mathrm{odd}}\\
 \text{interior}&(F_1,F_2,F_3)&C^{0,\alpha}.
\end{array}
\end{equation}
 Here $F_2^o:=\Pi_{\Phi_o}F_2$. The transverse component is odd so that its first-angular-mode lift $F_{2,\perp}(|Z|,y)Z/|Z|$ extends continuously across the axis, which requires $F_{2,\perp}=0$ there, and Lemma~\ref{lem:division-lifts}(b) gives its H\"older bound.  Define $\mathcal Y:=\mathcal Y^{0,\alpha}_{\nu-2}$ to consist of the triples satisfying \eqref{eq:conic-spaces} for which, at every corner $c\in\mathcal C$, one has
\begin{equation}\label{eq:corner-compatibility}
 F_3(c)=0.
\end{equation}
Its norm is the sum of the component norms in \eqref{eq:conic-spaces} over the charts of $\mathfrak U$. Lemma~\ref{lem:reduced-map} below shows that the nonlinear map itself takes values in $\mathcal Y$.

\begin{remark}\label{rem:corner-compatibility}
The corner condition \eqref{eq:corner-compatibility} is necessary for the $C^{2,\alpha}$ Dirichlet problem: if $u\in C^2(\{x,y\geqslant0\})$ vanishes on both axes, then we have
\[
 \partial_x^2u(0)=\partial_y^2u(0)=0,\qquad\text{so}\qquad \Delta_0u(0)=0,
\]
and the following lemma, based on \cite[Theorem~6.4.2.5]{Gri85}, shows that it is also sufficient.
\end{remark}

\begin{lemma}\label{lem:grisvard}
Let $U$ be a relative neighborhood of the corner $c$ in a quadrant chart $\{x,y\geqslant0\}$, let $\gamma=e^{2c_\gamma}|dw|^2$ with $c_\gamma$ smooth and even, and let $0<\alpha<1$. Let $u\in H^1(U)$ vanish on $\{x=0\}$ and on $\{y=0\}$, and solve $\Delta_\gamma u=F$ weakly, with $F\in C^{0,\alpha}(U)$ and $F(c)=0$. Then, for every relative neighborhood $U'$ of $c$ with compact closure in $U$, we have
\[
 \|u\|_{C^{2,\alpha}(U')}\leqslant C\bigl(\|F\|_{C^{0,\alpha}(U)}+\|u\|_{L^2(U)}\bigr).
\]
\end{lemma}
\begin{proof}
The odd reflections of $u$ and $F$ across both axes solve a weak equation near $c$ with bounded right-hand side, so  $u\in C^{1,\alpha}$ near $c$  by \cite[Theorem~8.17 and Corollary~8.36]{GT01}. The local corner expansion of \cite[Theorems~6.4.2.4 and~6.4.2.5]{Gri85}, applied to $\Delta_0u=e^{2c_\gamma}F$, gives $u=v+\lambda\operatorname{Im}(w^2\log w)$ near $c$,  with $w=x+iy$, $v\in C^{2,\alpha}$, and the stated bound for $v$ and $\lambda$. The function  $\operatorname{Im}(w^2\log w)$ is harmonic, vanishes on $\{y=0\}$ and equals $-\frac{\pi}{2}y^2$ on $\{x=0\}$, so the zero boundary values  imply $v_{xx}(c)=0$ and $v_{yy}(c)=\pi\lambda$. Consequently $e^{2c_\gamma(c)}F(c)=\Delta_0v(c)=\pi\lambda$, and  $F(c)=0$ eliminates the logarithmic term. Away from the corner,  the estimate is the boundary Schauder estimate \cite[Lemma~6.18 and Corollary~6.7]{GT01}.
\end{proof}

\subsection{Smoothness of the reduced map}\label{subsec:reduced-map}

\begin{lemma}\label{lem:reduced-map}
There is a neighborhood $\mathcal U_0$ of $0$ in $\mathcal X$ such that the map
\begin{equation}\label{eq:reduced-map}
 \mathcal E:J\times\mathcal U_0\longrightarrow\mathcal Y
\end{equation}
is well defined and smooth. In particular, $\mathcal E_3(\beta,q)(c)=0$ at every corner $c\in\mathcal C$, for every $(\beta,q)\in J\times\mathcal U_0$.
\end{lemma}

\begin{proof}
\emph{Axes.} Use the corner regularizations of \eqref{eq:regularized-corner-matrix}. By  Lemma~\ref{lem:regularized-chart}, they depend smoothly on $q$, with $\widehat\ell,\widehat s$ even, $\widehat\omega$ odd in each variable and $\widehat a\in C^{2,\alpha}$, and they satisfy the angle equalities \eqref{eq:angle-equalities}. 

By Lemma~\ref{lem:regularized-equations}, the regularized form \eqref{cmp:nonlinear-first-residuals} of the reduced equations \eqref{eq:matrix-reduced} defines smooth maps into the spaces \eqref{eq:conic-spaces}, and its transverse component $e^{2c_\gamma}\mathcal H_\perp$ is odd in each variable.

For the condition \eqref{eq:corner-compatibility}, evaluate the last line of \eqref{cmp:nonlinear-first-residuals} at $c$. At $c$, \eqref{eq:divided-even} gives $x^{-1}\partial_x(\widehat\ell+\widehat s)=\partial_x^2(\widehat\ell+\widehat s)$ and $y^{-1}\partial_y(\widehat\ell-\widehat s)=\partial_y^2(\widehat\ell-\widehat s)$, the first derivatives of the even functions $\widehat\ell$ and $\widehat s$ vanish, and $A_\Phi(c)=0$ by \eqref{eq:connection-regularity}. So \eqref{eq:corner-jets} gives
\[
 e^{2c_\gamma(c)}\mathcal E_3(c)=\Delta_0\widehat a(c)-\tfrac12\partial_x^2(\widehat\ell+\widehat s)(c)-\tfrac12\partial_y^2(\widehat\ell-\widehat s)(c)=0.
\]
At a rod, $\bar\omega$ and $e^{2c_\gamma}\mathcal H_\perp$ are odd in $x$.

\emph{Center.} For functions or $T_{\Phi_o}\mathbb H^2$-valued maps $F,\widetilde F$ with $F-F(o),\widetilde F-\widetilde F(o)\in R^\nu C^{2,\alpha}_{\mathrm b}$, we have $dF,d\widetilde F\in R^{\nu-1}C^{1,\alpha}_{\mathrm b}$, so  \eqref{eq:weighted-product} and the inclusion $R^{2\nu-2}C^{0,\alpha}_{\mathrm b}\subset R^{\nu-2}C^{0,\alpha}_{\mathrm b}$, valid as $\nu>0$,  bound $dF\cdot d\widetilde F$ in $R^{\nu-2}C^{0,\alpha}_{\mathrm b}$ by the product of the norms of $F-F(o)$ and $\widetilde F-\widetilde F(o)$ in $R^\nu C^{2,\alpha}_{\mathrm b}$, and $\Delta_{\gamma_\beta}$ maps constants plus $R^\nu C^{2,\alpha}_{\mathrm b}$ to $R^{\nu-2}C^{0,\alpha}_{\mathrm b}$.

\emph{Interior.} Interior charts use ordinary H\"older multiplication. All coefficients depend smoothly on $\beta$ in the fixed coordinates, and \eqref{eq:conic-map} preserves the finite symmetries. This proves the smooth mapping into $\mathcal Y$.
\end{proof}

\subsection{Recovery of the Einstein equation and boundary smoothness}\label{subsec:recovery}
\subsubsection{Recovering Einstein metrics}

For a solution of \eqref{eq:matrix-reduced}, the weighted Einstein defect  is the real part of a holomorphic quadratic differential on the  orbit surface \cite[Section~3.3]{MW17}, and, as in \cite[Section~2.2]{Wis03},  the angle conditions \eqref{eq:angle-equalities} make it vanish.

Let $\zeta:\overline\Omega\to\overline{\mathbb D}$, with $\mathbb D:=\{\zeta\in\mathbb C:|\zeta|<1\}$, be the homeomorphism which is conformal on $\Omega$ for $[h_1]$, normalized by $\zeta(o)=0$ and $\zeta(c_0)=1$. 
Its normalization and the order-$m$ rotation give
\begin{equation}\label{eq:normalized-quotient}
 \zeta(c_j)=e^{2\pi ij/m},\qquad\text{and}\qquad
 \zeta\circ\mathsf r=e^{2\pi i/m}\zeta.
\end{equation}
Since every quotient metric $h$ is conformal to $h_1$ before the identifications \eqref{eq:fixed-cone-identification}, $\zeta$ is a conformal coordinate for every such $h$. Unlike the central chart $w$ of Section~\ref{sec:common-coordinates}, $\zeta$ is global and straightens the right angles at the corners.

\begin{proposition}\label{prop:recovery}
Let $(\beta,q)\in J\times\mathcal U_0$ satisfy $\mathcal E(\beta,q)=0$, and let $g=g_{\beta,q}$. Then $\Ric_g=3g$ on $\pi^{-1}(\Omega\setminus\{o\})$, and $\partial_na=0$ on each rod, where $n$ is the $\gamma_\beta$-unit normal.
\end{proposition}

\begin{proof}
Set $T:=\rho(\Ric_g-3g)|_{T\Omega}$.  Since \eqref{eq:matrix-reduced} is elliptic with smooth coefficients on $\Omega\setminus\{o\}$, interior Schauder estimates \cite[Theorem~6.2]{GT01} make $q$, hence $g$ and $T$, smooth there, and the identities \eqref{eq:residual-tensor} below hold in the classical sense in $\Omega\setminus\{o\}$.

\emph{Step 1: the equations for $T$.}  By \eqref{eq:matrix-ricci} and Proposition~\ref{prop:reduced} we have $\operatorname{tr}_hT=0$. Since the scalar curvature of $g$ is constant, the contracted Bianchi identity gives $\diver_g(\Ric_g-3g)=0$, and we compute
\[
 \diver_hT=\rho\,\bigl(\diver_g(\Ric_g-3g)\bigr)|_{T\Omega}=0.
\]
Together with \eqref{eq:matrix-ricci}, this gives
\begin{equation}\label{eq:residual-tensor}
 T=-\rho\bigl(\Hess_h\ell+\tfrac12d\ell^2+\tfrac12\Phi^*g_{\mathbb H^2}\bigr)^\circ,\qquad \operatorname{tr}_hT=0,\qquad\text{and}\qquad \diver_hT=0.
\end{equation}
It remains to prove $T=0$.

\emph{Step 2: behavior at the sides, the corners and $o$.} In dimension two, the equations $\operatorname{tr}_hT=0$ and $\diver_hT=0$ are conformally invariant by \cite[Section~3.3, Proposition~2]{MW17}. We shall apply $\operatorname{tr}T=0$ and $\diver T=0$ with the Euclidean metric $|d\zeta|^2$ of \eqref{eq:normalized-quotient}, and we write $p_c:=\zeta(c)\in\partial\mathbb D$ for the image of a corner $c$. 

At a rod, in the rod coordinates $(x,y)$ of \eqref{eq:matrix-smooth-side}, with the rod regularizations $\rho=xe^{\bar\ell}$, $s=-\log x+\bar s$, $\omega=\frac{\bar\omega}x$ and $h=e^{2\bar a}|dw|^2$, we have $d\ell=\frac{dx}x+d\bar\ell$, $ds=-\frac{dx}x+d\bar s$ and $\Hess_0\log x+\frac{1}{2}(\frac{dx}x)^2+\frac{1}{2}(\frac{dx}x)^2=0$, where the two squares come from $\frac{1}{2}d\ell^2$ and $\frac{1}{2}ds^2$. By \eqref{eq:connection-regularity}, we also have $A_\Phi=O(x)$.
Using $(\Hess_hF)^\circ=(\Hess_0F-2d\bar a\,dF)^\circ$ and $\Phi^*g_{\mathbb H^2}=ds^2+A_\Phi^2$ in \eqref{eq:residual-tensor}, the terms of order $x^{-2}$ cancel by the identity for $\Hess_0\log x$ above, and, since $\rho=xe^{\bar\ell}$, this leaves
\[
 T=-e^{\bar\ell}\bigl(dx\,d(\bar\ell-\bar s-2\bar a)\bigr)^\circ+O(x).
\]
Since $\partial_x\bar\ell=\partial_x\bar s=0$ on $\{x=0\}$, $T$ extends continuously to $\{x=0\}$, with
\begin{equation}\label{eq:residual-boundary}
 T(\partial_x,\partial_y)|_{x=0}=e^{\bar\ell}\,\partial_y\bigl(\bar a-\tfrac12(\bar\ell-\bar s)\bigr)|_{x=0}=0,\qquad\text{and}\qquad
 T(\partial_x,\partial_x)|_{x=0}=e^{\bar\ell}\,\partial_x\bar a|_{x=0}.
\end{equation}
The first equality uses the angle equality $2\bar a=\bar\ell-\bar s$ on $\{x=0\}$ of \eqref{eq:angle-equalities}, and, once $T=0$ is known, the second gives the missing Neumann condition.

At a corner, use $w=x+iy$, $r_c:=|w|$ and $h=e^{2\widehat a}|dw|^2$, let $\Hess_0$ denote the Hessian for $|dw|^2$, and put $L_+:=\log(xy)$ and $L_-:=\log(y/x)$. All terms of \eqref{eq:residual-tensor} involving only $L_\pm$ cancel as in the flat case by
\[
 \Hess_0L_++\tfrac12dL_+^2+\tfrac12dL_-^2=0.
\]
 The remaining terms of \eqref{eq:residual-tensor} contain bounded second derivatives, or one logarithmic derivative multiplied by a first derivative which is $O(r_c)$ uniformly up to both axes: evenness gives $d\widehat\ell(c)=d\widehat s(c)=0$, \eqref{eq:corner-jets} gives $d\widehat a(c)=0$, and \eqref{eq:connection-regularity} gives $|A_\Phi|_0=O(r_c)$. Since $\rho=O(r_c^2)$  and $\rho(|dL_+|_0+|dL_-|_0)=O(r_c)$, we obtain
\begin{equation}\label{eq:residual-corner}
 |T|_{|dw|^2}=O(r_c^2),\qquad\text{and}\qquad |T|_{|d\zeta|^2}=O(1).
\end{equation}
For the second bound, $\zeta-p_c=\psi_c(w^2)$ near $c$, with $\psi_c$ holomorphic, $\psi_c(0)=0$ and $\psi_c'(0)\neq0$, since $\zeta$ maps the right angle  at $c$ conformally onto a neighborhood of $p_c$ in $\overline{\mathbb D}$. So $|d\zeta|/|dw|$ is bounded above and below by positive  multiples of $r_c$.

Near $o$, $\zeta$ is a conformal coordinate with $d\zeta(o)\neq0$, so \eqref{eq:fixed-cone-identification} gives $R=|\zeta|^\beta e^{\kappa}$ near $o$, with $\kappa$ smooth. From now on, the subscript $0$ refers to the Euclidean metric $|d\zeta|^2$, not to $|dw|^2$. Write $h=e^{2a_\zeta}|d\zeta|^2$ near $o$, with $|da_\zeta|_0=O(|\zeta|^{-1})$ by \eqref{eq:conformal-coefficients}. By \eqref{eq:central-domain} and the algebraic coordinate \eqref{eq:central-matrix-coordinate}, the $\zeta$-derivatives of order $k=1,2$ of $\ell$ and $U$ are $O(|\zeta|^{\beta\nu-k})$. Hence we have
\begin{equation}\label{eq:residual-center}
 |T|_0\leqslant C\rho\,\bigl(|\partial^2\ell|_0+|da_\zeta|_0|d\ell|_0+|d\ell|_0^2+|d\Phi|_0^2\bigr)=O(|\zeta|^{-2+\beta\nu}).
\end{equation}

\emph{Step 3: two harmonic contractions.} In $\zeta=u+iv$, since $\operatorname{tr}_0T=0$ and $\diver_0T=0$, we have
\[
 T=A(du^2-dv^2)+2B\,du\,dv,\qquad \partial_uA+\partial_vB=0\qquad\text{and}\qquad \partial_uB-\partial_vA=0.
\]
With  $X_\zeta:=u\partial_u+v\partial_v$ and $Y_\zeta:=-v\partial_u+u\partial_v$,  the functions $P_\zeta:=T(X_\zeta,X_\zeta)$ and $Q_\zeta:=T(X_\zeta,Y_\zeta)$ satisfy $dP_\zeta=*dQ_\zeta$, with $*du=dv$, so they are conjugate harmonic functions on $\mathbb D\setminus\{0\}$. By \eqref{eq:residual-center}, $|P_\zeta|+|Q_\zeta|\leqslant2|\zeta|^2|T|_0=O(|\zeta|^{\beta\nu})$, so both extend harmonically across $0$, with $P_\zeta(0)=Q_\zeta(0)=0$, and by \eqref{eq:residual-corner} they are bounded near each $p_c$. On an open arc of $\partial\mathbb D\setminus\{p_c:c\in\mathcal C\}$, $X_\zeta$ and $Y_\zeta$ are the Euclidean unit normal and tangent, and $\zeta$ maps the normal and tangent directions of a rod chart to those of the arc. So $Q_\zeta$ is continuous up to the arc by Step~2, and the first identity of \eqref{eq:residual-boundary} gives $Q_\zeta=0$ there.

 So $Q_\zeta$ is  bounded and harmonic on $\mathbb D$  and vanishes on $\partial\mathbb D$ outside finitely many points, hence $Q_\zeta=0$. Then  $P_\zeta$ is constant, equal to $P_\zeta(0)=0$, and, as $X_\zeta$ and $Y_\zeta$  span away from $0$  and $T$ is trace-free, $T=0$ on  $\Omega\setminus\{o\}$.

Finally, since $T=0$, the second identity in \eqref{eq:residual-boundary} gives $\partial_x\bar a=0$ on $\{x=0\}$, and in a rod chart $\gamma_\beta=e^{2c_\gamma}(dx^2+dy^2)$, so $n=\pm e^{-c_\gamma}\partial_x$. By \eqref{eq:conformal-coefficients}, $\bar a=a+c_\gamma$, and $c_\gamma$ is even across $x=0$ because $\gamma_1$ is even across the rods (Section~\ref{sec:common-coordinates}). Hence $\partial_na=0$.
\end{proof}

\subsubsection{Smoothness at the sides and corners}

\begin{corollary}\label{cor:fixed-set-regularity}
Let $(\beta,q)\in J\times\mathcal U_0$ satisfy $\mathcal E(\beta,q)=0$. Then $g_{\beta,q}$ extends smoothly across $\pi^{-1}(\partial\Omega)$.

\end{corollary}

\begin{proof}
We treat a corner. By Proposition~\ref{prop:recovery}, $\partial_x\widehat a=0$ on $\{x=0\}$ and $\partial_y\widehat a=0$ on $\{y=0\}$, so the even extension of $\widehat a$ lies in $C^{2,\alpha}$: its second derivatives are the even extensions of $\partial_x^2\widehat a$ and $\partial_y^2\widehat a$ and the odd extension in each variable of $\partial_x\partial_y\widehat a$, which vanishes on both axes since $\partial_x\widehat a$ and $\partial_y\widehat a$ do. All further estimates are interior Schauder estimates \cite[Theorem~6.2]{GT01} for \eqref{cmp:nonlinear-first-residuals} with zero left-hand sides, after the lifts of Lemma~\ref{lem:division-lifts}(b) or even reflection.

We prove that the even extensions of $\widehat\ell,\widehat s,\widehat a$ and the odd extension of $\widehat\omega$ are smooth by induction on $k$. Assume $\widehat\ell,\widehat s,\widehat a\in C^{k,\alpha}_{\mathrm{ev}}$ and $\widehat\omega\in C^{k,\alpha}_{\mathrm{odd}}$, initially with $k=2$, so that $A_\Phi\in C^{k-1,\alpha}$ by \eqref{eq:connection-regularity}. We improve the regularities of the components in this order:
\begin{enumerate}[label=\textup{(\arabic*)}]
\item $D_2\widehat\ell=-|d\widehat\ell|_0^2-6e^{2\widehat a}\in C^{k-1,\alpha}_{\mathrm{ev}}$, and the lift $\mathcal L_2$ gives $\widehat\ell\in C^{k+1,\alpha}_{\mathrm{ev}}$,
\item by (1) and Lemma~\ref{lem:division-lifts}(a), the divided derivatives $x^{-1}\partial_x\widehat\ell$ and $y^{-1}\partial_y\widehat\ell$ in the $D_1$-equation lie in $C^{k-1,\alpha}$, and the lift $\mathcal L_1$ gives $\widehat s\in C^{k+1,\alpha}_{\mathrm{ev}}$,
\item $D_\perp\widehat\omega=-e^{\widehat s}\langle d\widehat\ell-2d\widehat s,A_\Phi\rangle_0\in C^{k-1,\alpha}_{\mathrm{odd}}$, by \eqref{eq:connection-regularity}, since $\Delta\mathcal L_\perp\widehat\omega=\mathcal L_\perp(D_\perp\widehat\omega)$, the angular lift gives $\widehat\omega\in C^{k+1,\alpha}_{\mathrm{odd}}$, and
\item the right-hand side of the equation for $\widehat a$ now lies in $C^{k-1,\alpha}_{\mathrm{ev}}$, and even reflection gives $\widehat a\in C^{k+1,\alpha}_{\mathrm{ev}}$.
\end{enumerate}
 By the regularized form of the smoothness criterion at the end of Section~\ref{sec:smoothness}, these parities and the angle equalities \eqref{eq:angle-equalities}, which hold by  Lemma~\ref{lem:regularized-chart}, give the smooth extension of $g$. 

\end{proof}

\subsection{Central regularity}\label{subsec:central-regularity}

We use the following weighted Schauder estimate  from \cite[Section~4.3, Proposition~11]{MW17}.

\begin{lemma}\label{lem:cone-Schauder}
Equip the punctured disk $\{0<R<R_0,\ \theta\in\mathbb R/(2\pi\mathbb Z)\}$ with the reference metric $\gamma_\beta=\frac{1}{3}(dR^2+\beta^2R^2d\theta^2)$ of \eqref{eq:fixed-cone-identification}. Let $u$ be bounded, and suppose
\[
 \Delta_{\gamma_\beta}u=\mathfrak f
\]
weakly on the punctured disk. Assume that $u(R,\theta+\pi)=u(R,\theta)$. For $0<\delta<\frac{2}{\beta}$, if $\mathfrak f\in R^{\delta-2}C^{0,\alpha}_{\mathrm b}$, then $u$ has a limit $u_o$ at the vertex and
\begin{equation}\label{eq:cone-schauder}
 |u_o|+\|u-u_o\|_{R^\delta C^{2,\alpha}_{\mathrm b}(R<\frac{1}{2}R_0)}\leqslant C\bigl(\|u\|_\infty+\|\mathfrak f\|_{R^{\delta-2}C^{0,\alpha}_{\mathrm b}(R<R_0)}\bigr),
\end{equation}
where $C$ depends only on $\alpha$, $R_0$, upper and positive lower bounds for $\beta$, a positive lower bound for $\delta$, and a positive lower bound for $\frac{2}{\beta}-\delta$. The statement also applies componentwise to vector-valued functions. Moreover, for every integer $j\geqslant0$ and every $\delta\in\mathbb R$, if $\Delta_{\gamma_\beta}u=\mathfrak f$ on $\{0<R<R_0\}$, then we have
\begin{equation}\label{eq:annular-estimate}
 \|u\|_{R^\delta C^{j+2,\alpha}_{\mathrm b}(R<\frac{1}{2}R_0)}\leqslant C_j\bigl(\|u\|_{R^\delta L^\infty(R<R_0)}+\|\mathfrak f\|_{R^{\delta-2}C^{j,\alpha}_{\mathrm b}(R<R_0)}\bigr),
\end{equation}
where $\|u\|_{R^\delta L^\infty}:=\sup R^{-\delta}|u|$ and $C_j$ depends only on $j$, $\alpha$, $R_0$, bounds for $\beta^{\pm1}$ and an upper bound for $|\delta|$. The same holds for $\Delta_{\gamma_\beta}+C_1^a\partial_a+C_0$ if $RC_1$ and $R^2C_0$ are bounded in $C^{j,\alpha}_{\mathrm b}$, with $C_j$ depending also on the bounds for $RC_1$ and $R^2C_0$.
\end{lemma}

\begin{proof}
This is the scalar expansion in \cite[Section~4.3, Proposition~11]{MW17}, applied to the exact reference cone (their $\beta$ is our $\beta-1$). Schauder estimates on rescaled annuli give \eqref{eq:cone-schauder}. Similarly, \eqref{eq:annular-estimate} is the interior Schauder estimate on two nested fixed annuli for $u_\lambda(R',\theta):=u(\lambda R',\theta)$, which satisfies $\Delta_{\gamma_\beta}u_\lambda=\lambda^2\mathfrak f_\lambda$. See also \cite[Section~4]{MW17}.
\end{proof}

\emph{Setting.} Let $I\subset(0,1)$ be a compact interval, put $\beta_I:=\max I$, fix $\varepsilon$ in the interval $\bigl(0,\min\{2,\frac{2}{\beta_I}-2\}\bigr)$, and let $\beta\in I$ and $M>0$. On the punctured disk $\{0<R<R_0\}$, let $\rho>0$ and $a$ be $\mathcal G$-invariant functions, and let $\Phi$ be a $\mathcal G$-equivariant map to $\mathbb H^2$. Put $h:=e^{2a}\gamma_\beta$, $g:=h+\rho\Phi_{ij}\,d\phi_i\,d\phi_j$ and $U:=\Pi_{\Phi_o}\Phi$, and assume that
\begin{enumerate}[label=\textup{(H\arabic*)}]
\item the reduced equations \eqref{eq:matrix-reduced}, in the form \eqref{eq:central-matrix-system} below, hold in the sense of distributions on $\{0<R<R_0\}$, and
\item we have the bound
\begin{equation}\label{eq:central-quantitative-input}
 \|a\|_{L^\infty}+\|\rho\|_{L^\infty}+\|\rho^{-1}\|_{L^\infty}+\|d\rho\|_{L^\infty(\gamma_\beta)}+\|d\Phi\|_{L^\infty(\gamma_\beta)}\leqslant M.
\end{equation}
\end{enumerate}

\Needspace{12\baselineskip}
\begin{lemma}\label{lem:center-expansion}

Assume (H1) and (H2). Then, on $\{0<R<\frac{1}{2}R_0\}$, we have
\begin{equation}\label{eq:center-refined}
\begin{gathered}
 \rho=\rho(o)-\tfrac12e^{2a(o)}\rho(o)R^2+O_{\mathrm b}(R^{2+\varepsilon}),\\
 a-a(o)=O_{\mathrm b}(R^{2+\varepsilon}),\qquad\text{and}\qquad U=O_{\mathrm b}(R^{2+\varepsilon}),
\end{gathered}
\end{equation}
where, for every $j\geqslant0$, the three remainders are bounded in $R^{2+\varepsilon}C^{j,\alpha}_{\mathrm b}$ by constants depending only on $j$, $I$, $\varepsilon$, $M$, $R_0$ and $\alpha$. Consequently, with the flat product $g^0:=e^{2a(o)}\gamma_\beta+G_{ij}(o)\,d\phi_i\,d\phi_j$ and a constant $C$ depending only on $I$, $\varepsilon$, $M$, $R_0$ and $\alpha$, we have
\begin{equation}\label{eq:center-regularity}
 g-g^0=O_{\mathrm b}(R^2),\qquad C^{-1}g^0\leqslant g\leqslant Cg^0,\qquad |\Rm(g)|_g\leqslant C,\qquad\text{and}\qquad K_h=O_{\mathrm b}(R^2).
\end{equation}
\end{lemma}
\begin{remark}
Every metric $g_{\beta,q}$ with $\mathcal E(\beta,q)=0$ satisfies (H1) and (H2) near $o$, by \eqref{eq:central-domain}, since $\nu>1$.
\end{remark}

\begin{proof}
By \eqref{eq:central-quantitative-input}, $\rho$ and $\Phi$ are Lipschitz for $\gamma_\beta$, so they have limits $\rho(o)\geqslant M^{-1}$ and $\Phi(o)$ at $o$, and $\Phi(o)=\Phi_o$ by the symmetries, as in \eqref{eq:central-torus-metric}. Hence $\rho$, $a$ and $U$ are bounded, and Lemma~\ref{lem:central-modes} and Lemma~\ref{lem:cone-Schauder} apply to $\rho$, $a$ and $U$.

Applying the constant projection $\Pi_{\Phi_o}$ to the second equation of \eqref{eq:matrix-reduced}, with $\tau_\gamma(\Phi)$ given by \eqref{eq:matrix-tension}, gives the central system:
\begin{equation}\label{eq:central-matrix-system}
\begin{aligned}
 \Delta_\gamma\rho&=-6e^{2a}\rho,\\
 \Delta_\gamma U&=-\langle d\ell,dU\rangle_\gamma+|d\Phi|_\gamma^2U,\\
 \Delta_\gamma a&=\tfrac14\bigl(|d\ell|_\gamma^2-|d\Phi|_\gamma^2\bigr),\qquad\text{with }\gamma=\gamma_\beta.
\end{aligned}
\end{equation}

\emph{Step 1: first-order bounds.} Integrating \eqref{eq:central-quantitative-input} along radial segments yields
\[
\begin{gathered}
 |\rho-\rho(o)|\leqslant MR/\sqrt3,\qquad d_{\mathbb H^2}(\Phi,\Phi_o)\leqslant MR/\sqrt3,\\
 |U|_o=\sinh d_{\mathbb H^2}(\Phi,\Phi_o)\leqslant C(M,R_0)R,\qquad\text{and}\qquad |dU|_\gamma\leqslant C(M,R_0)|d\Phi|_\gamma.
\end{gathered}
\]
Thus the right-hand sides of \eqref{eq:central-matrix-system} are bounded in $L^\infty$ by \eqref{eq:central-quantitative-input}. Interior $C^{1,\alpha}$ regularity and estimates for weak solutions with bounded right-hand sides, from the $W^{2,p}$ estimate \cite[Theorem~9.11]{GT01} with $p>\frac{2}{1-\alpha}$ and the Sobolev embedding, on rescaled annuli, turn these $L^\infty$ bounds into the conical bounds
\[
 \|\rho-\rho(o)\|_{RC^{1,\alpha}_{\mathrm b}}+\|U\|_{RC^{1,\alpha}_{\mathrm b}}+\|a\|_{C^{1,\alpha}_{\mathrm b}}\leqslant C(I,M,R_0,\alpha)
\]
on $\{R<\frac{3}{4}R_0\}$. Hence the three right-hand sides of \eqref{eq:central-matrix-system} are bounded in $C^{0,\alpha}_{\mathrm b}$.

\emph{Step 2: weight $2$.} By Lemma~\ref{lem:central-modes}, $\rho$, $a$ and $U$ are $\pi$-periodic in $\theta$,  every allowed nonconstant harmonic degree is at least $\frac{2}{\beta}>2$, and $U$ has no constant mode. Lemma~\ref{lem:cone-Schauder} applies at weight $2$, with the uniform gap $\frac{2}{\beta}-2\geqslant\frac{2}{\beta_I}-2>0$, and gives $\rho-\rho(o),\ a-a(o),\ U\in R^2C^{2,\alpha}_{\mathrm b}$, with bounds depending only on $I$, $M$, $R_0$ and $\alpha$.

\emph{Step 3: removing the quadratic term.} Only the $\rho$-equation has nonzero forcing at the center, so we subtract its radial quadratic solution before improving the weight. Since $\Delta_{\gamma_\beta}R^2=12$, the first equation of \eqref{eq:central-matrix-system} gives
\[
 \Delta_{\gamma_\beta}\bigl(\rho-\rho(o)+\tfrac12e^{2a(o)}\rho(o)R^2\bigr)=-6\bigl(e^{2a}\rho-e^{2a(o)}\rho(o)\bigr).
\]
By Step~2 and the product estimates \eqref{eq:weighted-product}, this right-hand side and those of the equations of \eqref{eq:central-matrix-system} for $a-a(o)$ and $U$ lie in $R^2C^{0,\alpha}_{\mathrm b}$. So \eqref{eq:cone-schauder} at weight $2+\varepsilon$, allowed since $2+\varepsilon<\min\{4,\frac{2}{\beta}\}$, puts the three remainders of \eqref{eq:center-refined} in $R^{2+\varepsilon}C^{2,\alpha}_{\mathrm b}$, and iterating \eqref{eq:annular-estimate} at weight $2+\varepsilon$ on \eqref{eq:central-matrix-system} puts them in $R^{2+\varepsilon}C^{j,\alpha}_{\mathrm b}$ for every $j$, which is \eqref{eq:center-refined}.

Since $\frac{2}{\beta}-(2+\varepsilon)\geqslant\frac{2}{\beta_I}-2-\varepsilon>0$ for $\beta\in I$, the constants of Lemma~\ref{lem:cone-Schauder} depend only on $I$ and $\varepsilon$, which gives the bounds on the remainders. Next, \eqref{eq:center-regularity} follows from \eqref{eq:center-refined} and these bounds. Indeed, $R^{2+\varepsilon}\leqslant R_0^{\varepsilon}R^2$ gives $g-g^0=O_{\mathrm b}(R^2)$, with bounds depending only on $I$, $M$, $R_0$ and $\alpha$. As the vector fields $R\partial_R$ and $\partial_\theta$ have $g^0$-length comparable to $R$, this means
\[
 |\nabla_{g^0}^j(g-g^0)|_{g^0}\leqslant CR^{2-j}\qquad\text{for }j=0,1,2,
\]
and, with $\Rm(g^0)=0$ and the comparison of $g$ with $g^0$ below, these bounds give $|\Rm(g)|_g\leqslant C$. The comparison of $g$ with $g^0$ follows from \eqref{eq:central-quantitative-input}, which gives $|a-a(o)|\leqslant2M$, $M^{-2}\leqslant\rho/\rho(o)\leqslant M^2$ and $e^{-d}\Phi_o\leqslant\Phi\leqslant e^d\Phi_o$ for $d:=d_{\mathbb H^2}(\Phi,\Phi_o)\leqslant MR_0/\sqrt3$. Finally, $4K_h=|d\Phi|_h^2-|d\log\rho|_h^2$ by Proposition~\ref{prop:reduced}, and \eqref{eq:center-refined} gives $|d\Phi|_h^2=O_{\mathrm b}(R^{2+2\varepsilon})$ and $|d\log\rho|_h^2=O_{\mathrm b}(R^2)$.
\end{proof}

\section{Invertibility and local continuation}
\label{sec:kernel}

Throughout this section,
$g_*$ is as in Section~\ref{sec:spaces} with $0<\beta_*<1$, or $g_*=g_1$ with $\beta_*=1$, and we
linearize at $(\beta,q)=(\beta_*,0)$.
To simplify notation, we drop the stars: $\beta,\rho,\Phi,a$
are those of $g_*$. We first show that the linearization $L$ of \eqref{eq:interior-linear} has no kernel, by integration by parts (Section~\ref{subsec:kernel}). We then show that it is Fredholm of index zero, by comparison with an operator $P$, defined in \eqref{cmp:comparison-operator}, whose three equations are solved successively (Section~\ref{subsec:comparison}). We use the following notation throughout the section:
\begin{equation}\label{cmp:background}
 h:=e^{2a}\gamma,\qquad\text{and}\qquad \gamma:=\gamma_\beta.
\end{equation}

\subsection{The linearization}\label{subsec:linearization}

Here $\Phi:\Omega\to\mathbb H^2$ is the matrix-valued map of the fixed Einstein metric $g_*$, and the variations $\xi$ are sections of $\Phi^*T\mathbb H^2$. Let $\nabla$ denote the pullback connection and set
\begin{equation}\label{eq:weighted-Jacobi}
 L_\Phi\xi
 :=\operatorname{tr}_\gamma\nabla^2\xi+\nabla_{\nabla_\gamma\ell}\xi
       +\sum_j R^{\mathbb H^2}(\xi,d\Phi(e_j))d\Phi(e_j),
\end{equation}
where $(e_j)$ is $\gamma$-orthonormal and
$R^{\mathbb H^2}(\xi,V)V=-|V|^2\xi+\langle\xi,V\rangle V$.
This is the derivative in $\Phi$ of the second equation in
\eqref{eq:matrix-reduced}, with $\rho$ and $\gamma$ fixed.

Since $g_*$ satisfies \eqref{eq:matrix-reduced}, the projection
\eqref{eq:matrix-projection} acts as the identity on the first variation of the target equation $\mathcal E_2=0$ of \eqref{eq:conic-map}. The derivative $L:=D_q\mathcal E(\beta,0)$ is
\begin{equation}\label{eq:interior-linear}
 \begin{aligned}
 L_1q&=\Delta_\gamma f+2\langle d\ell,df\rangle_\gamma
                      +12e^{2a}b,\\
 L_2q&=L_\Phi\xi+d\Phi(\nabla_\gamma f),\\
 L_3q&=\Delta_\gamma b-\tfrac{1}{2}\langle d\ell,df\rangle_\gamma
                +\tfrac{1}{2}\langle d\Phi,\nabla\xi\rangle_\gamma,
 \qquad\text{for }q=(f,\xi,b),
 \end{aligned}
\end{equation}
where
$\langle d\Phi,\nabla\xi\rangle_\gamma:=\sum_j\langle d\Phi(e_j),\nabla_{e_j}\xi\rangle$.

The main result of this section is the following. 
\begin{proposition}\label{prop:invertibility}
At the background metric in \eqref{cmp:background}, the operator
\[
 L:\mathcal X\longrightarrow\mathcal Y
\]
is an isomorphism.
\end{proposition}

\subsection{The kernel vanishes}\label{subsec:kernel}

Write $\eta:=f+2b$, and recall $d\mu=\rho\,dA_h$ from \eqref{eq:conic-chart}.
By \eqref{eq:scalar-volume}, $\eta$ is the first variation of $\log dV_g$ at $q=0$.
We first prove that $L(f,\xi,b)=0$ implies $(f,\xi,b)=0$, by an integral identity whose terms are nonnegative, \eqref{cmp:kernel-nonnegative}.  The proof below relies on a singular version of the Lichnerowicz-Obata inequality, Proposition~\ref{prop:spectrum}, to control the scalar variations $f$ and $b$, while the negative curvature of $\mathbb H^2$ controls the shape variation $\xi$.

\begin{proposition}\label{prop:spectrum}
For every metric $g=g_{\beta,q}$ satisfying $\mathcal E(\beta,q)=0$ with $0<\beta\leqslant1$, and every
$f\in H^1(M,g)$ with $\int_M f\,dV_g=0$, we have
\begin{equation}\label{eq:spectral-inequality}
 \int_M|df|_g^2\,dV_g\geqslant4\int_M f^2\,dV_g.
\end{equation}
Here $H^1(M,g)$ is the
completion of $C_c^\infty(M\setminus\Sigma)$ in the norm
$\|f\|_{H^1}^2:=\int_M(f^2+|df|_g^2)\,dV_g$.
Equality in \eqref{eq:spectral-inequality} implies $f=0$.
\end{proposition}

\begin{proof}
For $\beta=1$, this is the classical Lichnerowicz-Obata theorem \cite{Oba62}. Let $\beta<1$. Corollary~\ref{cor:fixed-set-regularity} gives smoothness away from $\Sigma$, and \eqref{eq:center-regularity} shows that we satisfy the assumptions of \cite[Definition~1.2]{Mon18}, so \cite[Theorem~1.1]{Mon18} gives \eqref{eq:spectral-inequality}. If we had equality, the stratified Obata theorem \cite[Theorem~3.1]{Mon18} would give
\[
 g=dt^2+\sin^2t\,\bar g,\qquad\text{and}\qquad \Ric_{\bar g}=2\bar g,
\]
hence $\sec_g=1$ away from $\Sigma$, since the three-dimensional Einstein metric $\bar g$ has constant curvature. But the surfaces $\{\boldsymbol\phi=\phi_0\}$, $\phi_0\in\{0,\pi\}^2$, are fixed by the isometry $\mathsf c$, hence totally geodesic, and they are orthogonal to the torus orbits by \eqref{eq:matrix-ansatz}. So $\pi$ maps them isometrically onto $(\Omega\setminus\{o\},h)$, and $\sec_g=K_h$ on them, while $K_h$ tends to $0$ at $o$ by \eqref{eq:center-regularity}. 
\end{proof}

\begin{proposition}\label{prop:kernel}
At the background in \eqref{cmp:background}, the kernel of $L:\mathcal X\to\mathcal Y$ is $\{0\}$.
\end{proposition}
\begin{proof}
Let $q=(f,\xi,b)\in\ker L$, and put $u:=2b-f=\eta-2f$. With $\langle f\rangle_\mu:=\mu(\Omega)^{-1}\int_\Omega f\,d\mu$, $f_0:=f-\langle f\rangle_\mu$, $\|v\|_\mu^2:=\int_\Omega|v|^2\,d\mu$ and
\[
 \mathcal Q_\Phi(\xi):=\int_\Omega\Bigl[|\nabla\xi|_h^2
       +\sum_j|\xi\wedge d\Phi(e_j)|^2\Bigr]d\mu,
\]
where $|\xi\wedge V|^2:=|\xi|^2|V|^2-\langle\xi,V\rangle^2$ and $(e_j)$ is $h$-orthonormal, we prove the identity
\begin{equation}\label{cmp:kernel-nonnegative}
 0=3\bigl(\|df_0\|_\mu^2-4\|f_0\|_\mu^2\bigr)
       +6\|\eta-2f_0\|_\mu^2+\mathcal Q_\Phi(\xi)
       +12\mu(\Omega)\langle f\rangle_\mu^{\,2}.
\end{equation}
All its terms are nonnegative, the first one by Proposition~\ref{prop:spectrum}. The test $(2b-3f,-\xi,-2f)$ below turns the mixed second-order terms into divergences, and, together with the mean identity \eqref{cmp:kernel-mean}, it leaves the scalar spectral form, a square controlling the relative scale $u$, and the harmonic-map form $\mathcal Q_\Phi(\xi)$.

\emph{Two identities.} By \eqref{eq:conic-boundary} and \eqref{eq:corner-domain}, $u$ satisfies
\begin{equation}\label{cmp:kernel-angle}
 u+\xi_s=0\quad\text{on }\{x=0\},\qquad\text{and}\qquad u-\xi_s=0\quad\text{on }\{y=0\}.
\end{equation}
Write $L^\mu_iq:=e^{-2a}L_iq$ and $\Delta_\mu:=\operatorname{div}_\mu\nabla_h$, where $\operatorname{div}_\mu B:=\rho^{-1}\operatorname{div}_h(\rho B)$, with the pullback connection for target-valued $B$. Then \eqref{eq:interior-linear} reads
\[
\begin{aligned}
 L^\mu_1q&=\Delta_\mu f+\langle d\ell,df\rangle_h+12b,\\ L^\mu_2q&=e^{-2a}L_\Phi\xi+d\Phi(\nabla_hf),\\
 L^\mu_3q&=\Delta_\mu b-\langle d\ell,db\rangle_h-\tfrac12\langle d\ell,df\rangle_h+\tfrac12\langle d\Phi,\nabla\xi\rangle_h.
\end{aligned}
\]
The reduced equations \eqref{eq:matrix-reduced} give $\operatorname{div}_\mu\nabla_h\ell=-6$ and $\operatorname{div}_\mu d\Phi=0$, hence $\operatorname{div}_\mu\langle\xi,d\Phi\rangle^\sharp=\langle\nabla\xi,d\Phi\rangle_h$. Since $\langle\xi,R^{\mathbb H^2}(\xi,V)V\rangle=-|\xi\wedge V|^2$, we also have
\[
 -\langle\xi,e^{-2a}L_\Phi\xi\rangle=|\nabla\xi|_h^2+\sum_j|\xi\wedge d\Phi(e_j)|^2-\operatorname{div}_\mu\langle\xi,\nabla\xi\rangle^\sharp.
\]
Define the vector fields
\begin{equation}\label{cmp:kernel-mean-flux}
 W^\flat:=d(f+b)-\tfrac u2\,d\ell+\tfrac12\langle\xi,d\Phi\rangle,\qquad\text{and}\qquad
 V:=u\nabla_hf-2fW-\langle\xi,\nabla\xi\rangle^\sharp.
\end{equation}
We use two combinations of the equations $L^\mu_iq=0$: the sum
\begin{equation}\label{cmp:kernel-mean}
 L^\mu_1q+L^\mu_3q=\operatorname{div}_\mu W+3\eta,
\end{equation}
and their pairing with the test $(2b-3f,-\xi,-2f)$, which satisfies
\begin{equation}\label{cmp:kernel-test}
 (2b-3f)L^\mu_1q-\langle\xi,L^\mu_2q\rangle-2fL^\mu_3q
 =3|df|_h^2-12f^2+6u^2+|\nabla\xi|_h^2+\sum_j|\xi\wedge d\Phi(e_j)|^2+\operatorname{div}_\mu V.
\end{equation}
Indeed, $\operatorname{div}_\mu\nabla_h\ell=-6$, $\operatorname{div}_\mu\langle\xi,d\Phi\rangle^\sharp=\langle\nabla\xi,d\Phi\rangle_h$, $du=2\,db-df$ and $12b-3u=3\eta$ give \eqref{cmp:kernel-mean}. For \eqref{cmp:kernel-test}, write $2b-3f=u-2f$ and $-2fL^\mu_3q=-2f(L^\mu_1q+L^\mu_3q)+2fL^\mu_1q$, and use \eqref{cmp:kernel-mean}, $\operatorname{div}_\mu(fW)=\langle df,W^\flat\rangle_h+f\operatorname{div}_\mu W$ and the identity for $L_\Phi$ above. The remaining terms are divergences or cancel by the two algebraic identities
\[
 2W^\flat-du+u\,d\ell-\langle\xi,d\Phi\rangle=3\,df\qquad\text{and}\qquad 12ub-6f\eta=6u^2-12f^2.
\]

\emph{Regularity at rods and corners.} At a corner, $d\ell=\frac{dx}x+\frac{dy}y+d\widehat\ell$ and, in the frame \eqref{eq:target-rod-frame}, $\langle\xi,d\Phi\rangle=\xi_s\,ds+\xi_\perp A_\Phi$ with $ds=-\frac{dx}x+\frac{dy}y+d\widehat s$. Hence we have
\[
 W^\flat=d(f+b)-\tfrac u2\,d\widehat\ell+\tfrac12\xi_s\,d\widehat s+\tfrac12\xi_\perp A_\Phi-\frac{u+\xi_s}{2x}\,dx-\frac{u-\xi_s}{2y}\,dy,
\]
and \eqref{cmp:kernel-angle} with Lemma~\ref{lem:division-lifts}(a) makes the last two coefficients $C^{1,\alpha}$, while the other terms are $C^{1,\alpha}$ by the local norms \eqref{eq:side-domain-norm}, \eqref{eq:corner-domain} and the smoothness of the regularized background. At a rod, where $d\ell=\frac{dx}x+d\bar\ell$ and $ds=-\frac{dx}x+d\bar s$, the same holds without the singular term $-\frac{u-\xi_s}{2y}\,dy$. So $W$ and $V$ are $C^{1,\alpha}$ up to the rods and corners, the latter since $\langle\xi,\nabla\xi\rangle=\xi_s\,d\xi_s+\xi_\perp\,d\xi_\perp$ in the orthonormal frame. Consequently $\rho W$ and $\rho V$ vanish on $\partial\Omega$.

\emph{Integration by parts.} Integrate \eqref{cmp:kernel-mean} and \eqref{cmp:kernel-test}, whose left-hand sides vanish, against $d\mu=\rho\,dA_h$ on $\Omega\setminus\{R<\varepsilon\}$. Since $\operatorname{div}_\mu W=\rho^{-1}\operatorname{div}_h(\rho W)$, similarly for $V$, and $\rho W$ and $\rho V$ are $C^{1,\alpha}$ and vanish on $\partial\Omega$, the divergence theorem leaves only the integrals of $\rho\,W\cdot n$ and $\rho\,V\cdot n$ over $\{R=\varepsilon\}$. There, \eqref{eq:central-domain} and \eqref{eq:background-central} give $|W|_h+|V|_h=O(\varepsilon^{\nu-1})$, on a circle of length $O(\varepsilon)$, so these integrals are $O(\varepsilon^\nu)$, and consequently, letting $\varepsilon\downarrow0$ gives
\begin{equation}\label{cmp:kernel-energy}
 \int_\Omega\eta\,d\mu=0,\qquad\text{ and }\qquad
 0=3\|df\|_\mu^2-12\|f\|_\mu^2
      +6\|u\|_\mu^2+\mathcal Q_\Phi(\xi).
\end{equation}
Since $\int_\Omega\eta\,d\mu=0$ and $\int_\Omega f_0\,d\mu=0$, writing $u=(\eta-2f_0)-2\langle f\rangle_\mu$ and $f=f_0+\langle f\rangle_\mu$ turns \eqref{cmp:kernel-energy} into \eqref{cmp:kernel-nonnegative}.

The torus-invariant lift of $f_0$ belongs to $H^1(M,g)$, has zero mean and satisfies
\[
 \int_M(|df_0|_g^2-4f_0^2)\,dV_g
     =4\pi^2\bigl(\|df_0\|_\mu^2-4\|f_0\|_\mu^2\bigr).
\]
So every term of \eqref{cmp:kernel-nonnegative} vanishes. Proposition~\ref{prop:spectrum} gives $f_0=0$, the last term gives $\langle f\rangle_\mu=0$, hence $f=0$, and the second term gives $\eta=0$, hence $b=0$. Finally, $\mathcal Q_\Phi(\xi)=0$ gives $\nabla\xi=0$, hence $\xi=0$, since $\xi(o)=0$ by \eqref{eq:central-domain}.
\end{proof}

\subsection{A comparison operator of index zero}\label{subsec:comparison}

 Divided derivatives such as $\partial_xf/x$ are not compact from $C^{2,\alpha}_{\mathrm{ev}}$ to $C^{0,\alpha}$, so they must remain in the comparison operator $P$, which only changes zeroth-order terms. We define $P$, lower triangular as
\begin{equation}\label{cmp:comparison-operator}
\begin{gathered}
 Pq:=\bigl(A_1f,\ A_2\xi+d\Phi(\nabla_\gamma f),\ L_3q\bigr),\qquad\text{whose diagonal blocks are}\\
 A_1:=\Delta_\gamma+2\langle d\ell,d\,\cdot\rangle_\gamma-1,\qquad A_2:=L_\Phi-1,\qquad\text{and}\qquad A_3:=\Delta_\gamma.
\end{gathered}
\end{equation}
Then $\mathcal Kq:=(L-P)q=\bigl(f+12e^{2a}b,\ \xi,\ 0\bigr)$. Since $L$ and $\mathcal K$ map $\mathcal X$ to $\mathcal Y$, by Lemma~\ref{lem:reduced-map}, so does $P$. At a corner, after multiplication by $e^{2c_\gamma}$, the diagonal blocks of $P$ have the singular models of Remark~\ref{rem:conformal-parity}: $D_2$ for $A_1$, from the weight $\rho^2\sim x^2y^2$ of $\mathfrak a_1$ below, $D_1$ and $D_\perp$ for the two components of $A_2$, from the weight $\rho\sim xy$ and the first angular modes, and $\Delta_0$ for $A_3$. The nonlinear models are those of Lemma~\ref{lem:regularized-equations}, and \eqref{cmp:frame-jacobi} gives those of $A_2$.

\emph{The shifted equations.} The quadratic forms of the operators $-A_1,-A_2,-A_3$ are
\begin{equation}\label{cmp:positive-forms}
 \begin{aligned}
 \mathfrak a_1(f,f)&:=\int_\Omega(|df|_\gamma^2+f^2)\rho^2dA_\gamma,\\
 \mathfrak a_2(\xi,\xi)&:=\int_\Omega\left(
 |\nabla\xi|_\gamma^2+\sum_j|\xi\wedge d\Phi(e_j)|^2+|\xi|^2
 \right)\rho\,dA_\gamma,\\
 \mathfrak a_3(b,b)&:=\int_\Omega|db|_\gamma^2dA_\gamma.
 \end{aligned}
\end{equation}
All functions and sections below have the prescribed finite symmetries. Let $\mathscr H_i$ be the $L^2$ space for the measure in the $i$-th line of \eqref{cmp:positive-forms}, and let $\mathscr V_i$ be the completion, for the energy norm $\mathfrak a_i(u,u)^{1/2}$, of the components $f$, $\xi$, respectively $b$ with $b|_{\partial\Omega}=0$, with the regularity required in \eqref{eq:deformation-domain}. The forms are coercive: $\|u\|_{\mathscr H_i}^2\leqslant C\mathfrak a_i(u,u)$, which is immediate for $i=1,2$ and follows from the Poincar\'e inequality for $i=3$. So, for $F\in\mathscr H_i$, the Lax-Milgram theorem \cite[Theorem~5.8]{GT01} gives a unique weak solution $u\in\mathscr V_i$ of $A_iu=F$, that is, $\mathfrak a_i(u,v)=-(F,v)_{\mathscr H_i}$ for every $v\in\mathscr V_i$, with the energy estimate $\mathfrak a_i(u,u)\leqslant C\|F\|_{\mathscr H_i}^2$. For $F\in\mathcal Y$, every component $F_i$ belongs to $\mathscr H_i$, since the three measures are comparable near $o$ and $\nu>1$ gives the integrability of Remark~\ref{rem:weight-restrictions}(1).

\begin{lemma}\label{lem:shifted-regularity}
Let $F=(F_1,F_2,F_3)\in\mathcal Y$. The weak solutions of $A_1f=F_1$, $A_2\xi=F_2$ and $A_3b=F_3$, with $b|_{\partial\Omega}=0$, have the regularity and symmetries required for the components of $\mathcal X$ in \eqref{eq:deformation-domain},  with the following componentwise estimates:
\[
 \|f\|_{\mathcal X}\leqslant C\|F_1\|_{\mathcal Y},\qquad
 \|\xi\|_{\mathcal X}\leqslant C\|F_2\|_{\mathcal Y},\qquad\text{and}\qquad
 \|b\|_{\mathcal X}\leqslant C\|F_3\|_{\mathcal Y}.
\] Here the $\mathcal X$- and $\mathcal Y$ -norms denote the sums of the local norms involving the indicated component, without the coupled angle conditions.

\end{lemma}
\begin{proof}
Write $u$ for $f$, $\xi$ or $b$, and $F_i$ for the corresponding right-hand side. Each step below gives an estimate of $u$ on a smaller chart by the norm of $F_i$ in $\mathcal Y$, $\mathfrak a_i(u,u)^{1/2}$ and the norm of $u$ in $\mathscr H_i$.

\emph{The center.} By \eqref{eq:center-regularity}, respectively \eqref{eq:seed-bounds} at $g_1$, in the fixed trivialization $\xi^o$, each $A_i$ has the form
\begin{equation}\label{cmp:center-operator}
 A_i=\Delta_\gamma\operatorname{Id}+C_1^a\partial_a+C_0,
 \qquad\text{with}\qquad |C_1|_\gamma=O_{\mathrm b}(R)\quad\text{and}\quad C_0=O_{\mathrm b}(1).
\end{equation}
The finite-energy bound and $F_i\in\mathscr H_i$ give $\Delta_\gamma u\in L^2$. Lemma~\ref{lem:B-cutoffs} extends this weak equation over $o$, and averaging over $\mathcal G$ removes the symmetry restriction on the test functions, since $u$ and $F$ have the symmetries (the averaging argument of Lemma~\ref{lem:division-lifts}(c)). In the coordinates $(R\cos\theta,R\sin\theta)$, $\Delta_\gamma$ is a uniformly elliptic divergence-form operator with bounded coefficients, so the $L^\infty$ estimate of \cite[Theorem~8.17]{GT01}, applied componentwise, with the coupling terms treated as right-hand sides, gives $u\in L^\infty$. Then \eqref{eq:annular-estimate} with $\delta=0$ gives $u\in C^{2,\alpha}_{\mathrm b}$, so $C_1^a\partial_au+C_0u\in C^{0,\alpha}_{\mathrm b}\subset R^{\nu-2}C^{0,\alpha}_{\mathrm b}$, and \eqref{eq:cone-schauder}, applied to $\Delta_\gamma u=F_i-C_1^a\partial_au-C_0u$, gives $u-u(o)\in R^\nu C^{2,\alpha}_{\mathrm b}$, which is \eqref{eq:central-domain}. Indeed,  the weight is below every allowed nonconstant degree,  by Remark~\ref{rem:weight-restrictions}(3), and $\xi^o_0=0$ by Lemma~\ref{lem:central-modes}.

\emph{The scalar equations at rods and corners.} At a corner, we have
\[
 e^{2c_\gamma}A_1f
 =D_2f+2\langle d\widehat\ell,df\rangle_0-e^{2c_\gamma}f.
\]
The lift $\mathcal L_2$ of Section~\ref{subsec:division-lifts} and Lemma~\ref{lem:division-lifts} turn this into an ordinary weak elliptic equation with smooth coefficients on $\mathbb R^3\times\mathbb R^3$. Interior regularity and Schauder estimates \cite[Theorem~8.8, Lemma~9.16 and Theorems~9.19 and~6.2]{GT01} give the required even $C^{2,\alpha}$ regularity. For $A_3$, ordinary boundary Schauder estimates apply along the open rods. At the corners, the zero boundary values and $F_3(c)=0$ are exactly the hypotheses of the Grisvard estimate in Lemma~\ref{lem:grisvard}.

\emph{The equation for $\xi$ at rods and corners.} 
Lemma~\ref{lem:lifted-target} gives the even, respectively odd, $C^{2,\alpha}$ regularity of $\xi_s$ and $\xi_\perp$ at rods and corners, with the associated estimate.

A finite covering of $\Omega$,
the energy estimate, the bound $\|u\|_{\mathscr H_i}^2\leqslant C\mathfrak a_i(u,u)$ and the bound $\|F_i\|_{\mathscr H_i}\leqslant C\|F_i\|_{\mathcal Y}$ checked above give the asserted estimates.
\end{proof}

\begin{proposition}\label{cmp:comparison-isomorphism}
The operator $P:\mathcal X\to\mathcal Y$ is an isomorphism, with
\begin{equation}\label{cmp:comparison-estimate}
 \|P^{-1}F\|_{\mathcal X}\leqslant C\|F\|_{\mathcal Y}.
\end{equation}
The difference $\mathcal K=L-P$ is compact. Consequently $L$ is Fredholm of index zero.
\end{proposition}
\begin{proof}
Let $F\in\mathcal Y$. By Lemma~\ref{lem:shifted-regularity}, the weak solution $f$ of $A_1f=F_1$, and then the weak solution $\xi$ of $A_2\xi=F_2-d\Phi(\nabla_\gamma f)$, have the regularity of the components of $\mathcal X$, with $\|f\|_{\mathcal X}+\|\xi\|_{\mathcal X}\leqslant C\|F\|_{\mathcal Y}$. Indeed, $d\Phi(\nabla_\gamma f)$ is bounded in the second component of $\mathcal Y$ by $C\|f\|_{\mathcal X}$: at $o$ by \eqref{eq:central-domain}, and at the axes by \eqref{eq:divided-even-norm} for $\langle ds,df\rangle_\gamma$ and by the parities of \eqref{eq:connection-form-local} for the transverse component $\langle A_\Phi,df\rangle_\gamma$.

Next, let $b_0$ be an extension of the boundary values $\frac{1}{2}(f\mp\xi_s)$ prescribed by \eqref{eq:conic-boundary} and \eqref{eq:corner-domain}. We choose it as follows: near a corner, $b_0(x,y):=\frac{1}{2}\bigl[f(x,y)+\xi_s(x,0)-\xi_s(0,y)\bigr]$, which has these boundary values since $\xi_s(c)=0$ by \eqref{eq:corner-symmetry}, and near an open rod, $b_0:=\frac{1}{2}(f-\xi_s)$. Patch these extensions near $\partial\Omega$ by a partition of unity, so that $b_0=0$ near $o$, and average over $\mathcal G$. The average keeps the prescribed boundary values, which are $\mathcal G$-invariant because $\mathcal G$ maps rod frames to rod frames. Then $(f,\xi,b_0)\in\mathcal X$, with $\|b_0\|_{\mathcal X}\leqslant C(\|f\|_{\mathcal X}+\|\xi\|_{\mathcal X})$. Since $L$ maps $\mathcal X$ to $\mathcal Y$, $L_3(f,\xi,b_0)$ vanishes at the corners, so $F_3-L_3(f,\xi,b_0)$ belongs to the third component of $\mathcal Y$. Lemma~\ref{lem:shifted-regularity} gives the solution $v$ of $A_3v=F_3-L_3(f,\xi,b_0)$ with $v|_{\partial\Omega}=0$, and $q:=(f,\xi,b_0+v)\in\mathcal X$ satisfies $Pq=F$ and \eqref{cmp:comparison-estimate}.

Conversely, if $Pq=0$ for some $q=(f,\xi,b)\in\mathcal X$, then $f=0$ and $\xi=0$, by uniqueness of the weak solutions, so $b$ vanishes on $\partial\Omega$ and $A_3b=0$, hence $b=0$. So $P$ is an isomorphism.

Away from $o$, compactness of $\mathcal K$ follows from the compact inclusion $C^{2,\alpha}\hookrightarrow C^{0,\alpha}$. Near $o$, the formula for $\mathcal Kq$ and \eqref{eq:central-domain} give a bounded map into $C^{2,\alpha}_{\mathrm b}$, since the background coefficient $e^{2a}$ is bounded in $C^{2,\alpha}_{\mathrm b}$. Since $\nu<2$, the embedding
\begin{equation}\label{cmp:compact-weighted-embedding}
 C^{2,\alpha}_{\mathrm b}\hookrightarrow R^{\nu-2}C^{0,\alpha}_{\mathrm b}
\end{equation}
is compact \cite[Lemma~2.2]{PR00}, so $\mathcal K:\mathcal X\to\mathcal Y$ is compact.
\end{proof}
Together with Proposition~\ref{prop:kernel}, this proves Proposition~\ref{prop:invertibility}.

\subsection{Local continuation}\label{subsec:local-continuation}

\begin{corollary}\label{cor:local}
Let
$g_*$ be as in Section~\ref{sec:spaces}, so that $\mathcal E(\beta_*,0)=0$
with $0<\beta_*<1$, or $g_*=g_1$ with $\beta_*=1$. There are an open
interval $I_*$ containing $\beta_*$ and a smooth map
$q:I_*\to\mathcal X$, with $q(\beta_*)=0$, such that, for $\beta\in I_*$, we have
\[
 \mathcal E(\beta,q(\beta))=0.
\]
For $(\beta,\widetilde q)$ sufficiently close to $(\beta_*,0)$,
$\mathcal E(\beta,\widetilde q)=0$ implies $\widetilde q=q(\beta)$.
For $g_*=g_1$, the metrics $g_{\beta,q(\beta)}$, $\beta\in I_*\cap(0,1]$, form an Einstein branch starting at $g_1$, with the prescribed symmetries and smooth rods and corners, by Proposition~\ref{prop:recovery} and Corollary~\ref{cor:fixed-set-regularity}.
\end{corollary}
\begin{proof}
By Proposition~\ref{prop:invertibility}, the derivative in $q$
is invertible. The implicit function theorem applied to the smooth map
\eqref{eq:reduced-map} gives the curve and its local uniqueness.
\end{proof}

\begin{remark}
The strict inequality $\nu<2$ is what allows one pair of spaces
$\mathcal X,\mathcal Y$ for all $\beta$ near $1$. The target modes
$R^{2/\beta}e^{\pm2i\theta}$ have degree $\frac{2}{\beta}$, which crosses $2$ at
$\beta=1$. With a weight $\nu<2$, they and their $\beta$-derivatives stay
in $R^\nu C^{2,\alpha}_{\mathrm b}$: for instance
$\partial_\beta R^{2/\beta}=-2\beta^{-2}R^{2/\beta}\log R$ equals
$-2R^2\log R$ at $\beta=1$.
\end{remark}

\section{Compactness and continuation}
\label{sec:compactness}

We continue the branch of Corollary~\ref{cor:local} at $g_1$ to
every angle in $(0,1)$ by proving compactness on each
compact interval $I\subset(0,1)$.
A priori, the continuation could fail because the metrics collapse, because their curvature concentrates, or because the coordinates degenerate, and we rule out these three possibilities.

In this section and in Section~\ref{sec:collapse}, $g$ denotes a metric \eqref{eq:matrix-ansatz}, with the symmetries of Section~\ref{sec:geometry} or their lifts to the orbit space $\Omega^{(k)}$ of Section~\ref{sec:proof-main}, in one of the following two classes:
\begin{enumerate}[label=\textup{(C\arabic*)}]
\item $g=g_{\beta,q}$ as in \eqref{eq:perturbed-metric}, with $0<\beta<1$ and $\mathcal E(\beta,q)=0$, or
\item $g=g^{(k)}$, $k\geqslant1$, one of the smooth Einstein metrics of Section~\ref{sec:proof-main}, with $g^{(1)}=g_1$.
\end{enumerate}
Class (C2) is used only in Section~\ref{sec:collapse}, once these metrics have been constructed. We use the normalized orbit area $\psi:=\rho/\rho(o)$.
In both classes, $\psi$ is continuous on the closed disk $\overline\Omega$, positive on $\Omega$ and zero on $\partial\Omega$, by \eqref{eq:matrix-smooth-side} and \eqref{eq:matrix-smooth-corner}, and $\Phi(o)=\Phi_o$ by \eqref{eq:central-torus-metric}. In class (C2), the Fourier modes of $\ell$ and $\Phi$ in the normal angle $\beta\theta$ at $o$, $\beta=\frac{2}{k+1}$, have indices $j/\beta$ with $j$ as in Lemma~\ref{lem:central-modes}, so none has index $\pm1$, and $d\ell(o)=0$, $d\Phi(o)=0$.

All statements about convergence of coefficients use the fixed coordinates defined by $\Psi_\beta$ in \eqref{eq:fixed-cone-identification}.  Write
\[
 g=h+G_{ij}\,d\phi_i d\phi_j,\qquad
 h=e^{2a}\gamma_\beta,\qquad\text{and}\qquad G=\rho\Phi.
\]
 We use the superlevel sets $E_t:=\{\psi\geqslant t\}$, $t\in(0,1]$.
By Corollary~\ref{cor:concavity} below, $E_t$ is, for $t<1$, a closed disk with smooth boundary $\partial E_t=\{\psi=t\}$ and interior $\{\psi>t\}$. We call $E_t$ a \emph{core}, and note that $E_1 = \{o\}$. The normalization separates the central torus area scale $\rho(o)$ from the relative orbit area $\psi$, and the core estimates of Section~\ref{sec:intrinsic-cores} are independent of $\rho(o)$; for instance, $\psi=\cos x\cos y$ for $g_1$ in the square case, and $\psi$ tends to $r^2$ in Theorem~\ref{thm:collapsed-limit}.

\subsection{Geometric estimates and noncollapsing}
\label{sec:noncollapse-first}
 The results in this section use only \eqref{eq:matrix-reduced} and the regularity proved in Section~\ref{sec:spaces}.

\subsubsection{Estimates on the curvature and orbit areas}

Set
$B:=\Phi^*g_{\mathbb H^2}$ and recall the notation $\ell=\log\rho$. Proposition~\ref{prop:recovery} and
\eqref{eq:matrix-ricci} to \eqref{eq:matrix-reduced} give
\begin{equation}\label{eq:Einstein-identities}
 \begin{gathered}
 \Delta_h\rho=-6\rho,\qquad
 \tau_h\Phi+d\Phi(\nabla_h\ell)=0,\\
 \operatorname{Hess}_h\ell=(K_h-3)h-\tfrac12(d\ell\otimes d\ell+B),
 \qquad\text{and}\qquad 4K_h=|d\Phi|_h^2-|d\ell|_h^2.
 \end{gathered}
\end{equation}

\begin{lemma}\label{lem:bochner}
For $g$ in (C1) or (C2), we have $K_h\leqslant1$ on $\Omega$.
\end{lemma}
\begin{proof}
Define a $\Phi^*T\mathbb H^2$-valued symmetric $2$-tensor as follows
\[
 \mathcal D:=\rho^{-1}\operatorname{Sym}\nabla(\rho\,d\Phi),
 \qquad\text{so that}\qquad
 \operatorname{tr}_h\mathcal D
 =\tau_h\Phi+d\Phi(\nabla_h\ell)=0.
\]
 We first prove the identity
 \begin{equation}\label{eq:curvature-identities}
 \Delta_gK_h+2\langle d\ell,dK_h\rangle_h=\tfrac12|\mathcal D|_h^2+3(K_h-1)(K_h+3)+\tfrac34\det\nolimits_hB,
\end{equation}
where $\det_hB=\det(h^{-1}B)\geqslant0$.
The Bochner formula for the drift Laplacian $\Delta_g=\Delta_h+\langle d\ell,d\,\cdot\,\rangle_h$ of \eqref{eq:geometric-laplacian}, applied to $\ell$, and the Eells-Sampson Bochner formula \cite{ES64} for $\Phi$, whose tension field is $\tau_h\Phi=-d\Phi(\nabla_h\ell)$ and whose target has curvature $-1$, give (see \cite[(2.2) to (2.4)]{CZ17} for the drift terms)
\begin{equation}\label{eq:bochner-pair}
\begin{aligned}
 \tfrac12\Delta_g|d\ell|^2&=|\Hess_h\ell|^2+K_h|d\ell|^2-\Hess_h\ell(\nabla_h\ell,\nabla_h\ell),\\
 \tfrac12\Delta_g|d\Phi|^2&=|\nabla d\Phi|^2+K_h|d\Phi|^2-\langle\Hess_h\ell,B\rangle+|d\Phi|^4-|B|^2,
\end{aligned}
\end{equation}
where we used $\Delta_h\ell=-6-|d\ell|^2$ and all norms are taken with respect to $h$. In \eqref{eq:bochner-pair}, we substitute  \eqref{eq:Einstein-identities} and the identity $|d\Phi|^4-|B|^2=2\det_hB$, valid for a symmetric form $B$ on a surface.
The definition of $\mathcal D$ gives by expanding
\[
\begin{aligned}
 |\mathcal D|^2
 &=|\nabla d\Phi|^2+\langle d\ell,d|d\Phi|^2\rangle\\
 &\quad+\tfrac12\bigl(|d\ell|^2|d\Phi|^2+
                   B(\nabla_h\ell,\nabla_h\ell)\bigr).
\end{aligned}
\]
Subtracting the two Bochner identities \eqref{eq:bochner-pair} and using
$d|d\Phi|^2=d|d\ell|^2+4\,dK_h$ therefore gives
\eqref{eq:curvature-identities}.

\emph{The bound $K_h\leqslant1$.}  Put $F:=\rho(1-K_h)$. Multiplying $1-K_h$ by $\rho$ gives zero boundary values at the rods and corners and absorbs the extra drift term into the zeroth-order coefficient. Since $\Delta_g\rho=(|d\ell|_h^2-6)\rho$, by \eqref{eq:Einstein-identities} and \eqref{eq:geometric-laplacian}, the identity \eqref{eq:curvature-identities} and $4K_h=|d\Phi|_h^2-|d\ell|_h^2$ give
\[
 \Bigl[\Delta_g-\Bigl(3+\tfrac14|d\ell|_h^2+\tfrac34|d\Phi|_h^2\Bigr)\Bigr]F=-\tfrac\rho2|\mathcal D|_h^2-\tfrac{3\rho}4\det\nolimits_hB\leqslant0.
\]
The  smooth rod and corner models make $K_h$  bounded up to $\partial\Omega$, so $F$ extends continuously by $0$ to $\partial\Omega$. Moreover $F(o)=\rho(o)>0$, since $K_h(o)=0$, by \eqref{eq:center-regularity} in class (C1), and by \eqref{eq:Einstein-identities} and $d\ell(o)=d\Phi(o)=0$ in class (C2). A negative minimum of $F$ would therefore lie in $\Omega\setminus\{o\}$, where the left-hand side would be positive. Hence $F\geqslant0$, that is, $K_h\leqslant1$.

\end{proof}

\begin{corollary}\label{cor:concavity}
Let $g$ be in (C1) or (C2). Then, we have
\begin{equation}\label{eq:concavity}
 \frac{\Hess_h\sqrt\psi}{\sqrt\psi}=\tfrac12(K_h-3)h-\tfrac14B\leqslant-h.
\end{equation}
Consequently $o$ is the only critical point of $\psi$, $\psi(o)=\max\psi=1$, every superlevel set $E_t$ with $0<t<1$ is a closed disk with smooth boundary $\partial E_t=\{\psi=t\}$, and, for $q\in\Omega\setminus\{o\}$, we have
\begin{equation}\label{eq:base-radius}
\begin{gathered}
 4\psi(1-\psi)\leqslant|d\psi|_h^2\leqslant4(1-\psi^2),\qquad
 d_h(o,q)\leqslant\arccos\sqrt{\psi(q)},\qquad d_h(q,\partial\Omega)\leqslant\arcsin\sqrt{\psi(q)},\\
 \text{and}\qquad d_h(q,E_\delta)\leqslant\arcsin\sqrt\delta-\arcsin\sqrt{\psi(q)}\qquad\text{if }0<\psi(q)<\delta<1.
\end{gathered}
\end{equation}
In particular, for $q,q'\in\Omega$, we have
\begin{equation}\label{eq:coord-psi-distance}
 |\psi(q)-\psi(q')|\leqslant2d_h(q,q'),\qquad\text{and}\qquad
 \psi(q)\geqslant\min\bigl\{\tfrac12,\tfrac12d_h(q,\partial\Omega)^2\bigr\}.
\end{equation}
\end{corollary}
\begin{proof}
The identity in \eqref{eq:concavity} is the Hessian formula in \eqref{eq:Einstein-identities}, since $\frac{\Hess_h\sqrt\psi}{\sqrt\psi}=\frac{1}{2}\Hess_h\ell+\frac{1}{4}d\ell\otimes d\ell$, and the inequality follows from $K_h\leqslant1$ and $B\geqslant0$.
\medskip

\noindent\emph{Critical points.} At a critical point $p\neq o$, \eqref{eq:concavity} gives $\Hess_h\psi=2\sqrt\psi\,\Hess_h\sqrt\psi\leqslant-2\psi h$, so $p$ is a nondegenerate local maximum. When $o$ is conic, \eqref{eq:center-refined} gives $\psi=1-\frac{1}{2}e^{2a(o)}R^2+O_{\mathrm b}(R^{2+\varepsilon})$, so that $\partial_R\psi<0$ for small $R>0$. When $o$ is smooth, $K_h(o)=0$ and $\Hess_h\psi(o)=-3h(o)$ by \eqref{eq:Einstein-identities}. In both cases, $o$ is a strict local maximum, and the superlevel sets of $\psi$ near $o$ are disks.
\medskip

\noindent
\emph{Superlevel sets are disks.} Since $\psi>0$ on $\Omega$ and $\psi$ vanishes continuously on $\partial\Omega$, every superlevel set $\{\psi\geqslant t\}$, $t>0$, is compact. As all critical points are local maxima, decreasing $t$ only adds disjoint disks, one at each critical point, and never merges two components \cite[Theorems~3.1 and~3.2]{Mil63}. If $\psi$ had a critical point $p\neq o$, then $o$ and $p$ would lie in distinct components of $\{\psi\geqslant t\}$ for $t$ below $\min\{\psi(o),\psi(p)\}$, hence for every smaller $t>0$. This is impossible for $t$ below the positive minimum of $\psi$ on a path joining $o$ to $p$ in $\Omega$. Hence $o$ is the only critical point of $\psi$, $\max\psi=\psi(o)=1$, and every $E_t$, $0<t<1$, is a closed disk with smooth boundary $\{\psi=t\}$.
\medskip

\noindent
\emph{The level set flow.} On the open set where $d\psi\neq0$, define
\begin{equation}\label{eq:level-flow}
 \mathsf w:=|d\psi|_h^2,\qquad V_\psi:=\frac{\nabla_h\psi}{\mathsf w},\qquad\text{and}\qquad n:=\frac{\nabla_h\psi}{\sqrt{\mathsf w}}.
\end{equation}
Along a trajectory $\gamma$ of $V_\psi$, parametrized so that $\psi(\gamma(t))=t$, write $\mathsf w(t)$ for $\mathsf w(\gamma(t))$. As $o$ is the only critical point and the superlevel sets are compact, $\gamma$ is defined for $0<t<1$ and converges to $o$, the only point where $\psi=1$, as $t\to1$. For a unit tangent $T$ to the level sets, $\Hess_h\psi=2\sqrt\psi\,\Hess_h\sqrt\psi+\frac{1}{2\psi}d\psi\otimes d\psi$, \eqref{eq:concavity} and $\Delta_h\psi=-6\psi$ give
\begin{equation}\label{eq:level-flow-derivative}
 \Hess_h\psi(T,T)\leqslant-2t,\qquad\text{and}\qquad -8t\leqslant\frac{d\mathsf w}{dt}=2\Hess_h\psi(n,n)\leqslant\frac{\mathsf w}t-4t.
\end{equation}
\medskip

\noindent
\emph{Gradient bounds and distances.} By the expansions at $o$ in the paragraph on critical points, $\mathsf w(t)\to0$ as $t\to1$, and integrating \eqref{eq:level-flow-derivative} from $t$ to $1$ gives $4t(1-t)\leqslant\mathsf w\leqslant4(1-t^2)$. Since $|\gamma'|_h=\mathsf w^{-1/2}$, the length of $\gamma $ between the level sets $\{\psi=t_1\}$ and $\{\psi=t_2\}$, $t_1<t_2$, is at most
\[
 \int_{t_1}^{t_2}\frac{d\tau}{2\sqrt{\tau(1-\tau)}}=\arcsin\sqrt{t_2}-\arcsin\sqrt{t_1}.
\]
The trajectories therefore have finite length as $t\to0$ and as $t\to1$, and, by continuity of $\psi$, their endpoints lie on $\partial\Omega$ and at $o$. This proves \eqref{eq:base-radius} and the bounds extend to $q\in\partial\Omega$ by continuity. Finally, $|d\psi|_h\leqslant2$ gives the first inequality of \eqref{eq:coord-psi-distance} and the second follows from $d_h(q,\partial\Omega)\leqslant\arcsin\sqrt{\psi(q)}\leqslant\sqrt{2\psi(q)}$ when $\psi(q)\leqslant\frac{1}{2}$ and is immediate if $\psi(q)\geqslant\frac{1}{2}$.

\end{proof}

 \begin{lemma}\label{lem:core-retraction}

Let $g$ be in (C1) or (C2). 
For $\delta\in(0,1)$, the  flow of $V_\psi$ in \eqref{eq:level-flow} defines a map $P_\delta$ which does not increase lengths and extends to a retraction of $\overline\Omega$ onto $E_\delta$. Consequently, for $p,q\in E_\delta$, the length distance $d_{E_\delta}$ of $E_\delta$ satisfies
\begin{equation}\label{eq:core-retraction}
 d_{E_\delta}(p,q)=d_{\overline\Omega}(p,q).
\end{equation}
\end{lemma}
\begin{proof}

Fix $\tau\in(0,1)$, parametrize the  circle $\{\psi=\tau\}$  by $\Theta\in\mathbb R/2\pi\mathbb Z$, and transport $\Theta$ along the  flow of $V_\psi$ of  \eqref{eq:level-flow}. This identifies $\Omega\setminus\{o\}$ with $(0,1)\times\mathbb R/2\pi\mathbb Z$, with $t:=\psi$ and $\partial_t=V_\psi$. The flow maps level sets to level sets, so  $h(V_\psi,\partial_\Theta)=0$  and $h=\mathsf w^{-1}dt^2+|\partial_\Theta|_h^2\,d\Theta^2$, since $|V_\psi|_h^2=\mathsf w^{-1}$. As $[\partial_t,\partial_\Theta]=0$  and $\Hess_h\psi(\partial_\Theta,\partial_\Theta)=2\sqrt\psi\,\Hess_h\sqrt\psi(\partial_\Theta,\partial_\Theta)<0$ by \eqref{eq:concavity}, we get
\begin{equation}\label{eq:core-coordinates}
 \tfrac12\,\partial_t|\partial_\Theta|_h^2=h(\nabla_{\partial_\Theta}V_\psi,\partial_\Theta)=\frac{\Hess_h\psi(\partial_\Theta,\partial_\Theta)}{\mathsf w}<0.
\end{equation}

Put $P_\delta(t,\Theta):=(\max\{t,\delta\},\Theta)$. For a path $c=(t(s),\Theta(s))$, the image of its part in $\{t<\delta\}$ is $(\delta,\Theta(s))$, of speed $|\dot\Theta\,\partial_\Theta|_h(\delta,\Theta)\leqslant|\dot\Theta\,\partial_\Theta|_h(t,\Theta)\leqslant|\dot c|_h$, by \eqref{eq:core-coordinates}. So $P_\delta$ is $1$-Lipschitz for the length distance of $\Omega$, and it extends to the completion $\overline\Omega$. Paths meeting the rods or corners are limits of interior paths  whose lengths converge to that of the original path, by the smooth rod and corner models \eqref{eq:matrix-smooth-side} and \eqref{eq:matrix-smooth-corner}, so the distance of $\overline\Omega$ is the infimum of the lengths of paths in $\Omega$. Projecting these paths by $P_\delta$ gives $d_{E_\delta}\leqslant d_{\overline\Omega}$ on $E_\delta$, and the reverse inequality is immediate.
\end{proof}

\begin{lemma}\label{lem:scalar}
Let $g$ be in (C1) or (C2). Recall that $0<\psi\leqslant1$ and $\Phi(o)=\Phi_o$. On $\Omega\setminus\{o\}$:

\begin{enumerate}[label=\textup{(\alph*)}]
\item $1-\psi^{-2}\leqslant K_h\leqslant1$ and $|d\Phi|_h\leqslant2\psi^{-1}$,
\item $d_{\mathbb H^2}(\Phi,\Phi_o)\leqslant2\operatorname{artanh}\sqrt{1-\psi}$, and
\item $\rho(o)\bigl(1-\sqrt{1-\psi}\bigr)^2\Phi_o\leqslant G\leqslant\rho(o)\bigl(1+\sqrt{1-\psi}\bigr)^2\Phi_o\leqslant4\rho(o)\Phi_o$. In particular, every torus orbit has diameter at most $2\operatorname{diam}(\Sigma)$.
\end{enumerate}
\end{lemma}
\begin{proof}
By \eqref{eq:Einstein-identities}, Lemma~\ref{lem:bochner} and \eqref{eq:base-radius}, we have $4K_h\geqslant-|d\ell|_h^2=-\mathsf w\psi^{-2}\geqslant-4(\psi^{-2}-1)$ and $|d\Phi|_h^2=4K_h+\mathsf w\psi^{-2}\leqslant4\psi^{-2}$, which is (a).

Let $\gamma$ be the trajectory through $q\in\Omega\setminus\{o\}$ of the flow of $V_\psi$ in \eqref{eq:level-flow}, parametrized by $t=\psi$, so that $\gamma(t)\to o$ as $t\to1$. Along $\gamma$, $\bigl|\frac{d}{dt}\Phi(\gamma(t))\bigr|^2\leqslant\frac{|d\Phi|_h^2}{\mathsf w}=\psi^{-2}+\frac{4K_h}{\mathsf w}\leqslant\frac{1}{\psi^2(1-\psi)}$, and since $\Phi\to\Phi_o$ at $o$, we get
\[
 d_{\mathbb H^2}(\Phi,\Phi_o)\leqslant\int_\psi^1\frac{dt}{t\sqrt{1-t}}=2\operatorname{artanh}\sqrt{1-\psi},
\]
which is (b).

If $\Phi'\in\mathbb H^2$ lies at distance $D'$ from $\Phi_o$, then $e^{-D'}\Phi_o\leqslant\Phi'\leqslant e^{D'}\Phi_o$, since the eigenvalues of $\Phi_o^{-1/2}\Phi'\,\Phi_o^{-1/2}$ are $e^{\pm D'}$ for the normalization \eqref{eq:intro-matrix-ansatz}. Since $e^{2\operatorname{artanh}\sqrt{1-\psi}}=\frac{1+\sqrt{1-\psi}}{1-\sqrt{1-\psi}}$ and $\psi=\bigl(1-\sqrt{1-\psi}\bigr)\bigl(1+\sqrt{1-\psi}\bigr)$, $G=\rho(o)\psi\Phi$ gives (c). 
\end{proof}

\subsubsection{The volume variation}

We now prove a Schläfli type formula for the variation of the angle and volume.

\begin{proposition}\label{prop:Schlafli}

Let $t\mapsto(\beta(t),q(t))$ be a $C^1$ curve in
$\mathbb R\times\mathcal X_\nu^{2,\alpha}$ satisfying
$\mathcal E(\beta(t),q(t))=0$, and let
$g(t)=g_{\beta(t),q(t)}$ be as in
\eqref{eq:perturbed-metric}. Put $V(t):=\Vol(g(t))$ and denote
$\frac{d}{dt}$ by a dot. Then, we have
\begin{equation}\label{eq:Schlafli}
 \dot V=\frac{2\pi}{3}\Area(\Sigma,g)\,\dot\beta.
\end{equation}
\end{proposition}

\begin{proof}
Proposition~\ref{prop:recovery} gives
$\Ric_{g(t)}=3g(t)$, hence $D\Ric_g[\dot g]=3\dot g$ and
$D S_g[\dot g]=0$. The volume variation is $\dot V=\tfrac12\int_M\operatorname{tr}_g\dot g\,dV_g$ while the scalar-curvature variation is
\[
 0=-3\operatorname{tr}_g\dot g+\operatorname{div}_g\Upsilon^\sharp,
 \qquad \text{ with } \qquad \Upsilon:=\operatorname{div}_g\dot g-d\operatorname{tr}_g\dot g.
\]
This is the first-variation identity underlying the Schl\"afli formula
\cite[Theorem~8]{RS00}, see also \cite{RS99}. The cone corollary following \cite[Theorem~9]{RS00} is stated for nonpositive scalar curvature, so we give a direct proof in our setting.

Each $g(t)$ is smooth across $\pi^{-1}(\partial\Omega)$, with rod angles $2\pi$ (Corollary~\ref{cor:fixed-set-regularity}), and $t\mapsto q(t)$ is $C^1$ into $\mathcal X$. Moreover $\partial_n\dot b=0$ on the rods, by Proposition~\ref{prop:recovery} applied to each $q(t)$, so the even extension of $\dot b$ is $C^1$ and $\dot g$ is a $C^1$ tensor across $\pi^{-1}(\partial\Omega)$, and the divergence theorem on $M\setminus\{R<\varepsilon\}$ produces no boundary term at rods and corners.
Only the tube about
$\Sigma$ contributes to the boundary integral.
From \eqref{eq:fixed-cone-identification} and
\eqref{eq:central-domain}, on its boundary $R=\varepsilon$, we have
\[
 \begin{aligned}
 \Upsilon(\partial_R)=-\frac{2\dot\beta}{\beta R}+O(R^{\nu-1}),\qquad
 n=-\sqrt3e^{-a}\partial_R,\qquad\text{and}\qquad
 dA_g=\frac{e^a\beta R}{\sqrt3}(1+O(R^\nu))\,d\theta\,dA_{g|_\Sigma}.
 \end{aligned}
\]
Integration over $M\setminus\{R<\varepsilon\}$ gives
\eqref{eq:Schlafli}:
\[
 6\dot V
 =\lim_{\varepsilon\downarrow0}\int_{R=\varepsilon}\Upsilon(n)\,dA_g
 =4\pi\dot\beta\Area(\Sigma,g).
\]
\end{proof}

\begin{proposition}\label{prop:noncollapse}
Every metric $g=g_{\beta,q}$ as in \eqref{eq:perturbed-metric} with $0<\beta<1$ and $\mathcal E(\beta,q)=0$
satisfies
\begin{equation}\label{eq:area-volume}
\frac{2\pi\beta}{15}\Area(\Sigma,g)
 \leqslant\Vol(g)\leqslant\frac{2\pi\beta}{3}\Area(\Sigma,g).
\end{equation}
\end{proposition}

\begin{proof}
 By \eqref{eq:scalar-volume}, we have
\begin{equation}\label{eq:volume-psi}
  \Vol(g)=\Area(\Sigma,g)\int_\Omega\psi\,dA_h.
\end{equation}
By Corollary~\ref{cor:concavity}, $E_t$ is a closed disk with smooth boundary. If $T$ is the unit tangent of $\partial E_t=\{\psi=t\}$ and $n=\nabla\psi/|d\psi|_h$
its inward normal, the geodesic curvature $k_t$ of the level set $\{\psi=t\}$ is then
\[
\begin{aligned}
 k_t=\langle\nabla_TT,n\rangle_h
 &=-\frac{\Hess_h\psi(T,T)}{|d\psi|_h}
 =\frac{t}{|d\psi|_h}\left(3-K_h+\tfrac12B(T,T)\right).
\end{aligned}
\]
Gauss-Bonnet, with the central cone angle $2\pi\beta$, reads
\begin{equation}\label{eq:level-Gauss-Bonnet}
 2\pi\beta=\int_{E_t}K_h\,dA_h+
 \int_{\psi=t}\frac{t}{|d\psi|_h}
       \left(3-K_h+\tfrac12B(T,T)\right)ds_h.
\end{equation}

\emph{Two weighted identities.} The weight $dt$ cancels the curvature contribution, and the weight $2t\,dt$ converts the logarithmic-gradient contribution $\psi^2|d\log\psi|_h^2$ into the Dirichlet energy $|d\psi|_h^2$ of $\psi$. For $p\in\{1,2\}$, integrating \eqref{eq:level-Gauss-Bonnet} against $pt^{p-1}dt$ on $(0,1)$, the coarea formula and Fubini's theorem give
\begin{equation}\label{eq:level-flux}
 2\pi\beta=\int_\Omega\psi^p\bigl(3p+(1-p)K_h+\tfrac p2B(T,T)\bigr)dA_h.
\end{equation}
Here we first integrate \eqref{eq:level-Gauss-Bonnet} for $t\in(\varepsilon,1-\varepsilon)$ and then let $\varepsilon\downarrow0$. The expansions \eqref{eq:matrix-smooth-side} at the rods, \eqref{eq:matrix-smooth-corner} at the corners and \eqref{eq:center-refined} at the center justify the limits.
For $p=1$, $B(T,T)\geqslant0$ gives $2\pi\beta\geqslant3\int_\Omega\psi\,dA_h$, which is the upper bound of \eqref{eq:area-volume}, by \eqref{eq:volume-psi}. For $p=2$, integration by parts, using $\psi=0$ on $\partial\Omega$, \eqref{eq:center-refined} at $o$ and $\Delta_h\psi=-6\psi$, gives $\int_\Omega|d\psi|_h^2\,dA_h=6\int_\Omega\psi^2\,dA_h$. With $B(T,T)\leqslant\operatorname{tr}_hB=4K_h+|d\log\psi|_h^2$, $K_h\leqslant1$ and $\psi\leqslant1$, it follows that
\[
 2\pi\beta\leqslant\int_\Omega\psi^2(6+3K_h)\,dA_h+\int_\Omega|d\psi|_h^2\,dA_h\leqslant15\int_\Omega\psi^2\,dA_h\leqslant15\int_\Omega\psi\,dA_h,
\]
which is the lower bound of \eqref{eq:area-volume}, by \eqref{eq:volume-psi}.

\end{proof}

\begin{corollary}\label{cor:volume-branch}
Along the branch from $g_1$, the function $\beta\mapsto\beta^{-1}\Vol(g_\beta)$ is nondecreasing, and, with $V_1:=\Vol(g_1)$,
\begin{equation}\label{eq:volume-lower}
 V_1\beta^5\leqslant\Vol(g_\beta)\leqslant \beta V_1.
\end{equation}
\end{corollary}
\begin{proof}
Combining \eqref{eq:Schlafli} with \eqref{eq:area-volume} gives, along the branch,
\begin{equation}\label{eq:log-volume-derivative}
 \frac{1}{\beta}\leqslant\frac{d}{d\beta}\log\Vol(g_\beta)=\frac{2\pi}{3}\frac{\Area(\Sigma,g_\beta)}{\Vol(g_\beta)}\leqslant\frac{5}{\beta}.
\end{equation}
Integrate \eqref{eq:log-volume-derivative} from $\beta$ to $\beta'<1$ on the branch and let
$\beta'\uparrow1$. Corollary~\ref{cor:local} at $g_1$ gives
$\Vol(g_{\beta'})\to V_1$, proving \eqref{eq:volume-lower}. 
\end{proof}

\subsection{Control of the geometry near the central torus}\label{sec:intrinsic-cores}
The estimates in this subsection require no lower bound for the volume. Along the branch, Corollary~\ref{cor:volume-branch} gives one, but Lemma~ \ref{lem:smooth-cores} is also applied in Section~\ref{sec:collapse} to the metrics $g^{(k)}$ of class (C2), whose volumes tend to zero by Corollary~\ref{cor:volumes}.
We work on the superlevel sets $E_t=\{\psi\geqslant t\}$ defined at the beginning of this section. The prescribed coordinate $R$ is used only to check the regularity of $\sigma$ at $o$, in Lemma~\ref{lem:flattened-cores}.

Lemma~\ref{lem:scalar} controls $K_h$ and the first derivatives of $\psi$ and $\Phi$ on $E_t$ in terms of $t$ only. To obtain coordinates of fixed size, we flatten $h$ on $E_t$ by a conformal factor $e^{2\sigma}$ with $|\sigma|\leqslant\frac{1}{4}\log\frac{1}t$: in the resulting flat, respectively conic, coordinates the reduced equations \eqref{eq:matrix-reduced} have bounded coefficients.
\begin{lemma}
\label{lem:flattened-cores}
For $0<t<1$, the Friedrichs Dirichlet problem
\begin{equation}\label{eq:core-flattening}
 \Delta_h\sigma=K_h\quad\text{on }E_t,\qquad\text{and}\qquad \sigma|_{\partial E_t}=0
\end{equation}
has a unique solution, invariant under the central symmetries, with
\begin{equation}\label{eq:core-flattening-bound}
 |\sigma|\leqslant\tfrac14\log(\psi/t)\leqslant\tfrac14\log(1/t).
\end{equation}
The metric $\widetilde h=e^{2\sigma}h$ is flat away from $o$, has the same
cone angle at $o$, and satisfies
$t^{1/2}h\leqslant\widetilde h\leqslant t^{-1/2}h$.
For a smooth center, $\widetilde h$ is smooth and flat on all of $E_t$.
\end{lemma}
This is a local counterpart of the flattening of a doubled orbit space in \cite[Section~2.4, proof of Proposition~2.3, (2.37)]{Chr08}.
\begin{proof}
Corollary~\ref{cor:concavity} makes $E_t$ a closed disk with smooth boundary,
and Lemma~\ref{lem:scalar}(a) bounds its curvature: $|K_h|\leqslant t^{-2}$ on $E_t$. Hence $K_h\in L^2(E_t,dA_h)$, and the Riesz representation theorem in $H^1_0(E_t)$, normed by the Dirichlet energy, gives a unique weak solution $\sigma\in H^1_0(E_t)$ of \eqref{eq:core-flattening}. 

The estimate on $\sigma$ comes from barriers based on the conformal metrics $\psi^{\pm1/2}h$. By \eqref{eq:Einstein-identities}, $\Delta_h\log\psi=\psi^{-1}\Delta_h\psi-|d\ell|_h^2=-6-|d\ell|_h^2$ and $4K_h=|d\Phi|_h^2-|d\ell|_h^2$, so we have
\[
 K_h-\tfrac14\Delta_h\log\psi=\tfrac32+\tfrac14|d\Phi|_h^2>0,\qquad\text{and}\qquad
 K_h+\tfrac14\Delta_h\log\psi\leqslant-\tfrac12:
\]
the metrics $\psi^{1/2}h$ and $\psi^{-1/2}h$ have curvatures of opposite signs. Consequently, as $\Delta_h\sigma=K_h$, the functions $\frac{1}{4}\log(\psi/t)\pm\sigma\in H^1(E_t)$ are weakly superharmonic and vanish on $\partial E_t$. The weak minimum principle \cite[Theorem~8.1]{GT01} gives $\frac{1}{4}\log(\psi/t)\pm\sigma\geqslant0$, which is \eqref{eq:core-flattening-bound}. Near $o$, $\Delta_{\gamma_\beta}\sigma=e^{2a}K_h$ is bounded in $C^{0,\alpha}_{\mathrm b}$ by \eqref{eq:center-regularity}, and $\sigma$ is invariant, so Lemma~\ref{lem:cone-Schauder} with $\delta=2$ shows that $\sigma$ has a limit $\sigma(o)$ and $\sigma-\sigma(o)\in R^2C^{2,\alpha}_{\mathrm b}$, and in particular $\sigma$ has no logarithmic term. Thus the conformal change preserves the cone angle, while
$K_{\widetilde h}=e^{-2\sigma}(K_h-\Delta_h\sigma)=0$ away from $o$.
Uniqueness preserves the symmetries. 
\end{proof}

\emph{A local estimate.} Let $U'$ and $U$ be charts with $\overline{U'}\subset U$, and write $h=e^{2u}\gamma$, where $\gamma$ is a fixed smooth metric on $U$. By Proposition~\ref{prop:reduced}, \eqref{eq:matrix-reduced} reads $\Delta_\gamma\psi=-6e^{2u}\psi$, $\Delta_\gamma\Phi=|d\Phi|_\gamma^2\Phi-d\Phi(\nabla_\gamma\log\psi)$ and $\Delta_\gamma u=K_\gamma-\frac{1}{4}(|d\Phi|_\gamma^2-|d\log\psi|_\gamma^2)$. For $p>\frac{2}{1-\alpha}$, the interior $W^{2,p}$ estimate and then Schauder estimates \cite[Theorems~9.11 and~6.2]{GT01} therefore give
\begin{equation}\label{eq:local-regularity}
 C^{-1}\gamma\leqslant h\leqslant C\gamma,\quad
 \psi+\psi^{-1}+|d\psi|_\gamma+|d\Phi|_\gamma+d_{\mathbb H^2}(\Phi,\Phi_o)\leqslant C\ \text{ on }U
 \ \Longrightarrow\ \|(u,\psi,\Phi)\|_{C^{j,\alpha}(U')}\leqslant C_j,
\end{equation}
where $C_j$ depends only on $C$, $U$, $U'$, the coefficients of $\gamma$, $j$ and $\alpha$. Bounds for $\rho=\rho(o)\psi$ also use the bound for $\rho(o)$.
\begin{lemma}\label{lem:curv-central}
 For every compact interval $I\subset(0,1)$, there are $r_*,C>0$, depending only on $I$ and the symmetry class, such that every metric $g=g_{\beta,q}$ as in \eqref{eq:perturbed-metric} with $\beta\in I$ and $\mathcal E(\beta,q)=0$ satisfies
\begin{equation}\label{eq:curv-central}
 \sup_{\{0<d_g(\Sigma,\cdot)<r_*\}}|\Rm(g)|\leqslant C.
\end{equation}
 \end{lemma}

\begin{proof}
\emph{Flat balls.} Let $0<t<1$, and let $\widetilde h=e^{2\sigma}h$ be the flat metric of Lemma~\ref{lem:flattened-cores} on the core $E_t$, smooth in class (C2), and with one cone point $o$ of angle $2\pi\beta<2\pi$ in class (C1). For $q\in E_t$ in class (C2), and for $q=o$ in class (C1), the $\widetilde h$-exponential map at $q$ is an isometry from the Euclidean disk, respectively the flat cone of angle $2\pi\beta$, of radius $d_{\widetilde h}(q,\partial E_t)$ onto $B_{\widetilde h}(q,d_{\widetilde h}(q,\partial E_t))$. Indeed, $\widetilde h$ has no conjugate points, and a first point in the cut locus would give a flat geodesic bigon in $E_t$, whose two interior angles would sum to $0$ by Gauss-Bonnet.

On the universal cover $\mathbb R^2$ of the torus orbits put $y_a:=\sqrt{\rho(o)}\,\phi_a$.
Then
$g=h+\psi\Phi_{ab}\,dy_ady_b$. This change is only for local estimates
and does not change the generators $X_1,X_2$ of \eqref{eq:matrix-ansatz}, whose flows are $2\pi$-periodic.

\emph{ The conic center.} Apply Lemma~\ref{lem:flattened-cores} on
$E_{1/2}$. Since $|d\psi|_h\leqslant2$, this yields
\[
 |\sigma|\leqslant\tfrac14\log2,\qquad
 d_h(o,\partial E_{1/2})\geqslant\tfrac14,\qquad\text{and}\qquad
 d_{\widetilde h}(o,\partial E_{1/2})\geqslant\tfrac14\,2^{-1/4}.
\]
Put $r_0:=\frac{1}{4}\,2^{-1/4}$. The ball $B_{\widetilde h}(o,r_0)$ is a flat cone of angle $2\pi\beta$, by the paragraph on flat balls and $d_{\widetilde h}(o,\partial E_{1/2})\geqslant r_0$. Write there
\[
 \widetilde h=\tfrac13(d\widetilde R^2+\beta^2\widetilde R^2d\theta^2),\qquad\text{and}\qquad
 h=e^{-2\sigma}\widetilde h.
\]
Here $\widetilde R$ is the radius of this local chart, which is a priori distinct from the
prescribed $R$. The central symmetries act by their original rotations and reflections in the chart $(\widetilde R,\theta)$, so the mode restrictions
of Lemma~\ref{lem:central-modes} are unchanged.

Apply Lemma~\ref{lem:center-expansion} with $(\gamma_\beta,\rho,\Phi,a,R)$ replaced by $(\widetilde h,\psi,\Phi,-\sigma,\widetilde R)$. Proposition~\ref{prop:reduced} verifies (H1), since the reference metric may be changed and $\rho/\psi=\rho(o)$ is constant. On $B_{\widetilde h}(o,r_0)\subset E_{1/2}$, Lemmas~\ref{lem:scalar} and~\ref{lem:flattened-cores} give
\[
 |\sigma|\leqslant\tfrac14\log2,\qquad \tfrac12\leqslant\psi\leqslant1,\qquad\text{and}\qquad
 |d\psi|_{\widetilde h}+|d\Phi|_{\widetilde h}\leqslant e^{|\sigma|}\bigl(|d\psi|_h+|d\Phi|_h\bigr)\leqslant6\cdot2^{1/4},
\]
which gives \eqref{eq:central-quantitative-input} with a universal $M$. Lemma~\ref{lem:center-expansion}, with $g^0=e^{-2\sigma(o)}\widetilde h+(\Phi_o)_{ab}\,dy_ady_b$, gives $|\Rm(g)|\leqslant C$ and $C^{-1}g^0\leqslant g\leqslant Cg^0$ on the chart.
Since $C=C(I)$, this chart contains a tube $\{0<d_g(\Sigma,\cdot)<r_*\}$, proving \eqref{eq:curv-central}.

\end{proof}

\begin{lemma}\label{lem:smooth-cores}
For every $g$ in (C2) and $\delta\in(0,1)$, the core $E_\delta$ admits a normal chart at $o$ and, where needed, a cylinder chart in the coordinates $(t,\Theta)$ of the proof of Lemma~\ref{lem:core-retraction}, on coordinate domains independent of $g$. In these charts, $h$ is uniformly positive and $h$, $\psi$ and $\Phi$ have uniform $C^{j,\alpha}$ bounds, and $\sup_{\pi^{-1}(E_\delta)}|\Rm(g)|\leqslant C_\delta$. The constants depend only on $\delta$, $j$, $\alpha$ and the symmetry class, not on $\rho(o)$.
\end{lemma}
\begin{proof}
\emph{ Cores of metrics with a smooth center.} We use the flat balls and the fiber coordinates $y_a:=\sqrt{\rho(o)}\,\phi_a$ of the proof of Lemma~\ref{lem:curv-central}. Apply Lemma~\ref{lem:flattened-cores} on $E_{\delta/2}$. Then, we have
\[
 |\sigma|\leqslant\tfrac14\log(2/\delta),\qquad
 d_h(E_\delta,\partial E_{\delta/2})\geqslant\tfrac\delta4,\qquad\text{and}\qquad
 d_{\widetilde h}(E_\delta,\partial E_{\delta/2})
       \geqslant\tfrac\delta4(\delta/2)^{1/4}.
\]
 With $r_\delta:=\frac{\delta}{4}(\frac{\delta}{2})^{1/4}$,  the  flat balls of radius $r_\delta$  about the points of $E_\delta$ are Euclidean disks contained in $E_{\delta/2}$. In such a chart $h=e^{-2\sigma}|dx|^2$, which with the bound for $\sigma$ below gives the metric comparison. On $E_{\delta/2}$, Lemmas~\ref{lem:scalar} and~\ref{lem:flattened-cores} give $\frac{\delta}{2}\leqslant\psi\leqslant1$, $|\sigma|\leqslant\frac{1}{4}\log(2/\delta)$, $|d\psi|_h\leqslant2$, $|d\Phi|_h\leqslant4\delta^{-1}$ and $d_{\mathbb H^2}(\Phi,\Phi_o)\leqslant2\operatorname{artanh}\sqrt{1-\delta/2}$. These are the hypotheses of \eqref{eq:local-regularity} with $\gamma=|dx|^2$, $u=-\sigma$, $U=B(q,r_\delta)$ and $U'=B(q,\frac{1}{2}r_\delta)$, with constants depending only on $\delta$.

In the fiber coordinates $y_a$, the metric $g=e^{-2\sigma}|dx|^2+\psi\Phi_{ab}\,dy_ady_b$ then has coefficients bounded in every $C^{j,\alpha}$ and uniformly positive on $B(q,\frac{1}{2}r_\delta)\times\mathbb R^2$, which gives the curvature bound.

\emph{Core charts.} The bounds above with $\delta=\frac{1}{2}$ give a normal chart of $h$ at $o$ on a ball $B_h(o,r_1)$, with $r_1>0$ independent of $g$, in which $h$, $\psi$ and $\Phi$ are bounded in every $C^{j,\alpha}$. Since $\Hess_h\psi(o)=-3h(o)$ in class (C2), by \eqref{eq:Einstein-identities}, there is $t_*<1$ independent of $g$ such that $E_{t_*}$ lies in this chart and $\{\psi=t_*\}$ is a radial graph, parametrized by the normal angle $\Theta$. On $\{\delta\leqslant\psi\leqslant t_*\}$, \eqref{eq:base-radius} gives $|d\psi|_h^2\geqslant4\delta(1-t_*)$, so the derivative bounds on $E_{\delta/2}$ control all derivatives of $V_\psi$, and the coordinates $(t,\Theta)$ of Lemma~\ref{lem:core-retraction}, with reference level set $\{\psi=t_*\}$, are the cylinder chart, with the stated bounds.
\end{proof}

\subsection{Exclusion of curvature concentration}
\label{sec:curv-exclusion}

The edge Chern-Gauss-Bonnet formula \cite[Theorem~2.1]{AL13}, with $\chi(\Sigma)=0$ and $\chi(M)=m\in\{4,6\}$, gives, for every metric $g=g_{\beta,q}$ as in \eqref{eq:perturbed-metric} with $0<\beta<1$ and $\mathcal E(\beta,q)=0$,
\begin{equation}\label{eq:curv-energy}
 \int_{M\setminus\Sigma}|\Rm(g)|^2\,dV_g
 =32\pi^2\bigl(\chi(M)-(1-\beta)\chi(\Sigma)\bigr)
 =32\pi^2m.
\end{equation}

Killing vector fields are controlled along a path by their $1$-jets. Let $X$ be a Killing vector field, and $\gamma:[0,1]\to M$ a path of length at most $Ar$ along which $|\Rm|\leqslant Ar^{-2}$, for some $A,r>0$. Then, for a constant $C_A$ depending only on $A$, we have
\begin{equation}\label{eq:killing-transport}
 C_A^{-1}\bigl(r^{-1}|X|+|\nabla X|\bigr)(\gamma(0))\leqslant\bigl(r^{-1}|X|+|\nabla X|\bigr)(\gamma(1))\leqslant C_A\bigl(r^{-1}|X|+|\nabla X|\bigr)(\gamma(0)).
\end{equation}
 After scaling the metric by $r^{-2}$,  the Killing identities give a first-order linear system for $(X,\nabla X)$ along  $\gamma$, parametrized by arclength, with bounded coefficients, and Gronwall's inequality in both directions  gives \eqref{eq:killing-transport}, as in the proof of \cite[Proposition~16]{CLW08}, compare \cite[p.~404]{Don08}. No lower bound for the injectivity radius or the volume is used. Differentiating the Killing identities bounds the higher derivatives of $X$ in the same way where the derivatives of the curvature are bounded.

The following lemma shows that the rod vectors, and hence \eqref{eq:rod-determinant}, pass to geometric limits. It is also used for the Ricci-flat limits of Section~\ref{sec:collapse}.

Let $(U_i,g_i,p_i)$ be pointed Riemannian four-manifolds converging smoothly on compact subsets to a complete connected pointed Riemannian four-manifold $(\mathcal N,g_{\mathcal N},p)$. Assume that:
\begin{itemize}
\item $\mathbb T^2:=\mathbb R^2/2\pi\mathbb Z^2$ acts effectively and isometrically on each $U_i$, generated by Killing vector fields $X_{1,i},X_{2,i}$ with $2\pi$-periodic flows, and
\item the rods and corners of $U_i$ have the models \eqref{eq:matrix-smooth-side} and \eqref{eq:matrix-smooth-corner}, where a \emph{corner} is a point fixed by $\mathbb T^2$, and the rod vectors are, up to sign, $e_1,e_2$ in the square case and $e_1,e_2,e_1+e_2$ in the hexagon case.
\end{itemize}

\begin{lemma}\label{lem:torus-limit}

Let $(U_i,g_i,p_i)$ and $(\mathcal N,g_{\mathcal N},p)$ satisfy the two assumptions above.
\begin{enumerate}[label=\textup{(\alph*)}]
\item \emph{Limit action.} If corners $c_i\in U_i$ converge to $c\in\mathcal N$, then, after a subsequence, $X_{1,i},X_{2,i}$ converge smoothly on compact subsets to Killing vector fields $X_1,X_2$ of $g_{\mathcal N}$. They generate an effective isometric $\mathbb T^2$-action on $\mathcal N$, and on each ball about $c$ the convergence maps can be chosen equivariant. They then preserve the stabilizers, hence the rods, rod vectors and corners.
\item \emph{Existence of a corner.} If moreover $g_{\mathcal N}$ is nonflat and Ricci-flat, with finite curvature energy, then, after a subsequence, there are corners $c_i\in U_i$ converging to a point of $\mathcal N$.
\end{enumerate}
\end{lemma}
\begin{proof}
\emph{(a)} An arbitrary normalization of the first jets of the limiting Killing fields need not preserve the original integral lattice $2\pi\mathbb Z^2$ of periods. At a corner, the primitive rod fields have unit rotation derivatives in the normal planes, whatever the scale, and this fixes their normalization.

 After a subsequence, the adjacent rod vectors $v_1,v_2$ at $c_i$ are fixed. By the corner model, $V_{a,i}:=(v_a)_1X_{1,i}+(v_a)_2X_{2,i}$ vanishes at $c_i$, and $\nabla V_{1,i}(c_i)$, $\nabla V_{2,i}(c_i)$ are the rotations $J\oplus0$ and $0\oplus J$ of an orthogonal splitting of $T_{c_i}U_i$ into two planes, with $J:=\bigl(\begin{smallmatrix}0&-1\\1&0\end{smallmatrix}\bigr)$ in orthonormal bases. As $(v_1,v_2)\in\mathrm{GL}(2,\mathbb Z)$, the $1$-jets of $X_{1,i},X_{2,i}$ at $c_i$ are therefore bounded and independent, uniformly in $i$, and  \eqref{eq:killing-transport}, with its differentiated form, gives, on each compact subset $Q$ of the converging regions and for $q\in Q$,
\begin{equation}\label{eq:killing-propagation}
 \|X\|_{C^{j+1}(Q)}\leqslant C_{Q,j}\bigl(|X(q)|+|\nabla X(q)|\bigr).
\end{equation}

 Hence $X_{a,i}\to X_a$ smoothly on compact subsets, and the limits, with $2\pi$-periodic flows and the above $1$-jets at $c$, generate an effective action with the lattice $2\pi\mathbb Z^2$ \cite[proof of Theorem~27]{CLW08}. 

 On each fixed ball about $c$, the center-of-mass construction \cite[Propositions~3.1 and~3.7]{GK73} modifies the convergence embeddings by a $C^\infty$-small amount to embeddings $F_i$ with $(F_i)_*X_a=X_{a,i}$, which intertwine the actions and therefore preserve the stabilizers.

\emph{(b)}  Bases $X'_{1,i},X'_{2,i}$ of the span of $X_{1,i},X_{2,i}$, orthonormal for the inner product of $1$-jets at $p_i$, converge, by \eqref{eq:killing-propagation} with $q=p_i$, to independent commuting Killing vector fields $Y_1,Y_2$ of $g_{\mathcal N}$. 

 By \cite[Theorem~1.1]{Web19}, every Killing vector field of $g_{\mathcal N}$ vanishes somewhere, and the isometries fixing a point form a compact group. So the closures $K_1,K_2$ of the flows of $Y_1,Y_2$ are compact, $\mathbb T':=K_1K_2$ is a torus, and a Killing vector field whose flow is dense in $\mathbb T'$ vanishes at a point $c$, again by \cite{Web19}, which is fixed by $\mathbb T'$.  The commuting skew endomorphisms $\nabla Y_1(c)$ and $\nabla Y_2(c)$ are independent, since a Killing vector field vanishing at $c$ is determined by its derivative there, so in a common orthonormal basis they are $\operatorname{diag}(a_1J,a_2J)$ and $\operatorname{diag}(b_1J,b_2J)$ with $(a_1,a_2)$ and $(b_1,b_2)$ independent. Choose $(\lambda_1,\lambda_2)$ with $\lambda_1a_j+\lambda_2b_j\neq0$ for $j=1,2$, so that $Z:=\lambda_1Y_1+\lambda_2Y_2$ has invertible derivative $\nabla Z(c)$, and put $Z_i:=\lambda_1X'_{1,i}+\lambda_2X'_{2,i}$. As $Z_i\to Z$ smoothly, $Z_i$ vanishes at a point $c_i$, isolated in the vanishing set of $Z_i$, with $c_i\to c$. The torus acting on $U_i$ commutes with $Z_i$, which is a constant combination of its generators, so it preserves the vanishing set of $Z_i$. The orbit of $c_i$ is connected and lies in this set, in which $c_i$ is isolated, so it is $\{c_i\}$: the point $c_i$ is fixed by the torus, that is, a corner.
\end{proof}

\begin{corollary}\label{cor:ale-exclusion}
Let $(\mathcal N,g_{\mathcal N})$ be a complete Ricci-flat four-manifold with finite curvature energy, Euclidean volume growth and an effective isometric action of $\mathbb R^2/2\pi\mathbb Z^2$, whose orbits are principal orbits, rods and corners, and whose primitive rod vectors $v,v'$ satisfy
\[
 |\det(v,v')|=1\qquad\text{whenever }v\neq\pm v'.
\] Then $(\mathcal N,g_{\mathcal N})$ is isometric to the Euclidean $\mathbb R^4$.
\end{corollary}
\begin{proof}
 By \cite[Theorem~1.5]{BKN89}, $(\mathcal N,g_{\mathcal N})$ is ALE.

\emph{An invariant cross-section of the end.} The tangent cone $\mathbb R^4/\Gamma$ at infinity alone does not identify the stabilizers, so we construct a  $\mathbb T^2$-invariant cross-section of the end. Let $r$ be  an ALE radius, extended smoothly inside, and let $H$ be the average of $r^2\circ t$ over $t\in\mathbb T^2$, an invariant proper function. Since $\Hess r^2=2g_{\mathcal N}+o(1)$ and $\mathbb T^2$  acts by isometries, $\Hess H=2g_{\mathcal N}+o(1)$, and $|dH|=O(r)$. Along the radial rays of the ALE coordinates,  whose acceleration is $o(r^{-1})$, this gives $\partial_r^2H=2+o(1)$, hence $\partial_rH=2r+o(r)>0$. So the large level sets $Y$ of $H$ are invariant radial graphs over $\mathbb S^3/\Gamma$.

 The action on $Y$ is effective: an element  fixing $Y$ pointwise also fixes  its invariant unit normal, hence is the identity isometry.  Since $\pi_1(Y)\approx\Gamma$ is finite, the orbit space $Y/\mathbb T^2$ is an interval  whose endpoints are circle orbits, with stabilizers generated by the rod vectors  $v_-,v_+$ of the two unbounded rods of $\mathcal N$  \cite[Section~2.5]{OR70}. Thus
\[
 \Gamma\approx\mathbb Z^2/(\mathbb Zv_-+\mathbb Zv_+),\qquad\text{and}\qquad |\Gamma|=|\det(v_-,v_+)|,
\]
 compare \cite[Section~2.2, (2.6)]{CS04}. In the application to blow-ups, the two unbounded rods need not arise from adjacent sides of the orbit space, which is why the hypothesis concerns all pairs of distinct rod vectors. It gives $|\Gamma|=1$, so the asymptotic volume ratio of $g_{\mathcal N}$ is one, and the equality case of the Bishop-Gromov inequality gives $(\mathcal N,g_{\mathcal N})\approx\mathbb R^4$.
\end{proof}

For the remainder of Section~\ref{sec:compactness}, we assume a positive lower bound for the volume. For noncollapsed sequences, fix a compact interval $I\subset(0,1)$ and let $g_i=g_{\beta_i,q_i}$ be metrics as in \eqref{eq:perturbed-metric} with $\mathcal E(\beta_i,q_i)=0$ and
\begin{equation}\label{eq:coord-assumptions}
 \beta_i\in I\qquad\text{and}\qquad \Vol(g_i)\geqslant v>0.
\end{equation}
Constants then depend only on $I$, $v$ and the symmetry class. The results of Sections~\ref{sec:noncollapse-first} and~\ref{sec:intrinsic-cores} stated for class (C2), Lemma~\ref{lem:torus-limit} and Corollary~\ref{cor:ale-exclusion} use no such bound, and they are also used for the collapsing metrics of Section~\ref{sec:collapse}.

Since $\operatorname{Ric}(g_i)=3g_i$ and the cone angle along $\Sigma$ is $2\pi\beta_i<2\pi$, \cite[Theorem~A and Corollary~4.1]{BKMR21}  and the diameter estimate of \cite{Mon18} give, with $\mathsf V(r):=2\pi^2\int_0^r\sin^3t\,dt$ and for $0<r\leqslant\pi$, the Bishop-Gromov bounds
\begin{equation}\label{eq:local-noncollapse}
 \operatorname{diam}(M,g_i)\leqslant\pi,\qquad\text{ and }\qquad
 v\frac{\mathsf V(r)}{\mathsf V(\pi)}
 \leqslant\Vol_{g_i}B_{g_i}(p,r)\leqslant\mathsf V(r).
\end{equation}

\begin{proposition}\label{prop:curv-bound}
There is $C=C(I,v)$ such that every metric $g=g_{\beta,q}$ as in \eqref{eq:perturbed-metric} with $\mathcal E(\beta,q)=0$, $\beta\in I$ and $\Vol(g)\geqslant v$ satisfies
\[
 \sup_{M\setminus\Sigma}|\Rm(g)|\leqslant C.
\]
\end{proposition}

\begin{proof}
Suppose, towards a contradiction, that $\Lambda_i=\max|\Rm(g_i)|=|\Rm(g_i)|(p_i)\to+\infty$.
By \eqref{eq:curv-central}, $d_{g_i}(p_i,\Sigma)\geqslant r_*$.
Thus $\Sigma$ escapes after rescaling by $\Lambda_i$: $d_{\Lambda_ig_i}(p_i,\Sigma)\geqslant\Lambda_i^{1/2}r_*\to+\infty$.
Noncollapse \eqref{eq:local-noncollapse} and Einstein compactness \cite[Theorem~2.2 and Section~3]{And89} therefore give
\[
 (M\setminus\Sigma,\Lambda_i g_i,p_i)\longrightarrow(\mathcal N,g_{\mathcal N},p),
 \qquad \Ric(g_{\mathcal N})=0,\qquad\text{and}\qquad |\Rm(g_{\mathcal N})|(p)=1,
\]
where $\mathcal N$ has Euclidean volume growth, by \eqref{eq:local-noncollapse}, and finite curvature energy, by \eqref{eq:curv-energy}. Since $\Sigma$ escapes, Lemma~\ref{lem:torus-limit}(b), applied with $U_i:=M\setminus\Sigma$, gives corners $c_i$ converging to a point $c\in\mathcal N$. Lemma~\ref{lem:torus-limit}(a) then gives an effective isometric action of $\mathbb R^2/2\pi\mathbb Z^2$ on $\mathcal N$, generated by the limits of $X_1,X_2$, and convergence maps on balls about $c$ which preserve the stabilizers. So every orbit of $\mathcal N$ is the image of an orbit of $M\setminus\Sigma$ with the same stabilizer: a principal orbit, a corner, or a rod with one of the rod vectors of Section~\ref{sec:geometry}, which satisfy \eqref{eq:rod-determinant}. Corollary~\ref{cor:ale-exclusion} therefore gives $(\mathcal N,g_{\mathcal N})\approx\mathbb R^4$, which contradicts $|\Rm(g_{\mathcal N})|(p)=1$.
\end{proof}

\subsection{Compactness in the prescribed coordinates}
\label{sec:coord-convergence}

Before the identifications $\Psi_\beta$, every quotient metric is conformal to $h_1$ (Section~\ref{sec:common-coordinates}), so it suffices to bound the coefficients of the metrics in the fixed charts.

\paragraph{The fixed-point set of $\mathsf c$.}\label{sec:reflected-section}
At the rods and corners, controls of the conformal factor are delicate: maximum principles and elliptic estimates on the orbit space would require boundary conditions. We avoid them by working on the fixed-point set of the isometry $\mathsf c$, a totally geodesic surface whose Gauss curvature is a sectional curvature of $g$, and on which the rods and corners are interior points.

Let $g$ be in (C1) or (C2). It is smooth across $\pi^{-1}(\partial\Omega)$, by Corollary~\ref{cor:fixed-set-regularity} and, in class (C2), Section~\ref{subsec:smooth-covers}.
Following Donaldson \cite[Section~3.1, p.~401]{Don08}, we introduce
\begin{equation}\label{eq:reflected-section}
 S:=\operatorname{Fix}(\mathsf c)\subset M\setminus\Sigma,\qquad
 S\cap\pi^{-1}(\Omega\setminus\{o\})=\coprod_{\phi_0\in\{0,\pi\}^2}S_{\phi_0},\qquad\text{and}\qquad S_{\phi_0}:=\{\boldsymbol\phi=\phi_0\}.
\end{equation}
As the fixed-point set of the isometry $\mathsf c$, $S$ is a smooth totally geodesic surface, and $\pi$ maps each $S_{\phi_0}$ isometrically onto $(\Omega\setminus\{o\},h)$, so that, on $S$, we have
\begin{equation}\label{eq:reflected-facts}
 K_h=\sec_g(TS),\qquad |\nabla_h^jK_h|\leqslant|\nabla_g^j\Rm_g|,\qquad\text{and}\qquad G_{ab}=g(X_a,X_b)|_S,
\end{equation}
with $|\nabla_h^jG_{ab}|$ bounded by the $g$-norms of $\nabla_g^iX_a$, $0\leqslant i\leqslant j$. 

The point is that $S$ extends smoothly across the rods and corners, so that the bounds \eqref{eq:reflected-facts} hold there as well, and more importantly that its completion is a \textit{closed surface}.
In the Cartesian coordinates \eqref{eq:cartesian}, $S=\{\mathsf Y_1=\mathsf Y_2=0\}$ near a corner, so the signed coordinates $(x,y):=(\mathsf X_1,\mathsf X_2)$ near a corner, and $(x,y):=(\mathsf X_1,y)$ near a rod, are coordinates on $S$, the \emph{reflected charts}, in which
\begin{equation}\label{eq:reflected-chart}
 h=e^{2\widehat a}(dx^2+dy^2)\ \text{at a corner},\qquad\text{and}\qquad h=e^{2\bar a}(dx^2+dy^2)\ \text{at a rod},
\end{equation}
with $\widehat a$, respectively $\bar a$, smooth and even in each signed variable.
Since $x_a^2d\phi_a=\mathsf X_a\,d\mathsf Y_a$ on $S$, we have, along $S$ near a corner,
\begin{equation}\label{eq:reflected-metric}
 g=e^{2\widehat a}(d\mathsf X_1^2+d\mathsf X_2^2)+\widehat G_{ab}\,d\mathsf Y_a\,d\mathsf Y_b,
\end{equation}
where $\widehat G$ is extended with corresponding parity.

 In class (C1), the completion $\bar S$ of $S$ adds the four points $\{o\}\times\{0,\pi\}^2$, which are cone points of angle $2\pi\beta$.  It is a closed surface, the union of four copies of $\overline\Omega$ glued along the rods,  with four faces, $2m$ edges and $m$ vertices, so that $\chi(\bar S)=4-m$, compare \cite[Example~1.20]{DJ91}.

 We first bound $a$, in Lemma~\ref{lem:coord-conformal}, using $\Delta_{\gamma_\beta}a=K_{\gamma_\beta}-K_he^{2a}$ on  $\bar S$. The rods and corners are interior points of $\bar S$, so this bound requires no boundary condition, compare with the doubling of an orbit space in \cite[Proposition~2.1 and Section~2.4]{Chr08}. Once $a$ is bounded, the elliptic reduced equations, Lemma~\ref{lem:center-expansion} and \eqref{eq:local-regularity} bound all derivatives.

 \emph{Uniform geometry.} Let $g$  be as in \eqref{eq:coord-assumptions}.  By Lemma~\ref{lem:curv-central}  and
Proposition~\ref{prop:curv-bound}, $|\Rm(g)|\leqslant C$ on
$M\setminus\Sigma$. Hence $|K_h|\leqslant C$ on $S$ and on
$\Omega\setminus\{o\}$, by \eqref{eq:reflected-facts}. Fix $t>0$. Together with the noncollapsing
\eqref{eq:local-noncollapse}, the local estimate of
\cite[Theorem~4.3]{CGT82} bounds the injectivity radius of $g$ from
below on $\{d_g(\cdot,\Sigma)\geqslant t\}$, and the derivative estimates for Einstein metrics \cite{And89} give, with constants depending only on $I$, $v$ and $t$,
\begin{equation}\label{eq:uniform-geometry}
 \operatorname{inj}_g\geqslant c_t\qquad\text{and}\qquad|\nabla_g^j\Rm(g)|\leqslant C_j\qquad\text{on }\{d_g(\cdot,\Sigma)\geqslant t\}.
\end{equation}
These bounds do not use the upper bound for $a$ of Lemma~\ref{lem:coord-conformal} below. Since $S$ is totally geodesic, the restrictions $\exp_p^g|_{T_pS}$, for $p\in S\cap\{d_g(\cdot,\Sigma)\geqslant t\}$, give normal charts on $S$ of radius $c_t$, and \eqref{eq:reflected-facts} bounds all derivatives of $K_h$ there: $|\nabla_h^jK_h|\leqslant C_j$.

\begin{lemma}\label{lem:coord-conformal}
Every metric as in \eqref{eq:coord-assumptions} satisfies
$|a|\leqslant C$ on $\Omega$. Equivalently,
$e^{-2C}\gamma_\beta\leqslant h\leqslant e^{2C}\gamma_\beta$.
\end{lemma}

\begin{proof}
We keep the notation $\gamma_\beta$ for the reflection of $\gamma_\beta$ to the closed reflected surface $\bar S$  just constructed. It is
smooth away from the cone points, because $\gamma_1$ is even across the
rods and at the corners (Section~\ref{sec:common-coordinates}), and $h=e^{2a}\gamma_\beta$ on every copy $S_{\phi_0}$. It remains to bound the solution $a$ of the curvature equation $\Delta_{\gamma_\beta}a=K_{\gamma_\beta}-K_he^{2a}$.

\emph{A negatively curved reference metric.} By Gauss-Bonnet on $\bar S$, with $\chi(\bar S)=4-m$ and four cone points of angle $2\pi\beta$, the following constant is positive:
\[
 c_\beta:=-\frac{\int_{\bar S}K_{\gamma_\beta}\,dA_{\gamma_\beta}}{\Area(\bar S,\gamma_\beta)}=\frac{2\pi(m-4\beta)}{\Area(\bar S,\gamma_\beta)}>0.
\]
Let $v_\beta$ be the solution with $\int_{\bar S}v_\beta\,dA_{\gamma_\beta}=0$ of
\[
 \Delta_{\gamma_\beta}v_\beta=K_{\gamma_\beta}+c_\beta.
\]
For $\beta\in I$, the metrics $\gamma_\beta$ have uniform divergence-form ellipticity in the cone charts and uniform Poincar\'e constants, so the $L^\infty$ estimate of \cite[Theorem~8.17]{GT01} and the normalization $\int_{\bar S}v_\beta\,dA_{\gamma_\beta}=0$ give $\|v_\beta\|_{L^\infty}\leqslant C_I$. The metric $k_\beta:=e^{2v_\beta}\gamma_\beta$ has cone angle $2\pi\beta$ at the four cone points and
\[
 \begin{gathered}
 K_{k_\beta}=e^{-2v_\beta}\bigl(K_{\gamma_\beta}-\Delta_{\gamma_\beta}v_\beta\bigr)=-c_\beta e^{-2v_\beta}\in[-C_I,-c_I],\\
 \text{and}\qquad \operatorname{diam}(\bar S,k_\beta)\geqslant c_I>0\qquad\text{for }\beta\in I.
 \end{gathered}
\]

\emph{Lower bound.} Write $h=e^{2u}k_\beta$, so that $u=a-v_\beta$ and
\[
 \Delta_{k_\beta}u=K_{k_\beta}-K_he^{2u}.
\]
The weak form of this equation holds across the four cone points by Lemma~\ref{lem:B-cutoffs}, since $h$ and $k_\beta$ have the same cone angle there, so that $u$ is continuous with finite energy \cite[Proposition~15]{MW17}. By Lemma~\ref{lem:curv-central}  and Proposition~\ref{prop:curv-bound},
$K_h\geqslant-C_1$.
For a constant $u_0$ with $C_1e^{2u_0}<c_I$, the right-hand side is negative on the open set $\{u<u_0\}$, so $u$ is weakly superharmonic there, with $u=u_0$ on its boundary, and the weak minimum principle \cite[Theorem~8.1]{GT01} makes this set empty. Hence $u\geqslant-C$.

\emph{Diameter.} By \eqref{eq:base-radius}, every point of the closure of a copy $S_{\phi_0}$ is joined to the cone point $o_{\phi_0}:=(o,\phi_0)$ by a path of $h$-length at most $\frac{\pi}{2}$ in that copy. For a corner $c$, the point $\tilde c$ of $S$ over $c$ lies in the closure of all four copies, so every point of $\bar S$ lies at distance at most $\pi$ from $\tilde c$, and
\begin{equation}\label{eq:reflected-diameter}
 \operatorname{diam}(\bar S,h)\leqslant2\pi.
\end{equation}

\emph{Upper bound.}  By the lower bound, $\Delta_hu=e^{-2u}K_{k_\beta}-K_h$ is bounded: $|\Delta_hu|\leqslant C$. Let $p$ be a minimum point of $u$, and put $u_p:=u-u(p)\geqslant0$. Cover $\bar S$ by balls of a fixed radius in two kinds of charts: normal coordinates of $S$ at $h$-distance at least $\frac{1}{2}r_*$ from the cone points, given by \eqref{eq:uniform-geometry}, and the cone charts $(\widetilde R\cos\theta,\widetilde R\sin\theta)$ constructed in the proof of Lemma~\ref{lem:curv-central}  near them. $\Delta_h$ is uniformly elliptic in both, for $\beta\in I$. By \eqref{eq:reflected-diameter}, every $q\in\bar S$ is joined to $p$ by a chain of at most $C(I,v)$ overlapping such balls. Applied to $u_p$ from ball to ball, starting from $u_p(p)=0$, the Harnack inequality \cite[Theorems~8.17 and~8.18]{GT01} gives $u_p(q)\leqslant C$. 

 $h\geqslant e^{2u(p)}k_\beta$ and \eqref{eq:reflected-diameter} give $e^{u(p)}c_I\leqslant e^{u(p)}\operatorname{diam}(\bar S,k_\beta)\leqslant2\pi$, so $u\leqslant u(p)+C\leqslant C$. Hence $|u|\leqslant C$ and $|a|\leqslant|u|+|v_\beta|\leqslant C$.

\end{proof}

\begin{proposition}\label{prop:coord-compactness}
Assume \eqref{eq:coord-assumptions} and $\beta_i\to\beta_*\in I$. Then, after passing to a subsequence, there is an Einstein metric $g_\infty$ with cone angle $2\pi\beta_*$ along $\Sigma$, smooth across the rods and corners and satisfying the assumptions made on $g_*$ in Section~\ref{sec:spaces}, such that, in the prescribed coordinates \eqref{eq:fixed-cone-identification} and for $1<\nu<2$, we have
\begin{equation}\label{eq:coord-deformation-convergence}
 q_i=\bigl(\ell_i-\ell_\infty,\ \Phi_\infty\log(\Phi_\infty^{-1}\Phi_i),\
 a_i-a_\infty\bigr)\longrightarrow0\quad\text{in }\mathcal X^{2,\alpha}_\nu,
\end{equation}
where the space $\mathcal X^{2,\alpha}_\nu$ is centered at $g_\infty$.
\end{proposition}

\begin{proof}[Proof of Proposition~\ref{prop:coord-compactness}]
\emph{Step 1: bounds in the fixed charts.} By
\eqref{eq:area-volume} and \eqref{eq:local-noncollapse},
$c\leqslant\rho_i(o)\leqslant C$. By
Lemma~\ref{lem:coord-conformal}, $e^{-2C}\gamma_{\beta_i}\leqslant h_i\leqslant e^{2C}\gamma_{\beta_i}$.

(i) \emph{Interior.} Let $U$ be a closed interior chart or the central
chart. By \eqref{eq:coord-psi-distance} and Lemma~\ref{lem:coord-conformal},
$\psi_i\geqslant\delta(U)>0$ on $U$. Lemma~\ref{lem:scalar} then bounds
$|d\rho_i|_\gamma$, $|d\Phi_i|_\gamma$ and $d_{\mathbb H^2}(\Phi_i,\Phi_o)$
on $U$. On interior charts, with Lemma~\ref{lem:coord-conformal}, these are the hypotheses of the regularity improvement \eqref{eq:local-regularity} with $\gamma=\gamma_{\beta_i}$ and $u=a_i$, with constants depending only on $I$, $v$ and $U$. Since $c\leqslant\rho_i(o)\leqslant C$, \eqref{eq:local-regularity}, applied to $(a_i,\psi_i,\Phi_i)$ with $\rho_i=\rho_i(o)\psi_i$, bounds $(\rho_i,\Phi_i,a_i)$ in every $C^k$ on interior charts.

(ii) \emph{Center.} On the central chart, (i) gives (H1) and (H2) of Lemma~\ref{lem:center-expansion} with a uniform constant $M$. Here we use
$|d\rho_i|_\gamma+|d\Phi_i|_\gamma=e^{a_i}(|d\rho_i|_h+|d\Phi_i|_h)$.
Lemma~\ref{lem:center-expansion}, with a fixed $\varepsilon$, therefore gives, uniformly in $i$
and in all $\mathrm b$ derivatives,
\begin{equation}\label{eq:coord-central-derivatives}
 a_i-a_i(o),\quad \rho_i-\rho_i(o),\quad
 \Pi_{\Phi_o}\Phi_i=O_{\mathrm b}(R^2),
\end{equation}
with the remainders of \eqref{eq:center-refined} uniformly bounded.

(iii) \emph{Rods and corners.}
We treat a corner $c$. By Lemma~\ref{lem:coord-conformal}, the fixed corner chart, extended by reflection, is a reflected chart \eqref{eq:reflected-chart} of fixed size, in which $h_i=e^{2\widehat a_i}(dx^2+dy^2)$ with $\widehat a_i$ even and $|\widehat a_i|\leqslant C$.

 The Gauss curvature of $h_i$ gives
\begin{equation}\label{eq:reflected-curvature-equation}
 \Delta_0\widehat a_i=-K_{h_i}e^{2\widehat a_i}.
\end{equation}
By \eqref{eq:uniform-geometry} and \eqref{eq:reflected-facts}, $|\nabla_h^jK_{h_i}|\leqslant C_j$, and successive $W^{2,p}$ and Schauder estimates \cite[Theorems~9.11 and~6.2]{GT01} for \eqref{eq:reflected-curvature-equation} give uniform $C^{j,\alpha}$ bounds for $\widehat a_i$.

 Let $X_v$ and $X_{v'}$ be the Killing fields of the two rod vectors $v$ and $v'$ at $c$, which vanish on $\{x=0\}$ and on $\{y=0\}$ respectively, and let $G$ be the Gram matrix in this basis.  These Killing fields are bounded by Lemma~\ref{lem:scalar}(c), hence in every $C^j$ on balls of fixed radius, since $\nabla^*\nabla X=\Ric(X)=3X$ for each of them, and interior elliptic estimates \cite[Theorems~9.11 and~6.2]{GT01} apply on the balls given by \eqref{eq:uniform-geometry}. By \eqref{eq:reflected-facts} and the bounds for $\widehat a_i$, the entries $G_{ab}$ are bounded in every $C^j$ in the chart $(x,y)$. They are even in each variable, because the reflections of $S$ across the lifts of the two rods at $c$ are the half-turns $\boldsymbol\phi\mapsto\boldsymbol\phi+\pi v$ and $\boldsymbol\phi\mapsto\boldsymbol\phi+\pi v'$, which fix $X_v$ and $X_{v'}$. Since $X_v=0$ on $\{x=0\}$ and $X_{v'}=0$ on $\{y=0\}$, the functions $G_{11},G_{12}$ vanish with their first $x$-derivatives on $\{x=0\}$, and $G_{22},G_{12}$ with their first $y$-derivatives on $\{y=0\}$. Lemma~\ref{lem:division-lifts}(a) with $(k,l)=(2,0)$, $(0,2)$ and $(2,2)$, which loses two, two and four derivatives, gives
\[
 G_{11}=x^2H_{11},\qquad G_{22}=y^2H_{22},\qquad G_{12}=x^2y^2H_{12},\qquad\text{and}\qquad
 \widehat G_i=\begin{pmatrix}H_{11}&xyH_{12}\\ xyH_{12}&H_{22}\end{pmatrix},
\]
with $H_{ab}$ even and bounded in every $C^j$. This bounds $e^{2\widehat a_i}$ and $\widehat G_i$, which are the coefficients of $g_i$ along $S$ in the Cartesian chart \eqref{eq:cartesian}, by \eqref{eq:reflected-metric}. On an open-rod chart, the same Killing vector field estimates and division by the single normal variable bound $\bar G_i$ in every $C^j$.

  Since $\rho_i=x\sqrt{\det\bar G_i}$ for $x\geqslant0$, the same formula defines a smooth odd extension of $\rho_i$ across an open rod, used in Step~3. The extensions satisfy $\Delta_0\rho_i=-6e^{2\bar a_i}\rho_i$ in a reflected rod chart, by \eqref{eq:Einstein-identities}. They are bounded by $\rho_i(o)\leqslant C$, and the same argument bounds $e^{2\bar a_i}$ in every $C^j$, so interior elliptic estimates bound them in every $C^j$ on a smaller reflected chart, without  a uniform positive lower bound for $\bar G_i$.

\emph{Step 2: convergence of the coefficients.} By Step~1 and the Arzel\`a-Ascoli theorem, a diagonal subsequence of the coefficients converges in every $C^k$ on each closed chart away from $o$: $(\rho_i,\Phi_i,a_i)$ on interior charts, and $\widehat a_i$, $\widehat G_i$, respectively $\bar a_i$, $\bar G_i$ and the odd extensions of $\rho_i$, extended with their parities, on corner, respectively rod, charts. At $o$, $a_i(o)$ and $\rho_i(o)$ converge, and \eqref{eq:coord-central-derivatives} gives uniform convergence across $o$. In particular, $a_i$, $\rho_i$ and the entries of $G_i$ converge uniformly on $\overline\Omega$ in the prescribed charts.

\emph{Step 3: positivity of the limit.} The limits of the positive definite matrices $\widehat G_i$ and $\bar G_i$ are a priori only positive semidefinite, and at a rod $\det\bar G_i=(\rho_i/x)^2$ could tend to zero; Hopf's lemma, applied to the limit of the odd extensions of $\rho_i$, excludes this. On interior charts and on the central chart, $\psi_i\geqslant\delta(U)$, $\rho_i(o)\geqslant c$ (Step~1) and Lemma~\ref{lem:scalar}(c) bound $G_i$ from below by a positive multiple of $\Phi_o$, uniformly in $i$. At a corner $c$, the angle conditions of \eqref{eq:matrix-smooth-corner} and the oddness of $\widehat G_{12}$ give $\widehat G_\infty(c)=e^{2\widehat a_\infty(c)}I\geqslant e^{-2C}I$. At an open rod, the limit $\rho_\infty$ of the odd extensions is positive on $\{x>0\}$, by \eqref{eq:coord-psi-distance} and $\rho_i(o)\geqslant c$, and it vanishes on $\{x=0\}$. Since $\Delta_0\rho_\infty\leqslant0$ on $\{x>0\}$, Hopf's lemma \cite[Lemma~3.4]{GT01} gives $\partial_x\rho_\infty>0$ on $\{x=0\}$, so $\det\bar G_\infty=(\rho_\infty/x)^2>0$ up to $\{x=0\}$, with $\rho_\infty/x:=\partial_x\rho_\infty$ on $\{x=0\}$. As a limit of positive definite matrices, $\bar G_\infty$ is positive semidefinite, hence positive definite. So $\widehat G_\infty$ and $\bar G_\infty$ are uniformly positive definite on smaller closed corner and rod charts, and so are $\widehat G_i$ and $\bar G_i$ for large $i$, by uniform convergence.

\emph{Step 4: the limit metric and convergence in $\mathcal X$.} Uniform positivity, the parity and angle conditions \eqref{eq:matrix-smooth-side} and \eqref{eq:matrix-smooth-corner}, the finite symmetries and $\Ric(g_i)=3g_i$ pass to the limit. So $g_\infty$ is Einstein on $\pi^{-1}(\Omega\setminus\{o\})$, smooth across the rods and corners and invariant under the torus and $\mathcal G$, and its cone angle along $\Sigma$ is $2\pi\beta_*$, since $h_\infty=e^{2a_\infty}\gamma_{\beta_*}$ with $a_\infty$ continuous at $o$. By \eqref{eq:coord-central-derivatives}, it satisfies \eqref{eq:background-central}, hence the assumptions made on $g_*$ in Section~\ref{sec:spaces}. Moreover, by Step~3, the regularized coordinates \eqref{eq:gram-regular-coordinates} of $\widehat G_i$ and $\bar G_i$ converge in every $C^k$, and the inverse map in  Lemma~\ref{lem:regularized-chart} converts this convergence into convergence of the components of $q_i$ in the norms \eqref{eq:side-domain-norm} and \eqref{eq:corner-domain}. On interior charts, it is immediate. At $o$, $f_i(o)$ and $b_i(o)$ tend to $0$, and the remainders $f_i-f_i(o)$, $b_i-b_i(o)$ and $\xi_i^o$ are bounded in $R^2C^{3,\alpha}_{\mathrm b}$ by \eqref{eq:coord-central-derivatives} and tend to zero away from $o$. As $\nu<2$, the embedding $R^2C^{3,\alpha}_{\mathrm b}\hookrightarrow R^\nu C^{2,\alpha}_{\mathrm b}$ is compact \cite[Lemma~2.2]{PR00}, which proves \eqref{eq:coord-deformation-convergence}.
\end{proof}

We can now prove the existence of our Einstein metrics with prescribed cone angle along $\Sigma$.

\emph{Proof of Theorem~\ref{thm:angle}.} Let $(\beta_0,1]$, with $\beta_0\in[0,1)$, be the maximal interval to which the branch of Corollary~\ref{cor:local} at $g_1$ extends as a curve of solutions. By the local uniqueness in Corollary~\ref{cor:local}, this extension is unique. Suppose, towards a contradiction, that $\beta_0>0$, fix $\beta_I\in(\beta_0,1)$ and let $\beta_i\downarrow\beta_0$. By \eqref{eq:volume-lower}, $\Vol(g_{\beta_i})\geqslant V_1\beta_0^5$, so Proposition~\ref{prop:coord-compactness}, with $I:=[\beta_0,\beta_I]$, gives an Einstein metric $g_\infty$ with cone angle $2\pi\beta_0$ along $\Sigma$, smooth across the rods and corners and satisfying the assumptions made on $g_*$ in Section~\ref{sec:spaces}, such that $q_i\to0$ in the space $\mathcal X$ centered at $g_\infty$. Corollary~\ref{cor:local} at $g_*:=g_\infty$ gives a curve of solutions on an open interval containing $\beta_0$. Since $q_i\to0$ in $\mathcal X$, local uniqueness shows that this curve passes through $g_{\beta_i}$ for large $i$, so it agrees with the branch above $\beta_0$ and extends it below $\beta_0$, a contradiction. Hence $\beta_0=0$. The remaining assertions of Theorem~\ref{thm:angle} follow from Proposition~\ref{prop:recovery}, Corollary~\ref{cor:fixed-set-regularity} and Corollary~\ref{cor:local}.

\section{Proof of \texorpdfstring{Theorem~\ref{thm:main}}{Theorem 1.1} and first consequences}\label{sec:proof-main}

Let $(g_\beta)_{0<\beta\leqslant1}$ be the family of Theorem~\ref{thm:angle}, in either symmetry class, and let $k\geqslant2$ be an integer.

\subsection{Smooth covers and topology}\label{subsec:smooth-covers}
\paragraph{Changing the angular period.}

The half-turn $(\theta,\boldsymbol\phi)\mapsto(\theta+\pi,\boldsymbol\phi)$ is an isometry of $g_\beta$: it is $\mathsf r_x\mathsf r_y$ of \eqref{eq:group-actions} in the square case and $\mathsf r^3\mathsf c$, with $\mathsf r$ and $\mathsf c$ of \eqref{eq:hexagon-lifts}, in the hexagon case, since $\mathsf A^3=-I$. So the coefficients $h,G$ of $g_\beta$ are $\pi$-periodic in the polar angle $\theta$ about $o$, and they may be read with any angular period $(k+1)\pi$. They define an orbit space $\Omega^{(k)}$, with $n_k:=\frac{m}{2}(k+1)$ rods, and on the resulting manifold $M^{(k)}$, which is $M_k$, respectively $N_k$, as shown below, a metric $g^{(k)}_\beta$ given by \eqref{eq:matrix-ansatz}. 

The rotation $\mathsf r$ of Section~\ref{sec:common-coordinates}, with its lift, is an isometry of $g^{(k)}_\beta$. It induces on $\Omega^{(k)}$ the rotation $\theta\mapsto\theta+\frac{2\pi}m$ by one rod, of order $n_k$. The metric $g^{(k)}_\beta$ is Einstein away from $\Sigma$, with cone angle $(k+1)\pi\beta$ along $\Sigma$. We choose $\beta_k:=\frac{2}{k+1}$, for which this angle is $2\pi$.

\paragraph{Smoothness.} Near the rods and corners, $g^{(k)}:=g^{(k)}_{\beta_k}$ coincides with $g_{\beta_k}$, so it is smooth there by Corollary~\ref{cor:fixed-set-regularity}. It is smooth across $\Sigma$ by the following statement.
\begin{corollary}\label{cor:smooth-cover}
Let $\mathcal E(\beta,q)=0$ with $\beta=\frac{2}{k+1}$ for
an integer $k\geqslant2$. Then $g^{(k)}$ extends smoothly across $\Sigma$.
The case $k=1$ is the prescribed smooth initial metric.
\end{corollary}
\begin{proof}
With $\Theta:=\beta\theta$, the normal angle has period $2\pi$.
By \eqref{eq:central-domain},
$\rho$, $a$ and $U$ are constants plus functions in
$R^\nu C^{2,\alpha}_{\mathrm b}$, hence $C^{1,\delta}$ functions of the
Cartesian coordinates $(R\cos\Theta,R\sin\Theta)$ for
$0<\delta<\min\{\alpha,\nu-1\}$. In these coordinates
$\gamma_\beta=\tfrac13|d(Re^{i\Theta})|^2$.
As $\rho$, $a$ and $U$ are in $H^1$ and the right-hand sides of \eqref{eq:central-matrix-system} are bounded, Lemma~\ref{lem:B-cutoffs} extends \eqref{eq:central-matrix-system} weakly across the origin. Interior elliptic regularity \cite[Chapters~6 and~9]{GT01} then makes $\rho$, $a$, $U$, and hence $g^{(k)}$, smooth.
\end{proof}

\paragraph{Topology.} Up to sign, the rod vectors of $\Omega^{(k)}$ alternate between $(1,0)$ and $(0,1)$ in the square case, and repeat the cycle $(1,0),(1,1),(0,1)$ in the hexagon case. These are the classical rod configurations of $M_k=\#_k({\mathbb S^2}\times {\mathbb S^2})$ and of $N_k=k\mathbb{CP}^2\#(2k+1)\overline{\mathbb{CP}}{}^2$ \cite{OR70}. This proves Theorem~\ref{thm:main}.

\begin{corollary}\label{cor:volumes}
$V_k:=\Vol(g^{(k)})\to0$ as $k\to+\infty$.
\end{corollary}
\begin{proof}
By the Bonnet-Myers theorem, $\operatorname{diam}(g^{(k)})\leqslant\pi$. If $V_k\geqslant v>0$ along a subsequence, Bishop-Gromov bounds the volumes of unit balls from below, and the finiteness theorem \cite[Theorem~1.12]{CN15} leaves finitely many diffeomorphism types, whereas $\chi(M_k)=2(k+1)$ and $\chi(N_k)=3(k+1)$ are unbounded.
\end{proof}

\subsection{The square family}

\begin{proposition}
   \label{prop:square-instability}
Every smooth Einstein metric from Theorem~\ref{thm:main} on $M_k=\#_k({\mathbb S^2}\times {\mathbb S^2})$ is linearly unstable.
 \end{proposition}

\begin{proof}[Proof of Proposition~\ref{prop:square-instability}]
Write $\Ric(g)=\lambda g$, with $\lambda>0$.
An orientation-reversing isometry is given by $\mathsf c_1$ of \eqref{eq:group-actions} which extends across the rods and corners. Since $b_2^+(M_k)=k\geqslant 1$, we can choose a nonzero smooth harmonic
self-dual two-form $\omega_+$ on $M_k$.
The torus acts trivially on cohomology, so $\omega_+$ is torus-invariant and $|\omega_+|_g$ is constant on each torus orbit, by uniqueness of the harmonic representative.
The form $\omega_-:=\mathsf c_1^*\omega_+$ is harmonic and anti-self-dual,
and, since $\mathsf c_1$ preserves every torus orbit, we have
\[
 |\omega_-|_g^2=|\omega_+|_g^2.
\]

Define $H_{ij}:=(\omega_+)_i{}^p(\omega_-)_{jp}$ and $(\mathring{\Rm}_gH)_{ij}:=R_{ikjl}H^{kl}$. By \cite[Proposition~2.1, equation~(12)]{Ozu26}, $H$ is trace-free and divergence-free, $|H|_g^2=|\omega_+|_g^2|\omega_-|_g^2=|\omega_+|_g^4$, and, since $\Ric_g=\lambda g$, we have
\[
 \int_{M_k}\bigl\langle\nabla^*\nabla H-2\mathring{\Rm}_gH,H\bigr\rangle_g\,dV_g
 =-2\lambda\int_{M_k}|\omega_+|_g^4\,dV_g-\frac{1}{2}\int_{M_k}\bigl|d(|\omega_+|_g^2)\bigr|_g^2\,dV_g<0.
\]

Thus the smooth tensor $H$, which is trace-free and divergence-free, is a destabilizing direction.
\end{proof}

We call a Riemannian metric \emph{static}
with respect to a circle generator when it has the local form
$N^2d\phi_1^2+g_B$, with $N$ and $g_B$ independent of $\phi_1$. It is moreover \textit{axisymmetric} when it has the form \eqref{eq:ansatz}.
Replacing $d\phi_1^2$ by $-dt^2$ gives a Lorentzian static axisymmetric metric.

\begin{proposition}
    \label{prop:wick}
Wick rotation of the metrics on $M_k=\#_k({\mathbb S^2}\times {\mathbb S^2})$ of Theorem~\ref{thm:main} gives static axisymmetric Lorentzian Einstein metrics with cosmological constant $3$ and $k+1$ horizons, all with the same surface gravity. Equivalently, their Hawking temperatures agree \cite{GH77}.
\end{proposition}
\begin{proof}
Write $g^{(k)}=N^2d\phi_1^2+g_B$, with $N:=u$ and $g_B:=h+v^2d\phi_2^2$, by \eqref{eq:ansatz}. For such a static metric, $\Ric=3g$ is equivalent to $\Ric(g_B)=N^{-1}\Hess_{g_B}N+3g_B$ and $\Delta_{g_B}N=-3N$, which are also the equations $\Ric=3\mathbf g$ for $\mathbf g:=-N^2dt^2+g_B$. The horizons of $\mathbf g$ are the components of $\{N=0\}$, which lie over the $k+1$ rods of $\Omega^{(k)}$ with rod vector $(1,0)$, and the surface gravity of $\partial_t$ there is $|dN|_{g_B}$. By \eqref{eq:rod-angle-ratio} with angle $2\pi$, $N=\sqrt{G_{11}}$ grows like the distance to these rods, so $|dN|_{g_B}=1$ on all of them.
\end{proof}

\subsection{Free cyclic quotients}\label{subsec:free-quotients}

By \eqref{eq:CGB}, $\int_{M^{(k)}}|\Rm|^2=32\pi^2\chi(M^{(k)})=16\pi^2m(k+1)$, which grows linearly in $k$. Before blowing up, we therefore pass to free cyclic quotients with fixed curvature energy.

In the polar coordinates $(R,\theta)$ about $o$ and the torus angles $\boldsymbol\phi$, put 
\[
 \mathsf r_k(\theta,\boldsymbol\phi):=(\theta+\pi,\boldsymbol\phi)\qquad\text{and}\qquad
 \mathsf t_k(\theta,\boldsymbol\phi):=\bigl(\theta,\boldsymbol\phi+\tfrac{2\pi}{k+1} e_1\bigr).
\]
By the lift of the half-turn above, $\mathsf r_k$ is an isometry of $g^{(k)}$ of order $k+1$, and so is the torus translation $\mathsf t_k$. They commute, so $\varpi_k:=\mathsf t_k\circ\mathsf r_k$ is an isometry of order $k+1$, and we put
\[
 Z^{(k)}:=M^{(k)}/\langle\varpi_k\rangle,
\]
the quotient by the group generated by $\varpi_k$,

with the quotient metric $g_Z^{(k)}$. For $0<j<k+1$, $\mathsf r_k^j$ fixes only the points over $o$, where it acts trivially, whereas $\mathsf t_k^j$ is a nontrivial translation of $\Sigma$, hence  $\varpi_k^j$ has no fixed point. Both maps preserve the orientation. They commute with the torus action, which descends to $Z^{(k)}$. Away from the image $\Sigma_Z$ of $\Sigma$ in $Z^{(k)}$, the stabilizers are those of $M^{(k)}$, while a point of $\Sigma_Z$ has stabilizer $\langle\mathsf t_k\rangle\approx\mathbb Z_{k+1}$, since $\mathsf r_k$ fixes $\Sigma$ pointwise.  Thus $\Sigma_Z$ is the  unique orbit with \emph{nontrivial finite stabilizer}. The orbit space $\Omega_Z^{(k)}:=Z^{(k)}/\mathbb T^2$, the quotient of $\Omega^{(k)}$ by the rotation $\theta\mapsto\theta+\pi$ induced by  $\varpi_k$, has $\frac{m}{2}$ rods and $\frac{m}{2}$ corners.  

\begin{corollary}
\label{cor:finite-fundamental-groups}
For $k\geqslant1$, the quotients $(Z^{(k)},g_Z^{(k)})$ are closed
oriented Einstein four-manifolds satisfying
\[
 \Ric(g_Z^{(k)})=3g_Z^{(k)},\qquad
 \pi_1(Z^{(k)})=\mathbb Z_{k+1},\qquad
 \chi(Z^{(k)})=\tfrac m2,\qquad\text{and}\qquad \tau(Z^{(k)})=2-\tfrac m2.
\]
Their integral homology groups are
\[
 H_1(Z^{(k)};\mathbb Z)\approx\mathbb Z_{k+1},\qquad
 H_2(Z^{(k)};\mathbb Z)\approx\mathbb Z^{m/2-2}\oplus\mathbb Z_{k+1},\qquad\text{and}\qquad
 H_3(Z^{(k)};\mathbb Z)=0.
\]
Moreover, their curvature energy, volume and diameter satisfy
\[
 \begin{gathered}
 \int_{Z^{(k)}}|\Rm(g_Z^{(k)})|^2\,dV_{g_Z^{(k)}}=16\pi^2m,
 \\
 \Vol(Z^{(k)})=\frac{V_k}{k+1}=o(k^{-1}),
 \qquad\text{and}\qquad
 \operatorname{diam}(Z^{(k)},g_Z^{(k)})\leqslant\pi.
 \end{gathered}
\]
In particular, the square quotients are rational homology four-spheres with the minimum curvature energy $64\pi^2$ for closed oriented positive Einstein four-manifolds, and the hexagon quotients have
the rational homology of $\overline{\mathbb{CP}}{}^2$ and curvature
energy $96\pi^2$.
\end{corollary}
\begin{proof}
The free action above has order $k+1$, and $M^{(k)}$ is simply connected,
so $M^{(k)}\to Z^{(k)}$ is the universal cover and
$\pi_1(Z^{(k)})=\mathbb Z_{k+1}$.
The Einstein equation descends to the quotient. Multiplicativity of
Euler characteristic and signature gives
\[
 \chi(Z^{(k)})=\frac{\frac{m}2(k+1)}{k+1}=\frac{m}{2},\qquad\text{and}\qquad
 \tau(Z^{(k)})=\frac{(k+1)(2-\frac{m}2)}{k+1}=2-\frac{m}{2}.
\]
Since $H_1(Z^{(k)};\mathbb Z)=\mathbb Z_{k+1}$, Poincar\'e duality gives
$H_3(Z^{(k)};\mathbb Z)=H^1(Z^{(k)};\mathbb Z)=0$ and hence
$b_2(Z^{(k)})=\frac{m}{2}-2$. Poincar\'e duality and the universal coefficient
theorem identify the torsion subgroup of $H_2$ with
$\operatorname{Ext}(H_1,\mathbb Z)\approx\mathbb Z_{k+1}$, proving the
homology formulae.

As $g_Z^{(k)}$ is Einstein, \eqref{eq:CGB} and $\chi(Z^{(k)})=\frac{m}{2}$ give
\[
 \int_{Z^{(k)}}|\Rm(g_Z^{(k)})|^2\,dV_{g_Z^{(k)}}=32\pi^2\chi(Z^{(k)})=16\pi^2m.
\] The volume is divided by the covering
degree $k+1$, and $V_k\to0$ by Corollary~\ref{cor:volumes}. The diameter bound $\operatorname{diam}\leqslant\pi$ is the Bonnet-Myers theorem. Finally, any closed connected oriented positive
Einstein four-manifold $Y$ has $b_1(Y)=0$ by Bochner's formula, so \eqref{eq:CGB} gives
\[
 \int_Y|\Rm|^2\,dV=32\pi^2\chi(Y)
 =32\pi^2\bigl(2+b_2(Y)\bigr)\geqslant64\pi^2.
\]
Thus the square quotients attain the minimum. Orientability is needed since the round $\mathbb{RP}^4$ has energy $32\pi^2$.
\end{proof}

\section{Collapse and Ricci-flat limits}\label{sec:collapse}

This section studies two limits as $k\to+\infty$. The compact manifolds $(M^{(k)},g^{(k)})$ collapse to an explicit disk (Theorem~\ref{thm:collapsed-limit}), and curvature rescalings of their quotients $(Z^{(k)},g_Z^{(k)})$, whose curvature energy is fixed, converge to complete Ricci-flat metrics (Theorem~\ref{thm:periodic-models}).

Let $g^{(k)}$ be the smooth metric on $M^{(k)}$ of Section~\ref{sec:proof-main}, obtained from $g_{\beta_k}$ with $\beta_k=\frac{2}{k+1}$; it is in class (C2) of Section~\ref{sec:compactness}. Its orbit space $\Omega^{(k)}$ is a disk with $\frac{m}{2}(k+1)$ rods and as many corners about a smooth central point $o$.  We write $\psi_k:=\rho_k/\rho_k(o)$, and  we omit the index $k$ of $\rho$, $\psi$ and of the cores $E_\delta$  when no confusion arises.

\subsection{The collapsed disk}
The limits of this subsection are collapsed limits of manifolds with a uniform
Ricci lower bound, as in the general framework of Cheeger-Colding
\cite{CC97}. Here the reduced equations \eqref{eq:Einstein-identities} identify the limit explicitly.

Let $F(r):=\frac{1}r-r^2$, and
\begin{equation}\label{eq:kottler-orbit}
 h_\infty:=\frac{dr^2}{F(r)}+\frac{4}{9}F(r)\,d\vartheta^2,\qquad \text{for }0<r\leqslant1\text{ and }\vartheta\in\mathbb R/2\pi\mathbb Z.
\end{equation}
 Since $F(1)=0$ and $F'(1)=-3$, the point $r=1$ is a smooth central point. The distance to the center is $\frac{2}{3}\arccos r^{3/2}$, so the incomplete end $r\to0$ lies at distance $\frac{\pi}{3}$.

\begin{theorem}\label{thm:collapsed-limit}
For either family, as $k\to+\infty$, $\rho_k(o)\to0$, and the reduced data $(h,\psi,\Phi)$ converge to $(h_\infty,r^2,\Phi_o)$ smoothly on each core $E_\delta$, $0<\delta<1$, in the charts of Lemma~ \ref{lem:smooth-cores}. Consequently, $(M^{(k)},g^{(k)})$ converges in the Gromov-Hausdorff sense to the metric completion $B_\infty$ of $h_\infty$, and $\operatorname{diam}(M^{(k)},g^{(k)})\to\frac{2\pi}{3}$.
\end{theorem}

\begin{proof}[Proof of Theorem~\ref{thm:collapsed-limit}]
\emph{Step 1: compactness and limit of the  torus shape.}
By Lemma~ \ref{lem:smooth-cores}, the coefficients of $(h,\psi,\Phi)$ in the
normal chart at $o$ and in the cylinder are bounded in every $C^j$. A diagonal
subsequence therefore converges smoothly on every $E_\delta$.

The  lift of $\mathsf r$ acts on the torus angles by $\mathsf B$ of \eqref{eq:hexagon-symmetries} in the square case, where it exchanges the two circles, and by $\mathsf A$ in the hexagon case, by \eqref{eq:hexagon-lifts}. So the invariance of $g$ gives
\begin{equation}\label{eq:iota-equivariance}
 \Phi\circ\mathsf r=\iota\circ\Phi,\qquad\text{where}\qquad\iota(\Phi'):=\mathsf B\Phi'\mathsf B,\quad\text{respectively}\quad\iota(\Phi'):=\mathsf A^{-\mathsf T}\Phi'\mathsf A^{-1},
\end{equation}
and $\iota$ is an isometry of $\mathbb H^2$ fixing $\Phi_o$. In the square case, $\Phi_k$ lies on the diagonal geodesic, which $\iota$ reverses, with unique fixed point $\Phi_o$ on that geodesic. In the hexagon case, $\iota$ is a rotation of order three, with unique fixed point $\Phi_o$. The rotation $\mathsf r$  fixes $o$ and $\psi$, so it commutes with $\exp_o$ and with the flow of $V_\psi$, and it acts in both charts by $\Theta\mapsto\Theta+\frac{4\pi}{m(k+1)}$, which tends to the identity. Passing \eqref{eq:iota-equivariance} to the limit gives $\Phi_\infty=\iota\circ\Phi_\infty$, with $\Phi_\infty$ on the diagonal geodesic in the square case, so $\Phi_\infty\equiv\Phi_o$.

\emph{Step 2: identification of the limit.} Put $r:=\sqrt{\psi_\infty}$. Since $\Phi_\infty=\Phi_o$, \eqref{eq:concavity} and \eqref{eq:Einstein-identities} give
 \[
 \Hess_{h_\infty}r=-\frac{|dr|_{h_\infty}^2+3r^2}{2r}\,h_\infty,\qquad\text{hence}\qquad d\bigl(r|dr|_{h_\infty}^2+r^3\bigr)=0.
\]
 Since $r(o)=1$ and $dr(o)=0$, we get $|dr|_{h_\infty}^2=\frac{1}r-r^2=F(r)$. As the Hessian of $r$  is pure trace, $J\nabla r$ is a Killing field, where $J$ is the rotation by a right angle, and  $-\frac{2}{3}J\nabla r$ has period $2\pi$, since $\Hess_{h_\infty}r(o)=-\frac{3}{2}h_\infty(o)$. With the corresponding angle $\vartheta$, $h_\infty$ is \eqref{eq:kottler-orbit}, independently of the subsequence.

\emph{Step 3: convergence of the orbit spaces.} By \eqref{eq:core-retraction}, the distance of $E_\delta$ is the restriction of the distance of $\overline\Omega^{(k)}$, so, by Step~2, the cores $E_\delta$ converge as metric spaces to $\{r^2\geqslant\delta\}\subset B_\infty$. By \eqref{eq:base-radius}, $E_\delta$ is $\arcsin\sqrt\delta$-dense in $\overline\Omega^{(k)}$, and $\{r^2\geqslant\delta\}$ has the same property in $B_\infty$ by \eqref{eq:kottler-orbit}, since $F(r)\geqslant1-r^2$.
Letting $k\to+\infty$ and then $\delta\to0$, the orbit spaces $(\overline\Omega^{(k)},d_h)$ converge to $B_\infty$. 

\emph{Step 4: collapse of the torus orbits.} The torus orbits collapse:  \eqref{eq:area-volume} for $g_{\beta_k}$, whose central torus is that of $g^{(k)}$ and whose volume is $\beta_kV_k$, Corollary~\ref{cor:volumes} and Lemma~\ref{lem:scalar}(c) give
 \[
 \rho_k(o)=\frac{\Area(\Sigma)}{4\pi^2}\leqslant CV_k\longrightarrow0,\qquad\text{and}\qquad
 \sup_{q\in\Omega^{(k)}}\operatorname{diam}\pi^{-1}(q)\leqslant2\operatorname{diam}(\Sigma)\leqslant C\sqrt{\rho_k(o)}\longrightarrow0,
\]
where the last bound uses that the central shape $\Phi_o$ is fixed. Since $\pi$ is a submetry onto $(\overline\Omega^{(k)},d_h)$, with fibers of diameter tending to zero, $(M^{(k)},g^{(k)})$ converges to $B_\infty$ in the Gromov-Hausdorff sense.

Finally, the metric \eqref{eq:kottler-orbit} has curvature $K_{h_\infty}=-\frac{1}{2}F''(r)=1-r^{-3}\leqslant0$, so,  with $\tau:=d_{h_\infty}(o,\cdot)$, the Rauch comparison theorem gives $h_\infty\geqslant d\tau^2+\tau^2d\vartheta^2$. Two boundary points with opposite angles $\vartheta$ are therefore at $h_\infty$-distance at least $\frac{2\pi}{3}$, the flat distance, and every point lies within $\frac{\pi}{3}$ of the center, so $\operatorname{diam}(B_\infty)=\frac{2\pi}{3}$. The dimension of $\partial B_\infty$ is computed in Remark~\ref{rem:four-thirds}.
\end{proof}

 \begin{remark}\label{rem:kottler}
The metric $h_\infty$ of \eqref{eq:kottler-orbit} is the orbit metric of $g_{\mathrm{Kot},\varepsilon}:=h_\infty+\varepsilon^2r^2(\Phi_o)_{ij}\,d\phi_i\,d\phi_j$, the Riemannian Kottler metric with flat toroidal horizon, cosmological constant $3$ and \textit{negative} mass $\mathfrak m=-\frac{1}{2}$, so that $F=-2\mathfrak m/r-r^2$ \cite{Kot18}, which satisfies $\Ric(g_{\mathrm{Kot},\varepsilon})=3g_{\mathrm{Kot},\varepsilon}$ for every $\varepsilon>0$. The parameter $\varepsilon$ is the size of the flat torus orbits. Its end $r\to0$ is asymptotic to a Kasner metric with exponents $(-\frac{1}{3},\frac{2}{3},\frac{2}{3})$, with the exponent $-\frac{1}{3}$ belonging to the $\vartheta$-circle.
\end{remark}

\begin{remark}\label{rem:four-thirds}
Let $t:=\frac{\pi}{3}-\frac{2}{3}\arccos r^{3/2}$ be the distance to $\partial B_\infty$. Then we have $r^{3/2}=\sin\frac{3t}{2}$, so $\frac{4}{9}F(r)=\frac{4}{9}\cos^2\bigl(\frac{3t}{2}\bigr)\sin^{-2/3}\bigl(\frac{3t}{2}\bigr)$, and, with the $2\pi$-periodic coordinate $\vartheta$ of \eqref{eq:kottler-orbit}, the metric reads
\[
 h_\infty=dt^2+\Bigl(\frac{2}{3}\Bigr)^{8/3}t^{-2/3}\bigl(1+O(t^2)\bigr)\,d\vartheta^2\qquad \text{ as }t\to 0.
\]
This is a Grushin-type metric: in the notation $dr^2+r^{-4\alpha}dv^2$ of \cite[Section~3.1]{DHPW23}, $\alpha=\frac{1}{6}$. By \cite[Theorem~3.5]{DHPW23}, $\partial B_\infty$ is of Hausdorff dimension $\frac{4}{3}$, while $B_\infty$ itself has dimension $2$. Collapsed Ricci limits of non-integer dimension also appear in \cite{PW22,Pan23}.
\end{remark}

\begin{remark}\label{rem:segment-limit}
The generator  $\varpi_k$ of Section~\ref{sec:proof-main} acts on $\Omega^{(k)}$ by $\theta\mapsto\theta+\pi$, that is, by $\vartheta\mapsto\vartheta+\frac{2\pi}{k+1}$ in the normal angle of the limit in Theorem~\ref{thm:collapsed-limit}. Since the rotations of $\vartheta$ are isometries of $h_\infty$, equivariant Gromov-Hausdorff convergence shows that the quotients $(Z^{(k)},g_Z^{(k)})$ of Corollary~\ref{cor:finite-fundamental-groups} converge to $B_\infty/{\mathbb S^1}=[0,\frac{\pi}{3}]$, parametrized by the distance to the center. By Lemma~ \ref{lem:smooth-cores}, $|\Rm(g_Z^{(k)})|$ stays bounded over each core, while $\max|\Rm(g_Z^{(k)})|\to+\infty$ by Lemma~\ref{lem:curvature-localization}(i), so the curvature diverges only at the end $\frac{\pi}{3}$.
\end{remark}

\subsection{The Ricci-flat models}\label{subsec:models}

We now blow up, at the scale of curvature, the quotients $(Z^{(k)},g_Z^{(k)})$ of Corollary~\ref{cor:finite-fundamental-groups}, whose curvature energy  does not depend on $k$. Under $g\mapsto\Lambda g$,
we have $h\mapsto\Lambda h$, $G\mapsto\Lambda G$,
$\rho\mapsto\Lambda\rho$ and $\Phi\mapsto\Phi$, while the flows of $X_1,X_2$ keep the period $2\pi$.

The rotation $\mathsf r$ by one rod induces on $\Omega_Z^{(k)}$ an orientation-preserving isometry $\mathsf s_k$ of $h$, of order $\frac{m}{2}$, which preserves $\rho$ and maps each corner to an adjacent one. We call the totally geodesic spheres over the rods, fixed by their circle stabilizers, the \emph{fixed spheres}; they are isometric since $\mathsf r$ is transitive on the rods. The quotient maps each fixed sphere of $M^{(k)}$ isometrically onto a fixed sphere of $Z^{(k)}$.

  For the isometry $\iota$ of \eqref{eq:iota-equivariance}, we have
\begin{equation}\label{eq:cyclic-displacement}
 d_{\mathbb H^2}(\Phi',\Phi_o)\leqslant d_{\mathbb H^2}(\Phi',\iota\Phi'),
\end{equation}
for $\Phi'$ on the diagonal geodesic in the square case, and for every $\Phi'\in\mathbb H^2$ in the hexagon case. Indeed, in the square case $\iota$ reverses this geodesic at $\Phi_o$, so $d_{\mathbb H^2}(\Phi',\iota\Phi')=2d_{\mathbb H^2}(\Phi',\Phi_o)$. In the hexagon case, write $\Phi'=\operatorname{Exp}_{\Phi_o}V$. The differential of $\iota$ at $\Phi_o$ is the rotation by $\frac{2\pi}{3}$, and $\operatorname{Exp}_{\Phi_o}$ does not decrease distances, $\mathbb H^2$ being nonpositively curved, so $d_{\mathbb H^2}(\Phi',\iota\Phi')\geqslant|V-d\iota_{\Phi_o}V|=\sqrt3|V|\geqslant d_{\mathbb H^2}(\Phi',\Phi_o)$.

 To the best of our knowledge, the hexagonal limit $(X_{\mathrm{hex}},g_{\mathrm{hex}})$ is a new gravitational instanton.
\begin{theorem}\label{thm:periodic-models}

Put $\Lambda_k:=\max|\Rm(g_Z^{(k)})|$, and let $c_k$ be the corner of $Z^{(k)}$ at the end of a rod with rod vector $\pm e_1$, for the boundary orientation of $\Omega_Z^{(k)}$. Then $\Lambda_k\to+\infty$, and the whole sequence satisfies
\[
 (Z^{(k)},\Lambda_kg_Z^{(k)},c_k)\longrightarrow(X,g_X,c)
\]
smoothly on compact subsets. The limit $(X,g_X)$ has the following properties.
\begin{enumerate}[label=\textup{(\alph*)}]
\item \emph{Curvature.} It is complete, nonflat and Ricci-flat, with $|\Rm|\leqslant1$ and $\int_X|\Rm|^2\,dV=32\pi^2\chi(X)=16\pi^2m$, the curvature energy of every $Z^{(k)}$.
\item \emph{Torus action.} It has an effective isometric action of $\mathbb R^2/2\pi\mathbb Z^2$ generated by the limits of $X_1,X_2$.  Its orbit space is  a closed half-cylinder, whose boundary circle  consists of $\frac{m}{2}$ rods, with rod vectors $(e_1,e_2)$ in the square case and $(e_1,e_1+e_2,e_2)$ in the hexagon case, in cyclic order, permuted cyclically by an isometry of $X$.
\item \emph{Kasner end.}  Its end is asymptotic to  a Kasner metric with exponents $(-\frac{1}{3},\frac{2}{3},\frac{2}{3})$.
\end{enumerate}
We write $X_{\mathrm{sq}}$ and $X_{\mathrm{hex}}$ for $X$, and $g_{\mathrm{sq}}$ and $g_{\mathrm{hex}}$ for $g_X$, in the two cases.
\begin{enumerate}[label=\textup{(\roman*)}]
\item $X_{\mathrm{sq}}\approx {\mathbb S^4}\setminus {\mathbb T^2}$, and $g_{\mathrm{sq}}$ is a
one-period quotient of an equal-rod Euclidean
Myers-Korotkin-Nicolai metric \cite{My87,KN94,KRWY22}.
\item $X_{\mathrm{hex}}\approx\overline{\mathbb{CP}}{}^2\setminus {\mathbb T^2}$,
$\pi_1(X_{\mathrm{hex}})=\mathbb Z$, and $X_{\mathrm{hex}}$ contains compact fixed spheres of self-intersection $-1$. In particular, it is not spin.
The metric $g_{\mathrm{hex}}$ is not locally Hermitian in either orientation, and it is not static for any choice of generators of the torus action.
\end{enumerate}
In both cases the removed torus is a principal orbit of the standard
torus action.
\end{theorem}

At the same scale, the blow-ups of the original manifolds $M^{(k)}$ converge to the infinite cyclic cover of $X$ (Remark~\ref{rem:covers}). Curvature concentrates only near the corners (Lemma~\ref{lem:curvature-localization}). Curvature-scale compactness retains at least one corner, but all corners must remain within bounded distance at the curvature scale. The region between two suitable level sets of $\rho$ lies within a controlled distance of all corners and is noncollapsed at its own scale $L_k$ (Lemma~\ref{lem:corner-noncollapse}). Corollary~\ref{cor:ale-exclusion} then shows that $L_k$ is comparable to the curvature scale, which gives compactness at the corners (Proposition~\ref{prop:corner-compactness}). The isometries $\mathsf s_k$ above give a periodic Weyl description with equal rods (Lemma~\ref{lem:weyl}), from which we obtain the Kasner asymptotics and identify the limits.

\subsubsection[Compactness at the corners]{Compactness at the corners}

Put $g_{Z,k}:=g_Z^{(k)}$, recall $\Lambda_k=\max|\Rm(g_{Z,k})|$ from Theorem~\ref{thm:periodic-models}, and let $\mathcal C_k$ be the set of corners of $Z^{(k)}$.

\begin{lemma}\label{lem:curvature-localization}
The following hold for the sequence $(g_{Z,k})$.
\begin{enumerate}[label=\textup{(\roman*)}]
\item We have $\Lambda_k\to+\infty$.
\item Let $y_k\in Z^{(k)}$ and $r_k>0$ satisfy $Q_k:=|\Rm(g_{Z,k})|(y_k)\to+\infty$, $Q_kr_k^2\to+\infty$ and $|\Rm(g_{Z,k})|\leqslant4Q_k$ on $B_{g_{Z,k}}(y_k,r_k)$. Then, after a subsequence, $(Z^{(k)},Q_kg_{Z,k},y_k)$ converges smoothly on compact subsets to a complete, nonflat and Ricci-flat limit with an effective isometric action of $\mathbb R^2/2\pi\mathbb Z^2$, which has the original generators, rod vectors and rod and corner models, and no orbit with nontrivial finite stabilizer. Moreover, some corners $c_k\in\mathcal C_k$ satisfy $\sup_kQ_k^{1/2}d_{g_{Z,k}}(y_k,c_k)<+\infty$.
\item There is $C$ such that, for every $x\in Z^{(k)}$,
\begin{equation}\label{eq:curvature-away-corners}
 |\Rm(g_{Z,k})|(x)\,d_{g_{Z,k}}(x,\mathcal C_k)^2\leqslant C.
\end{equation}
\end{enumerate}
\end{lemma}
\begin{proof}
(i) Since $\int_{Z^{(k)}}|\Rm|^2=16\pi^2m$ and $\Vol(Z^{(k)})=V_k/(k+1)\to0$ by Corollary~\ref{cor:volumes}, we have $\Lambda_k^2\geqslant16\pi^2m/\Vol(Z^{(k)})\to+\infty$.

(ii) Put $g_k':=Q_kg_{Z,k}$. The small-energy curvature estimate \cite[Theorem~0.8, p.~489]{CT06}, which assumes no lower volume bound, gives $\varepsilon_0,C_0>0$ such that every Einstein metric with $\Ric=\lambda g$, $|\lambda|\leqslant3$ and $\int_{B(p,1)}|\Rm|^2\leqslant\varepsilon_0$ satisfies $|\Rm|(p)\leqslant C_0$. Fix $A>\max\{1,\sqrt{C_0}\}$. For large $k$, the metric $A^{-2}g_k'$ has Einstein constant $3A^2Q_k^{-1}\leqslant3$ and curvature $A^2>C_0$ at $y_k$, while $|\Rm(g_k')|\leqslant4$ on $B_{g_k'}(y_k,A)$, since $Q_kr_k^2\to+\infty$. The scale invariance of the energy and the Bishop-Gromov inequality for $\Ric(g_k')\geqslant0$ therefore give
\[
 \varepsilon_0<\int_{B_{g_k'}(y_k,A)}|\Rm|^2\,dV\leqslant16\Vol_{g_k'}B(y_k,A)\leqslant16A^4\Vol_{g_k'}B(y_k,1).
\]
So Einstein regularity with \cite{And89} gives smooth subconvergence of $(Z^{(k)},g_k',y_k)$ to a complete Ricci-flat limit, with curvature energy at most $16\pi^2m$ by scale invariance, nonflat since $|\Rm(g_k')|(y_k)=1$. The points of $Z^{(k)}$ fixed by the torus are its corners (Section~\ref{subsec:free-quotients}), so Lemma~\ref{lem:torus-limit}(b), applied with $U_k:=Z^{(k)}$, gives corners $c_k$ with $\sup_kd_{g_k'}(y_k,c_k)<+\infty$. Since $\psi_k(c_k)=0$ and $\psi_k$ is $2$-Lipschitz for $g_{Z,k}$, by \eqref{eq:base-radius}, which passes to $\Omega_Z^{(k)}$, we have, for $0<\delta<1$,
\[
 d_{g_k'}\bigl(y_k,\{\psi_k\geqslant\delta\}\bigr)\geqslant\tfrac\delta2Q_k^{1/2}-d_{g_k'}(y_k,c_k)\longrightarrow+\infty.
\]
So $\Sigma_Z$, the only orbit of $Z^{(k)}$ other than principal orbits, rods and corners, escapes, and Lemma~\ref{lem:torus-limit}(a), with the corners $c_k$, gives the limiting action. Its convergence maps preserve the stabilizers, so it has no orbit with nontrivial finite stabilizer.

(iii)  Otherwise, there are $x_k\in Z^{(k)}$ with $|\Rm(g_{Z,k})|(x_k)\varrho_k^2\to+\infty$, where $\varrho_k:=\frac{1}{2}d_{g_{Z,k}}(x_k,\mathcal C_k)\leqslant\frac{\pi}{2}$, by Corollary~\ref{cor:finite-fundamental-groups}. Let $y_k$ maximize $|\Rm(g_{Z,k})|(y)\bigl(\varrho_k-d_{g_{Z,k}}(x_k,y)\bigr)^2$ over $\overline B_{g_{Z,k}}(x_k,\varrho_k)$, and put
 \[
 r_k:=\tfrac12\bigl(\varrho_k-d_{g_{Z,k}}(x_k,y_k)\bigr)\qquad\text{and}\qquad Q_k:=|\Rm(g_{Z,k})|(y_k).
\]
Then $Q_kr_k^2\to+\infty$, hence $Q_k\to+\infty$, and $|\Rm(g_{Z,k})|\leqslant4Q_k$ on $B_{g_{Z,k}}(y_k,r_k)$, while $d_{g_{Z,k}}(y_k,\mathcal C_k)\geqslant\varrho_k\geqslant2r_k$. So (ii) gives corners at  $g_{Z,k}$-distance at most $CQ_k^{-1/2}=o(r_k)$ from $y_k$,  a contradiction, which proves \eqref{eq:curvature-away-corners}.
\end{proof}

\begin{lemma}\label{lem:corner-noncollapse}
There are $L_k\to0$ and $C$ such that all corners $c,c'\in\mathcal C_k$ satisfy
\begin{equation}\label{eq:corner-noncollapse}
 d_{g_{Z,k}}(c,c')\leqslant CL_k\qquad\text{and}\qquad\Vol_{g_{Z,k}}B_{g_{Z,k}}(c,CL_k)\geqslant C^{-1}L_k^4.
\end{equation}
\end{lemma}
\begin{proof}
 We choose the level set $\sigma_k$ of $\rho_k$ whose length $L_k$ equals the minimum of $|d\ell_k|_{h_k}^{-1}$ on $\sigma_k$, see \eqref{eq:rho-level-choice}. The annulus $\{t_k\leqslant\rho_k\leqslant2t_k\}$ then has width and circumference comparable to $L_k$, by \eqref{eq:rho-annulus-metric}, hence base area at least $C^{-1}L_k^2$, and Step~3 shows that its torus orbits have area at least $C^{-1}L_k^2$. We use the reduced data $h_k,\rho_k,\Phi_k$ of $g_{Z,k}$, with $\ell_k:=\log\rho_k$ and  $\psi_k=\rho_k/\rho_k(o)$, on the orbit space $\Omega_Z^{(k)}$ of Section~\ref{subsec:free-quotients}, the quotient of $\Omega^{(k)}$ by the rotation induced by  $\varpi_k$, of order $k+1$. The estimates \eqref{eq:base-radius} and \eqref{eq:core-coordinates}, the flow of $V_\psi$ and the retractions  of Lemma~\ref{lem:core-retraction}  pass to $\Omega_Z^{(k)}$. Below, $C$ denotes constants independent of $k$.

\emph{Step 1: the level set $\sigma_k$.}  Away from $o$, put $\varkappa:=|d\ell_k|_{h_k}^{-1}$, and let $N:=\varkappa\nabla_{h_k}\ell_k$ be the ascending unit normal to the level sets of $\rho_k$.  The Hessian identity in \eqref{eq:Einstein-identities}, $K_{h_k}\leqslant1$ (Lemma~\ref{lem:bochner}) and $B\geqslant0$ give
\[
 N\varkappa=-\varkappa^2\Hess_{h_k}\ell_k(N,N)\geqslant\tfrac12+2\varkappa^2.
\]
So the minimum of $\varkappa$ on the level sets of $\rho_k$ strictly increases along the flow of $V_\psi$, whereas their lengths strictly decrease, by \eqref{eq:core-coordinates}. As $\varkappa=\psi_k/|d\psi_k|_{h_k}$, \eqref{eq:base-radius} gives
\[
 \frac{\psi_k}{2\sqrt{1-\psi_k^2}}\leqslant\varkappa\leqslant\frac12\sqrt{\frac{\psi_k}{1-\psi_k}}.
\]
Hence near $\partial\Omega_Z^{(k)}$, $\varkappa\to0$ while, for fixed $k$, the lengths stay bounded below, and near $o$, $\varkappa\to+\infty$ while the lengths tend to zero. So there is a unique level set
\begin{equation}\label{eq:rho-level-choice}
 \sigma_k:=\{\rho_k=t_k\}\qquad\text{with}\qquad L_k:=\operatorname{Length}_{h_k}(\sigma_k)=\min_{\sigma_k}\varkappa.
\end{equation}
 Fix $\delta\in(0,1)$. By Lemma~\ref{lem:smooth-cores}, the level set $\{\psi_k=\delta\}$ of $\Omega^{(k)}$ has length at most $C_\delta$, so its image in $\Omega_Z^{(k)}$ has length at most $C_\delta/(k+1)$, which is smaller than $\frac{\delta}{2}\leqslant\min\varkappa$ there for large $k$. Hence $t_k<\delta\rho_k(o)$ for large $k$, so $t_k/\rho_k(o)\to0$ and $L_k\to0$.

On $\{\rho_k\leqslant2t_k\}$,  for large $k$,  the same identity and $\Delta_{h_k}\ell_k=-6-\varkappa^{-2}$ give $-Ch_k\leqslant\varkappa^2\Hess_{h_k}\ell_k\leqslant-\frac{1}{2}N^\flat\otimes N^\flat$. As $d\varkappa=-\varkappa^2\Hess_{h_k}\ell_k(N,\cdot)$, this gives, for $J$ tangent to the level sets of $\rho_k$,
 \begin{equation}\label{eq:rho-level-derivatives}
 \tfrac12\leqslant N\varkappa,\qquad |d\varkappa|_{h_k}\leqslant C,\qquad\text{and}\qquad -C|J|_{h_k}^2\leqslant\varkappa^2\Hess_{h_k}\ell_k(J,J)\leqslant0.
\end{equation}
 Integrating $|d\varkappa|_{h_k}\leqslant C$ along $\sigma_k$,  of length $L_k=\min_{\sigma_k}\varkappa$, gives $L_k\leqslant\varkappa\leqslant CL_k$ on $\sigma_k$.

Let $\theta\in\mathbb R/\mathbb Z$ be the arclength of $\sigma_k$ divided by $L_k$, transported along the flow of $V_\psi$, and put $u:=\log(\rho_k/t_k)$. Then  $h_k=\varkappa^2du^2+|\partial_\theta|_{h_k}^2d\theta^2$, $\partial_u\log\varkappa=N\varkappa$ and, by \eqref{eq:core-coordinates}, $\frac{1}{2}\partial_u|\partial_\theta|_{h_k}^2=\varkappa^2\Hess_{h_k}\ell_k(\partial_\theta,\partial_\theta)$. On $\{0\leqslant u\leqslant\log2\}$, \eqref{eq:rho-level-derivatives}  therefore gives $C^{-1}L_k\leqslant\varkappa\leqslant CL_k$ and $C^{-1}L_k\leqslant|\partial_\theta|_{h_k}\leqslant L_k$, so
\begin{equation}\label{eq:rho-annulus-metric}
 C^{-1}L_k^2(du^2+d\theta^2)\leqslant h_k\leqslant CL_k^2(du^2+d\theta^2)\qquad\text{on }\{t_k\leqslant\rho_k\leqslant2t_k\}.
\end{equation}
For $u\leqslant0$,  $\partial_u\log\varkappa\geqslant\frac{1}{2}$ gives $\varkappa(u,\theta)\leqslant e^{u/2}\varkappa(0,\theta)$, so each descending curve $\{\theta=\theta_0,\ u\leqslant0\}$ has length at most  $2\varkappa(0,\theta_0)\leqslant CL_k$. Let $P_k$ be the retraction of Lemma~\ref{lem:core-retraction} onto $\{\rho_k\geqslant t_k\}$, which maps each such curve to its initial point. Then $d_{h_k}(y,P_k(y))\leqslant CL_k$ for $y\in\{\rho_k<t_k\}$, and, by continuity, for $y\in\mathcal C_k$. A corner $c$ is fixed by the torus, so $d_{g_{Z,k}}(c,x)=d_{h_k}(c,\pi(x))$ for $x\in Z^{(k)}$, and, as $\sigma_k$ has length $L_k$, this gives the first inequality of \eqref{eq:corner-noncollapse}.

\emph{Step 2: a short path to a regular rod.} We find a short path from $\sigma_k$ to an open rod which avoids the corners, so that the curvature is controlled along it. There are at most $\frac{m}{2}$ points $P_k(c)$, $c\in\mathcal C_k$, on $\sigma_k$, which has $\theta$-length one, so the midpoint $x_k$ of a largest arc of $\sigma_k$ between them has $\theta$-distance at least $\frac{1}m$ from all of them. By \eqref{eq:rho-annulus-metric}, a path in $\{\rho_k\geqslant t_k\}$ from $x_k$ to $P_k(c)$ has length at least $C^{-1}L_k$: if it stays in $\{\rho_k\leqslant2t_k\}$, its $\theta$-length is at least $\frac{1}m$, and otherwise its $u$-length is at least $\log2$. By \eqref{eq:core-retraction}, $d_{h_k}(x_k,P_k(c))\geqslant C^{-1}L_k$. Let $\gamma_k$ be the descending curve from $x_k$. Since $P_k$ is $1$-Lipschitz and constant on $\gamma_k$, every $y\in\gamma_k$ satisfies $d_{h_k}(y,c)\geqslant C^{-1}L_k$ for $c\in\mathcal C_k$. So $\gamma_k$ has length at most $CL_k$ and ends at a point $e_k$ of an open rod. Hence the horizontal lift  $\widetilde\gamma_k$ of $\gamma_k$, from  $\widetilde x_k$ over $x_k$ to  $\widetilde e_k$ over $e_k$, stays at $g_{Z,k}$-distance at least $C^{-1}L_k$ from $\mathcal C_k$, and \eqref{eq:curvature-away-corners} gives $|\Rm(g_{Z,k})|\leqslant CL_k^{-2}$ along  $\widetilde\gamma_k$.

\emph{ Step 3: the torus scale.} The Killing field of the rod vector at the end of this path vanishes there, and its derivative is the rotation by a right angle of the normal plane, by the smooth angle condition; transporting its value and derivative to $\sigma_k$ along the path bounds the orbit area from below. On $\sigma_k$,  $|d\Phi_k|_{h_k}^2=4K_{h_k}+|d\ell_k|_{h_k}^2\leqslant4+\varkappa^{-2}$, so $\int_{\sigma_k}|d\Phi_k|_{h_k}\,ds_{h_k}\leqslant C$. The isometry $\mathsf s_k$ of Section~\ref{subsec:models} preserves  $\sigma_k$, and an arc of $\sigma_k$ joins  each $y\in\sigma_k$ to $\mathsf s_k(y)$, so \eqref{eq:iota-equivariance} and \eqref{eq:cyclic-displacement}  give  $d_{\mathbb H^2}(\Phi_k(y),\Phi_o)\leqslant C$, that is, $\Phi_k\leqslant C\Phi_o$ on $\sigma_k$. Let $v$ be the rod vector at $e_k$, and $X_v:=v_1X_1+v_2X_2$. By \eqref{eq:matrix-smooth-side},  $X_v(\widetilde e_k)=0$ and $|\nabla X_v|(\widetilde e_k)=\sqrt2$. The block form \eqref{eq:matrix-ansatz} gives  $|\nabla X_v|\leqslant C(|d\ell_k|_{h_k}+|d\Phi_k|_{h_k})|X_v|$ on principal orbits, and  at $\widetilde x_k$ we have $|d\ell_k|_{h_k}+|d\Phi_k|_{h_k}\leqslant CL_k^{-1}$ and $|X_v|^2=t_k\Phi_k(v,v)\leqslant Ct_k$. So \eqref{eq:killing-transport} along  $\widetilde\gamma_k$, with $r:=L_k$, gives
\[
 \sqrt2=|\nabla X_v|(\widetilde e_k)\leqslant C\bigl(L_k^{-1}|X_v|+|\nabla X_v|\bigr)(\widetilde x_k)\leqslant C\,\frac{\sqrt{t_k}}{L_k},\qquad\text{that is,}\qquad t_k\geqslant C^{-1}L_k^2.
\]

\emph{Step  4:  the volume.}  By Step~1 and \eqref{eq:rho-annulus-metric}, the corners and the points of $\{t_k\leqslant\rho_k\leqslant2t_k\}$ lie within $h_k$-distance $CL_k$ of $\sigma_k$,  so the preimage of this annulus lies in $B_{g_{Z,k}}(c,CL_k)$  for every $c\in\mathcal C_k$. Its base has $h_k$-area  at least $C^{-1}L_k^2$, by \eqref{eq:rho-annulus-metric}, and its torus orbits have area $4\pi^2\rho_k\geqslant4\pi^2t_k\geqslant C^{-1}L_k^2$, by \eqref{eq:scalar-volume} and Step~3. This gives the second inequality of \eqref{eq:corner-noncollapse}.
\end{proof}

We now prove the compactness at corners.

\begin{proposition}\label{prop:corner-compactness}
There is $C$ such that
\begin{equation}\label{eq:corner-distances}
 \pi\leqslant\Lambda_k^{1/2}d_{g_{Z,k}}(c,c')\leqslant C
\end{equation}
for all distinct $c,c'\in\mathcal C_k$. With $L_k$ as in Lemma~\ref{lem:corner-noncollapse}, we also have
\begin{equation}\label{eq:scale-comparison}
 C^{-1}\leqslant\Lambda_kL_k^2\leqslant C.
\end{equation}For $c_k\in\mathcal C_k$, a subsequence of $(Z^{(k)},\Lambda_kg_{Z,k},c_k)$ converges smoothly on compact subsets to a complete, nonflat and Ricci-flat $(X,g_X,c)$ with $|\Rm|\leqslant1$ and $\int_X|\Rm|^2\leqslant16\pi^2m$. 

The torus actions converge to an effective isometric action of $\mathbb R^2/2\pi\mathbb Z^2$ on $X$ with the original generators, rod vectors and rod and corner models, without orbits with nontrivial finite stabilizer, and the $\frac{m}{2}$ corners of $Z^{(k)}$ converge to the $\frac{m}{2}$ distinct corners of $X$.
\end{proposition}

\begin{proof}
By Lemma~\ref{lem:curvature-localization}, $\Lambda_k\to+\infty$ and \eqref{eq:curvature-away-corners} holds. Let $p_k$ satisfy $|\Rm(g_{Z,k})|(p_k)=\Lambda_k$. Lemma~\ref{lem:curvature-localization}(ii), with $y_k:=p_k$ and $r_k:=\pi$, gives a subsequence of $(Z^{(k)},\Lambda_kg_{Z,k},p_k)$ converging to a complete, nonflat and Ricci-flat limit $(X,g_X,p)$ with a torus action as in the statement, and corners $c_k$ with $\Lambda_k^{1/2}d_{g_{Z,k}}(p_k,c_k)\leqslant C$.

 Along a minimizing geodesic between two distinct corners, some combination of $X_1,X_2$ restricts to a nonzero Jacobi field vanishing at both ends, since, by the corner model \eqref{eq:matrix-smooth-corner}, the derivatives at a corner of the Killing fields of the two adjacent rod vectors are the rotations of two orthogonal planes, so they do not both vanish on a nonzero tangent vector. As $\sec(\Lambda_kg_{Z,k})\leqslant1$, this geodesic has $\Lambda_kg_{Z,k}$ -length at least $\pi$, which is the first inequality of \eqref{eq:corner-distances}, compare \cite[proof of Lemma~9]{Don08}.

Suppose that $\Lambda_kL_k^2\to+\infty$ along a subsequence. By Lemma~\ref{lem:corner-noncollapse} and the Bishop-Gromov inequality for $\Ric(g_{Z,k})=3g_{Z,k}$, $\Vol_{g_{Z,k}}B_{g_{Z,k}}(c_k,r)\geqslant C^{-1}r^4$ for $0<r\leqslant CL_k$, that is, $\Vol_{\Lambda_kg_{Z,k}}B(c_k,s)\geqslant C^{-1}s^4$ for $0<s\leqslant C\Lambda_k^{1/2}L_k$. As $c_k$ stays at bounded $\Lambda_kg_{Z,k}$-distance from $p_k$, the limit $X$ has Euclidean volume growth, and its curvature energy is at most $16\pi^2m$, since $\int|\Rm|^2\,dV$ is invariant under rescaling. By Corollary~\ref{cor:ale-exclusion}, whose determinant hypothesis holds by \eqref{eq:rod-determinant}, it is flat, a contradiction. Hence $\Lambda_kL_k^2\leqslant C$, and Lemma~\ref{lem:corner-noncollapse} gives the second inequality of \eqref{eq:corner-distances}. For distinct corners $c,c'$, \eqref{eq:corner-distances} and \eqref{eq:corner-noncollapse} give $\pi\leqslant\Lambda_k^{1/2}d_{g_{Z,k}}(c,c')\leqslant C\Lambda_k^{1/2}L_k$, which is the lower bound in \eqref{eq:scale-comparison}.

So all corners lie at bounded $\Lambda_kg_{Z,k}$-distance from $p_k$, and the limit may be based at $c:=\lim c_k$: $(Z^{(k)},\Lambda_kg_{Z,k},c_k)\to(X,g_X,c)$, with $|\Rm|\leqslant1$ and $\int_X|\Rm|^2\leqslant16\pi^2m$. The convergence maps of Lemma~\ref{lem:torus-limit}(a) preserve the stabilizers, so the corners of $X$ are the limits of corners of $Z^{(k)}$, and, by \eqref{eq:corner-distances}, all $\frac{m}{2}$ corners converge, with distinct limits.
\end{proof}

\subsubsection{Global Weyl coordinates}

\begin{lemma}\label{lem:weyl}
Let $X$ be as in Proposition~\ref{prop:corner-compactness}. Its orbit space $\overline\Omega_X:=X/\mathbb T^2$ is a closed half-cylinder, and on its interior there are coordinates $(\rho,z)$, with $z\in\mathbb R/P_{\mathrm W}\mathbb Z$ for some $P_{\mathrm W}>0$, in which locally
\begin{gather}
 g_X=e^{2a_{\mathrm W}}(d\rho^2+dz^2)+\rho\,\Phi_{ij}\,d\phi_i\,d\phi_j,\qquad\text{with}\qquad \rho=\sqrt{\det G}>0,\label{eq:weyl}\\
 \partial_\rho\bigl(\rho\,\Phi^{-1}\partial_\rho\Phi\bigr)+\partial_z\bigl(\rho\,\Phi^{-1}\partial_z\Phi\bigr)=0,\label{eq:weyl-harmonic}\\
 \partial_\rho a_{\mathrm W}=-\frac{1}{4\rho}+\frac{\rho}{4}\bigl(|\partial_\rho\Phi|_{\mathbb H^2}^2-|\partial_z\Phi|_{\mathbb H^2}^2\bigr),\qquad\text{and}\qquad \partial_za_{\mathrm W}=\frac{\rho}{2}\langle\partial_\rho\Phi,\partial_z\Phi\rangle_{\mathbb H^2},\label{eq:weyl-nu}\\
 K_h\leqslant0,\qquad |D\Phi|\leqslant\rho^{-1},\qquad\text{and}\qquad -\frac{1}{2\rho}\leqslant\partial_\rho a_{\mathrm W}\leqslant0,\label{eq:weyl-bounds}
\end{gather}
where $D$ is the flat derivative for $d\rho^2+dz^2$. The limit $\mathsf s$ of the isometries $\mathsf s_k$ of Section~\ref{subsec:models} acts by  $z\mapsto z+\frac{2P_{\mathrm W}}m$, so the rod endpoints are equally spaced in $z$, and
\begin{equation}\label{eq:weyl-shift}
 \Phi(\rho,z+\tfrac{2P_{\mathrm W}}m)=\iota\bigl(\Phi(\rho,z)\bigr).
\end{equation}
Going once around the $z$-circle does not translate the torus angles: the angular identification is $(z,\boldsymbol\phi)\sim(z+P_{\mathrm W},\boldsymbol\phi)$.
\end{lemma}

\begin{proof}
\emph{Local equations.} By the scaling rules at the beginning of Section~\ref{subsec:models}, the Einstein constant of $\Lambda_kg_{Z,k}$ is $3\Lambda_k^{-1}\to0$, so \eqref{eq:Einstein-identities}, Lemma~\ref{lem:scalar}(a) and \eqref{eq:concavity} give,  in the limit, $\Delta_h\rho=0$, $K_h\leqslant0$ and $\Hess_h\sqrt\rho\leqslant0$ on the principal part of $X$. If $\rho$ had an interior critical point $x_*$, then $\Hess_h\sqrt\rho\leqslant0$, restricted to the geodesics from $x_*$, would make $x_*$ a local maximum of $\rho$, and the strong maximum principle for the harmonic function $\rho$ would force $\rho$ to be constant, which contradicts $\rho=0$ at a corner. So $dz:=*d\rho$ gives local Weyl coordinates, in which \eqref{eq:weyl-harmonic} and \eqref{eq:weyl-nu} are the standard Weyl equations \cite[Section~2]{LS25}. The bound $K_h\leqslant0$ reads $|D\Phi|\leqslant\rho^{-1}$, since $4K_h=|d\Phi|_h^2-|d\ell|_h^2$. Then \eqref{eq:weyl-nu} gives the bounds on $\partial_\rho a_{\mathrm W}$ in \eqref{eq:weyl-bounds}.

\emph{Orbit space.} Put $V_\rho:=\nabla_h\rho/|d\rho|_h^2$, so that $|V_\rho|_h=e^{a_{\mathrm W}}$ in local Weyl coordinates. By \eqref{eq:weyl-bounds}, along a flow line of $V_\rho$ parametrized by $\rho$,  the speed $|V_\rho|_h$ is nonincreasing  and at most $(\rho_1/\rho)^{1/2}|V_\rho|_h(\rho_1)$ for $\rho<\rho_1$. By completeness, every flow line is defined for all $\rho\in(0,+\infty)$ and has finite length as $\rho\downarrow0$. Its endpoint is a rod or corner point, by Proposition~\ref{prop:corner-compactness}. The flow identifies the interior $\Omega_X$ of $\overline\Omega_X$ with $(0,+\infty)\times\{\rho=1\}$, so the level set $\{\rho=1\}$ is connected, hence a line or a circle.

The isometry $\mathsf s_k$ of Section~\ref{subsec:models}, which satisfies $\Phi_k\circ\mathsf s_k=\iota\circ\Phi_k$ by \eqref{eq:iota-equivariance},  moves a corner by a bounded $\Lambda_kh_k$-distance, by \eqref{eq:corner-distances}. The Arzel\`a-Ascoli theorem, applied to $\mathsf s_k$ and $\mathsf s_k^{-1}$ on the converging orbit spaces, gives a subsequence converging, smoothly on the interior, to an orientation-preserving isometry $\mathsf s$ of $h$ with $\rho\circ\mathsf s=\rho$ and $\Phi\circ\mathsf s=\iota\circ\Phi$. As the corners of $X$ are the distinct limits of those of $Z^{(k)}$, by Proposition~\ref{prop:corner-compactness}, $\mathsf s$ maps each corner to an adjacent one, and has order $\frac{m}{2}\geqslant2$. Since $\mathsf s$ preserves $\rho$, it commutes with the flow of $V_\rho$, so it restricts to an orientation-preserving diffeomorphism of $\{\rho=1\}$ of order $\frac{m}{2}$, which is impossible on a line. Hence $\{\rho=1\}$ is a circle, and the closed form $dz=*d\rho$ integrates to a coordinate $z\in\mathbb R/P_{\mathrm W}\mathbb Z$ on $\Omega_X$, with $P_{\mathrm W}:=\int_{\{\rho=1\}}*d\rho$.

\emph{Boundary.} Let $w=x+iy$ be a conformal chart of a rod or corner model \eqref{eq:matrix-smooth-side}, respectively \eqref{eq:matrix-smooth-corner}, of $g_X$. The odd extension of $\rho$ is harmonic, so $\rho+iz$ is, up to a constant, a holomorphic function of $w$. At a rod point $(0,y_0)$, where $\rho=xe^{\bar\ell}$, its derivative is $e^{\bar\ell(0,y_0)}\neq0$. At a corner, where $\rho=xye^{\widehat\ell}$, the oddness of $\rho$ in $x$ and in $y$ excludes a linear term, and its leading term is $-\frac{i}{2}e^{\widehat\ell(0)}w^2$, which maps the quadrant onto a half-plane. We may therefore replace the corner chart by the first-quadrant branch of $x+iy:=\sqrt{2i(\rho+i(z-z_*))}$, where $z_*$ is the $z$-coordinate of the corner and the two adjacent rods are suitably ordered: by the double-zero expansion, this is a smooth change of conformal coordinates with nonzero derivative, preserving the parities, and $\rho=xy$, $z-z_*=\frac{1}{2}(y^2-x^2)$, compare \cite[Section~6]{CS04}. In both cases, $(\rho,z)$ extends to a homeomorphism from a neighborhood of the point in $\overline\Omega_X$ onto a relative neighborhood of a point of $\{\rho=0\}$. With the flow of $V_\rho$, $\overline\Omega_X$ is therefore homeomorphic to the closed half-cylinder $[0,+\infty)\times\mathbb R/P_{\mathrm W}\mathbb Z$.

As $\mathsf s$ preserves $\rho$ and $dz$, it is a translation in $z$. Since it maps each of the $\frac{m}{2}$ corners to an adjacent one, changing the sign of $z$ if necessary, it is  $z\mapsto z+\frac{2P_{\mathrm W}}m$, so the rod endpoints are equally spaced in $z$, and $\Phi\circ\mathsf s=\iota\circ\Phi$ is \eqref{eq:weyl-shift}.

\emph{Angular period.} Since the distribution orthogonal to the orbits is integrable, by \eqref{eq:matrix-ansatz}, a horizontal lift of a circuit of the $z$-circle closes up to a torus translation independent of the circuit; we show that this translation vanishes. The surfaces $\{\boldsymbol\phi=\mathrm{const}\}$ of $M^{(k)}$ are orthogonal to the torus orbits, by \eqref{eq:matrix-ansatz}. Lift a loop going once around the image of $o$ in $\Omega_Z^{(k)}$, in the direction of increasing $\theta$, orthogonally to the orbits. Starting at $(\theta_0,\boldsymbol\phi_0)$ in $M^{(k)}$, the lift ends at
 \begin{equation}\label{eq:quotient-angular-displacement}
 (\theta_0+\pi,\boldsymbol\phi_0)=\varpi_k\Bigl(\theta_0,\boldsymbol\phi_0-\frac{2\pi}{k+1}\,e_1\Bigr).
\end{equation}
So, in $Z^{(k)}$, the lift ends on its initial orbit, with angular displacement $-\frac{2\pi}{k+1}e_1$. Closing it in this orbit gives a loop whose lift to $M^{(k)}$ joins $p:=(\theta_0,\boldsymbol\phi_0)$ to  $\varpi_kp$, hence generates $\pi_1(Z^{(k)})=\mathbb Z_{k+1}$ (Corollary~\ref{cor:finite-fundamental-groups}).   A loop going once around the orbit cylinder is the limit of loops going once around the image of $o$ in the quotient orbit spaces $\Omega_Z^{(k)}$: equivariant convergence identifies the neighborhoods of all rods and corners, in their cyclic order. By Proposition~\ref{prop:corner-compactness}, the metrics and the original torus generators converge, without change of integral basis, so the angular displacement $-\frac{2\pi}{k+1}e_1$ of these loops, given by \eqref{eq:quotient-angular-displacement}, tends to zero in the fixed torus. Thus the angular identification is $(z,\boldsymbol\phi)\sim(z+P_{\mathrm W},\boldsymbol\phi)$.
\end{proof}

\subsubsection[Kasner asymptotics, topology and identification]{Kasner asymptotics, topology and identification}\label{subsec:kasner-topology}

\begin{proof}[Proof of Theorem~\ref{thm:periodic-models}]
By Proposition~\ref{prop:corner-compactness} and Lemma~\ref{lem:weyl}, it remains to prove the Kasner asymptotics, the  uniqueness of the limit,  which gives the convergence of the whole sequence, the topology, the curvature energy and the identifications.

\emph{Kasner asymptotics.}
By \eqref{eq:weyl-shift} and \eqref{eq:cyclic-displacement}, with $\Phi$ on the diagonal geodesic in the square case by \eqref{eq:ansatz}, and by the bound $|D\Phi|\leqslant\rho^{-1}$ of \eqref{eq:weyl-bounds}, we have
 \begin{equation}\label{eq:end-shape}
 d_{\mathbb H^2}(\Phi(\rho,z),\Phi_o)
 \leqslant
 d_{\mathbb H^2}(\Phi(\rho,z),\Phi(\rho,z+\tfrac{2P_{\mathrm W}}m))
 \leqslant\frac{2P_{\mathrm W}}{m\rho}.
\end{equation}
Let $\Delta_3:=\partial_\rho^2+\rho^{-1}\partial_\rho+\partial_z^2$ be the Euclidean Laplacian on axisymmetric functions. By Lemma~\ref{lem:matrix-calculus}(iii), \eqref{eq:weyl-harmonic} reads $\Delta_3\Phi=|D\Phi|^2\Phi$, so $U:=\Pi_{\Phi_o}\Phi$ of \eqref{eq:central-matrix-coordinate} satisfies
\[
 \Delta_3U=|D\Phi|^2U.
\]
By \eqref{eq:end-shape} and $|U|_o=\sinh d_{\mathbb H^2}(\Phi,\Phi_o)$, we have $U=O(\rho^{-1})$, and $|D\Phi|\leqslant\rho^{-1}$, so $\Delta_3U=O(\rho^{-3})$.
Rescaled interior estimates \cite[Theorems~9.11 and~6.2]{GT01} on balls of radius $\rho/4$ in the
periodic lift, followed by differentiation of the equation $\Delta_3U=|D\Phi|^2U$, give
\[
 |D^jU|\leqslant C_j\rho^{-1-j},
 \qquad\text{for }j\geqslant0.
\]

\emph{The conformal factor.} The quadratures \eqref{eq:weyl-nu} then give
$D(a_{\mathrm W}+\frac{1}{4}\log\rho)=O(\rho^{-3})$.
Integration in $\rho$, and the $z$-equation, give a constant $c_{\mathrm K}$
such that
\[
 \left|D^j\left(a_{\mathrm W}+\tfrac14\log\rho-c_{\mathrm K}\right)\right|
 \leqslant C_j\rho^{-2-j},\qquad\text{for }j\geqslant0.
\]
By \eqref{eq:weyl}, these are the asserted asymptotics to
$e^{2c_{\mathrm K}}\rho^{-1/2}(d\rho^2+dz^2)
+\rho(\Phi_o)_{ij}\,d\phi_i\,d\phi_j$.
With $\varsigma:=\frac{4}{3}e^{c_{\mathrm K}}\rho^{3/4}$, this model has Kasner
exponents $(-\frac{1}{3},\frac{2}{3},\frac{2}{3})$.

\emph{Uniqueness of the limit.} Let $(X',g_{X'},c')$ be the limit along another subsequence, with Weyl coordinates $(\rho',z')$, shape $\Phi'$ and conformal factor $a_{\mathrm W}'$. The homothety $g_{X'}\mapsto\lambda g_{X'}$, with $\lambda:=P_{\mathrm W}/P_{\mathrm W}'$, multiplies $\rho'$ and $z'$ by $\lambda$, so we may assume that both Weyl periods equal $P_{\mathrm W}$. Normalize $z$ and $z'$ to vanish at $c$ and $c'$. By the choice of the corners $c_k$, the two orbit spaces then have the same rod endpoints  $\frac{2P_{\mathrm W}}m\mathbb Z$ and the same rod vectors, in the same cyclic order, and both have the untwisted angular identification of Lemma~\ref{lem:weyl}.

Identify  the two orbit spaces through $(\rho',z')=(\rho,z)$. At a corner, work in the common integral basis of the two adjacent rod vectors, ordered so that they collapse on $\{x=0\}$ and  $\{y=0\}$, respectively, and in the coordinates $\rho=xy$, $z-z_*=\frac{1}{2}(y^2-x^2)$ of the proof of Lemma~\ref{lem:weyl},  which are the same for both metrics. With $C:=\rho^{-1/2}\operatorname{diag}(x,y)$, \eqref{eq:matrix-smooth-corner} gives
\[
 \Phi=C\widehat GC^{\mathsf T},\qquad\text{and}\qquad \Phi'=C\widehat G'C^{\mathsf T},
\]
 where $\widehat G$ and $\widehat G'$  have determinant one, as $\rho=xy$. The map $S\mapsto CSC^{\mathsf T}$ is an isometry of $\mathbb H^2$, so $d_{\mathbb H^2}(\Phi,\Phi')=d_{\mathbb H^2}(\widehat G,\widehat G')$, which is bounded near the corner since $\widehat G$ and $\widehat G'$ are smooth and positive definite.  Near a rod point,  the same argument applies  in a common rod-adapted integral basis, with the common conformal coordinate $\rho+iz$ and $C:=\rho^{-1/2}\operatorname{diag}(\rho,1)$. These unimodular changes of basis act in the same way on both metrics and preserve their hyperbolic distance. Put $u:=d_{\mathbb H^2}(\Phi,\Phi')^2$. By \eqref{eq:weyl-harmonic}, $\Phi$ and $\Phi'$ are harmonic maps to $\mathbb H^2$ for the flat metric $d\rho^2+\rho^2d\vartheta^2+dz^2$, whose Laplacian on functions of $(\rho,z)$ is $\Delta_3$, so the convexity of the squared distance on $\mathbb H^2\times\mathbb H^2$ gives $\Delta_3u\geqslant0$. As $u$ is bounded near the axis, which has codimension two in $D_R^2\times(\mathbb R/P_{\mathrm W}\mathbb Z)$, the maximum principle  of \cite[Section~2]{Wei95} gives $\sup_{\{\rho<R\}}u\leqslant\sup_{\{\rho=R\}}u$, and the right-hand side tends to $0$ as $R\to+\infty$, by \eqref{eq:end-shape}. Hence $\Phi'=\Phi$.

The quadratures \eqref{eq:weyl-nu} then show that $a_{\mathrm W}'-a_{\mathrm W}$ is constant, and the angle condition of \eqref{eq:matrix-smooth-side} on any rod makes it vanish. Hence the rescaled $g_{X'}$ is isometric to $g_X$, by an isometry which maps $c'$ to $c$ and preserves the torus generators.  The original subsequential limits both satisfy $\max|\Rm|=1$, since the points of maximal curvature of Proposition~\ref{prop:corner-compactness} stay at bounded distance from the corners,  so the isometry between $g_X$ and $\lambda g_{X'}$ forces $\lambda=1$. All subsequential limits therefore coincide, and the whole sequence converges.

\emph{Topology.} Fill the principal boundary of $\{\rho\leqslant\rho_0\}$ by
$\mathbb T^2\times D^2$, with $\partial D^2$ identified with the
$z$-circle, as in \cite[Section~5.1]{KRWY22}. The resulting torus manifold has a disk orbit space
and the prescribed two or three rods, so \cite{OR70} identifies
it with $\mathbb S^4$ or $\overline{\mathbb{CP}}{}^2$,
respectively. Removing the core torus gives the asserted
diffeomorphism types. The radial retraction $\rho\mapsto(1-s)\rho$,
$0\leqslant s\leqslant1$, retracts $X$ onto the cyclic union of
its fixed spheres, so $\pi_1(X)=\mathbb Z$.
In the hexagon case, by \cite{OR70}, the sphere $F$ over the middle rod has
\[
 [F]\cdot[F]
 =-\frac{\det(e_1,e_2)}{\det(e_1,e_1+e_2)\det(e_1+e_2,e_2)}=-1,
\]
as do the other fixed spheres, by the symmetry $\mathsf s$ of Lemma~\ref{lem:weyl}. In particular,
$X_{\mathrm{hex}}$ is not spin.

\emph{Curvature energy.} The removed torus has Euler characteristic zero, so the topology above gives $\chi(X)=\chi(\mathbb S^4)=2$ in the square case and $\chi(X)=\chi(\overline{\mathbb{CP}}{}^2)=3$ in the hexagon case, that is, $\chi(X)=\frac{m}{2}$.  By the differentiated Kasner asymptotics above, the hypersurfaces $\{\varsigma=S\}$ satisfy $\Vol\{\varsigma=S\}=O(S)$, $|\mathrm{II}|=O(S^{-1})$ and $|\Rm|=O(S^{-2})$, so the boundary term of the Chern-Gauss-Bonnet formula on $\{\varsigma\leqslant S\}$ is $O(S^{-2})$. Letting $S\to+\infty$ gives, in our normalization of $|\Rm|$,
\[
 \int_X|\Rm|^2\,dV=32\pi^2\chi(X)=16\pi^2m,
\]
which is the curvature energy of every $Z^{(k)}$, by Corollary~\ref{cor:finite-fundamental-groups}. So the blow-ups retain all the curvature energy.

\emph{Square identification.}
The equal-rod Euclidean Myers-Korotkin-Nicolai metric of \cite[Section~4.1]{KRWY22}, with  circle directions $e_1$ and $e_2$  of period $2\pi$  and quotiented by one period, has alternating rods with rod vectors $e_1$ and $e_2$, with endpoints equally spaced in $z$, angle $2\pi$ at the rods  and the untwisted angular identification. By \cite[Section~4.1, (4.5)]{KRWY22}, its diagonal shape tends to a constant as $\rho\to+\infty$. The symmetry exchanging the two circles  and shifting $z$  by one rod fixes this constant, so it equals $\Phi_o=I$. After a homothety, its Weyl period is $P_{\mathrm W}$, and the uniqueness argument above  identifies $g_{\mathrm{sq}}$ with the one-period quotient of  this metric.

\emph{Hexagonal geometry.}
If $g_{\mathrm{hex}}$ were static with respect to a torus generator $V$, then
$g_{\mathrm{hex}}(V,\cdot)/g_{\mathrm{hex}}(V,V)$ would be closed.
For any independent torus generator $W$, this makes
$g_{\mathrm{hex}}(V,W)/g_{\mathrm{hex}}(V,V)$ constant. Subtracting that multiple of $V$
from $W$ makes $G$ diagonal in a fixed real basis.
Its kernels along the rods could then have only two directions,
contradicting the three prescribed rod directions.

If $g_{\mathrm{hex}}$ were locally Hermitian in either orientation, $W^+$ would
have at most two eigenvalues. By \cite[Proposition~5(iv)]{Der83},
either $W^+\equiv0$ or $W^+$ is nowhere zero.
In the first case the universal cover is hyperk\"ahler and spin,
contradicting the lift of a fixed sphere of odd self-intersection.
In the second case, \cite[Proposition~2.1]{Li24} gives a Hermitian
non-K\"ahler structure on a cover of degree at most two, which is
still complete, Ricci-flat and of finite curvature energy.
Since its end has the Kasner asymptotics above, it contradicts
\cite[Theorem~1.2]{Li24}.
\end{proof}

\begin{remark}\label{rem:covers}
The corner blow-ups $(M^{(k)},\Lambda_kg^{(k)},\widetilde c_k)$, with $\widetilde c_k$ over the corner $c_k$ of Theorem~\ref{thm:periodic-models},  converge to the infinite cyclic cover of $X$. Indeed,  the closed horizontal lift of one circuit of the orbit cylinder generates $\pi_1(X)=\mathbb Z$,  and the corresponding closed loop in $Z^{(k)}$ represents $\varpi_k$, which generates $\pi_1(Z^{(k)})=\mathbb Z_{k+1}$, by the generator calculation in the proof of Lemma~\ref{lem:weyl}, so its lift to $M^{(k)}$ closes only after $k+1$ circuits. Any fixed number of periods therefore lifts without identification for large $k$, and the limit has $z\in\mathbb R$. Its rods and corners repeat along $z$, so its curvature energy is infinite. The  estimates \eqref{eq:end-shape} and those following it still hold as $\rho\to+\infty$, while the Kasner end of Theorem~\ref{thm:periodic-models}  refers to the finite-period space $X$.
\end{remark}

\section{Further constructions}\label{sec:questions}

We expect that the same method produces other families of positive Einstein metrics. For each initial metric below, three ingredients of the proof have to be revisited: the symmetry-restricted indicial modes at the central torus, the positivity in the kernel identity of Proposition~\ref{prop:kernel}, and the rod restrictions used in compactness. Let us start with the Page metric \cite{Pag78}.
\begin{conjecture}\label{conj:page}
For every $k\geqslant1$, the manifold
\[
 P_k=(2k+1)\mathbb{CP}^{2}\#(2k+1)\overline{\mathbb{CP}}{}^{2}
\]
admits a positive Einstein metric obtained by deforming the Page metric by allowing conical singularities along a torus fiber.
\end{conjecture}
This conjecture is motivated by numerical approximations of these metrics on $P_1$, $P_2$ and $P_3$. A similar question can be asked about conical deformations of the toric Chen-LeBrun-Weber metric of \cite{CLW08}. Similar deformations of the Fubini-Study metric would reach the \textit{orientation reversals} of the manifolds of the hexagon construction. Numerical tests suggest that they give back the same Einstein metrics, with the opposite orientation.

Figure~\ref{fig:reachable-topologies} depicts the non-spin topologies reached by the hexagon construction, and those that the proposed constructions might reach, on
\[
 X_{p,q}=\#_p\mathbb{CP}^2\#_q\overline{\mathbb{CP}}{}^2,
 \qquad\text{and}\qquad X_{0,0}={\mathbb S^4},
\]
with the strict Hitchin-Thorpe inequalities
\cite{Hit74,Tho69} corresponding to
\[
 2(p+q+2)>3|p-q|
 \quad\Longleftrightarrow\quad p<5q+4\ \text{ and }\ q<5p+4.
\]

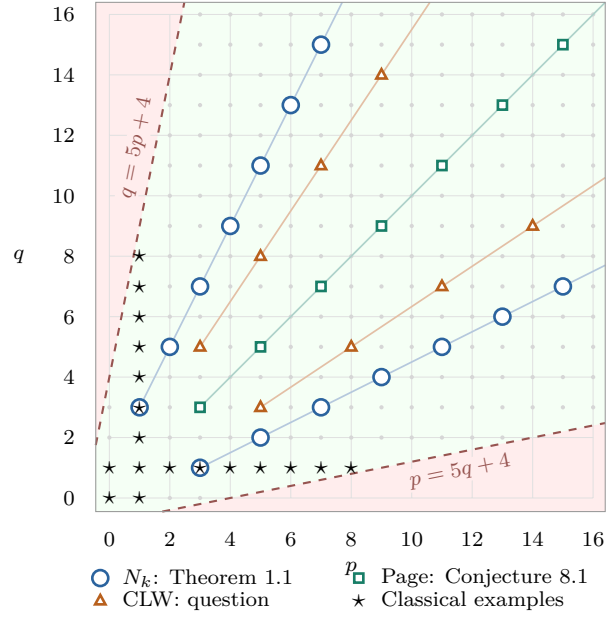
\begin{figure}[H]
\centering
\begingroup
\definecolor{topologyblue}{RGB}{43,99,158}
\definecolor{topologygreen}{RGB}{24,126,103}
\definecolor{topologyorange}{RGB}{184,96,29}
\definecolor{topologyred}{RGB}{150,84,80}
\begin{tikzpicture}[x=.40cm,y=.40cm,font=\scriptsize]
 \fill[green!4!white] (-.45,-.45) rectangle (16.4,16.4);
 \fill[red!7!white] (-.45,1.75)--(2.48,16.4)--(-.45,16.4)--cycle;
 \fill[red!7!white] (1.75,-.45)--(16.4,2.48)--(16.4,-.45)--cycle;
 \draw[step=2,black!12,line width=.25pt] (0,0) grid (16,16);
 \foreach \p in {0,...,16} {
   \foreach \q in {0,...,16} {
     \pgfmathtruncatemacro{\admissible}{(\q<5*\p+4)&&(\p<5*\q+4)}
     \ifnum\admissible=1
       \fill[black!16] (\p,\q) circle[radius=.8pt];
     \fi
   }
 }
 \draw[topologyred,dashed,line width=.8pt] (-.45,1.75)--(2.48,16.4);
 \draw[topologyred,dashed,line width=.8pt] (1.75,-.45)--(16.4,2.48);
 \node[topologyred,rotate=78.69,fill=red!7!white,inner sep=1pt]
   at (.89,11.6) {$q=5p+4$};
 \node[topologyred,rotate=11.31,fill=red!7!white,inner sep=1pt]
   at (11.6,.89) {$p=5q+4$};
 \draw[topologyblue!35,line width=.65pt] (1,3)--(7.7,16.4);
 \draw[topologyblue!35,line width=.65pt] (3,1)--(16.4,7.7);
 \draw[topologygreen!35,line width=.65pt] (3,3)--(16.4,16.4);
 \draw[topologyorange!40,line width=.65pt] (3,5)--(10.6,16.4);
 \draw[topologyorange!40,line width=.65pt] (5,3)--(16.4,10.6);
 \foreach \k in {1,...,7} {
   \pgfmathtruncatemacro{\q}{2*\k+1}
   \draw[topologyblue,fill=white,line width=1pt] (\k,\q) circle[radius=3pt];
   \draw[topologyblue,fill=white,line width=1pt] (\q,\k) circle[radius=3pt];
 }
 \foreach \p in {3,5,7,9,11,13,15} {
   \draw[topologygreen,fill=white,line width=1pt]
     (\p-.15,\p-.15) rectangle (\p+.15,\p+.15);
 }
 \foreach \k in {1,...,4} {
   \pgfmathtruncatemacro{\p}{2*\k+1}
   \pgfmathtruncatemacro{\q}{3*\k+2}
   \draw[topologyorange,fill=white,line width=1pt]
     (\p,\q+.19)--(\p-.17,\q-.12)--(\p+.17,\q-.12)--cycle;
   \draw[topologyorange,fill=white,line width=1pt]
     (\q,\p+.19)--(\q-.17,\p-.12)--(\q+.17,\p-.12)--cycle;
 }
 \foreach \p/\q in {0/0,1/0,0/1,1/1,1/2,2/1,1/3,3/1,1/4,4/1,
                      1/5,5/1,1/6,6/1,1/7,7/1,1/8,8/1} {
   \node[inner sep=0pt,font=\fontsize{10}{10}\selectfont] at (\p,\q) {$\star$};
 }
 \draw[black!40,line width=.4pt] (-.45,-.45) rectangle (16.4,16.4);
 \foreach \i in {0,2,...,16} {
   \node[below=4pt,font=\scriptsize] at (\i,-.45) {$\i$};
   \node[left=4pt,font=\scriptsize] at (-.45,\i) {$\i$};
 }
 \node[below=15pt] at (8,-.45) {$p$};
 \node[left=23pt] at (-.45,8) {$q$};
 \draw[topologyblue,fill=white,line width=1pt] (-.3,-2.65) circle[radius=3pt];
 \node[anchor=west] at (.1,-2.65) {$N_k$: Theorem~\ref{thm:main}};
 \draw[topologygreen,fill=white,line width=1pt]
   (8.1,-2.80) rectangle (8.4,-2.50);
 \node[anchor=west] at (8.65,-2.65) {Page: Conjecture~\ref{conj:page}};
 \draw[topologyorange,fill=white,line width=1pt]
   (-.3,-3.21)--(-.47,-3.52)--(-.13,-3.52)--cycle;
 \node[anchor=west] at (.1,-3.4) {CLW: question};
 \node[inner sep=0pt,font=\fontsize{10}{10}\selectfont] at (8.25,-3.4) {$\star$};
 \node[anchor=west] at (8.65,-3.4) {Classical examples};
\end{tikzpicture}
\endgroup
\caption{Constructed and proposed positive Einstein metrics on
$X_{p,q}=p\mathbb{CP}^{2}\#q\overline{\mathbb{CP}}{}^{2}$. Circles mark $(p,q)=(k,2k+1)$ of the present paper, squares mark the conjectural Page family $(2k+1,2k+1)$, triangles mark the predicted topologies of a proposed CLW family $(2k+1,3k+2)$.
The red regions and dashed boundaries violate the strict Hitchin-Thorpe inequalities.}
\label{fig:reachable-topologies}
\end{figure}

 We conjecture that all metrics constructed in this article and proposed in this section are linearly unstable. For the first family, this is Proposition~\ref{prop:square-instability}.

\begin{remark}
The present proof does not readily apply to these situations. With fewer symmetries, the first angular shape modes $R^{1/\beta}e^{\pm i\theta}$ are allowed at the central torus, and their second derivatives, of order $R^{1/\beta-2}$, may be unbounded for $\frac{1}{2}<\beta<1$, so the central spaces and Lemma~\ref{lem:center-expansion} must be modified. The corner symmetry giving \eqref{eq:corner-symmetry}, used to match the boundary values of $b$ in the proof of Proposition~\ref{cmp:comparison-isomorphism}, is also lost. The rod determinant condition \eqref{eq:rod-determinant} still holds for both metrics, however.
\end{remark}

\appendix

\section{Auxiliary proofs and estimates}\label{sec:appendix-auxiliary}

\subsection{Hyperbolic exponential and matrix formulas}\label{subsec:matrix-formulas}

\begin{lemma}\label{lem:matrix-calculus}
Let $S\in\mathbb H^2$ and $\xi\in T_S\mathbb H^2$, with $|\xi|_S^2:=\frac{1}{2}\operatorname{tr}\bigl((S^{-1}\xi)^2\bigr)$.
\begin{enumerate}[label=\textup{(\roman*)}]
\item The exponential map of $g_{\mathbb H^2}$ is
\begin{equation}\label{eq:matrix-exp}
 \operatorname{Exp}_S\xi=S\exp(S^{-1}\xi)=\cosh(|\xi|_S)\,S+\frac{\sinh(|\xi|_S)}{|\xi|_S}\,\xi.
\end{equation}
\item For every $\Phi\in\mathbb H^2$, the projection \eqref{eq:matrix-projection} restricts to an isomorphism $\Pi_S:T_\Phi\mathbb H^2\to T_S\mathbb H^2$, equal to the identity if $\Phi=S$. Moreover $U:=\Pi_S\Phi\in T_S\mathbb H^2$ is a global coordinate on $\mathbb H^2$, with inverse $\Phi=\sqrt{1+|U|_S^2}\,S+U$ as in \eqref{eq:central-matrix-coordinate}.
\item The maps \eqref{eq:matrix-exp}, \eqref{eq:matrix-projection} and \eqref{eq:central-matrix-coordinate} commute with the congruences $S\mapsto P^{\mathsf T}SP$, $|\det P|=1$, hence with the target symmetries, and the tension field is \eqref{eq:matrix-tension}.

\end{enumerate}
\end{lemma}

\begin{proof}
(i) Since $\operatorname{tr}(S^{-1}\xi)=0$, we have $(S^{-1}\xi)^2=|\xi|_S^2I$ in 2D, and summing the exponential series gives the second equality in \eqref{eq:matrix-exp}. The first one is the exponential map of the Riemannian symmetric space of positive definite symmetric $2\times2$ matrices, with the metric $|\xi|_S^2$ above, \cite[Chapter~II.10]{BH99}, in which $\mathbb H^2=\{\det=1\}$ is totally geodesic; compare \cite[Appendix]{CS04}. (ii) If $W\in T_\Phi\mathbb H^2$ and $\Pi_SW=0$, then we have
\[
 W=cS,\qquad\text{and}\qquad 0=\operatorname{tr}(\Phi^{-1}W)=c\operatorname{tr}(\Phi^{-1}S)\quad\Longrightarrow\quad c=0,
\]
since $\operatorname{tr}(\Phi^{-1}S)>0$. If $\Phi=cS+U$ with $U\in T_S\mathbb H^2$, then $c=\frac{1}{2}\operatorname{tr}(S^{-1}\Phi)>0$ and, since $S^{-1}U$ is trace-free, $1=\det\Phi=\det(cI+S^{-1}U)=c^2-|U|_S^2$. (iii) By the Koszul formula, the Levi-Civita connection of $g_{\mathbb H^2}$ is $\nabla_XY=dY(X)-\frac{1}{2}(X\Phi^{-1}Y+Y\Phi^{-1}X)$, and $(\Phi^{-1}\partial_j\Phi)^2=|\partial_j\Phi|_\Phi^2I$ as in (i). Hence, in an orthonormal frame, we compute
\[
 \tau_\gamma(\Phi)=\Delta_\gamma\Phi-\sum_j(\partial_j\Phi)\Phi^{-1}(\partial_j\Phi)
 =\Delta_\gamma\Phi-\tfrac12\sum_j\operatorname{tr}\bigl((\Phi^{-1}\partial_j\Phi)^2\bigr)\Phi,
\]
which is \eqref{eq:matrix-tension}. Equivalently, $\diver_\gamma(\Phi^{-1}d\Phi)=\Phi^{-1}\tau_\gamma(\Phi)$, and $\rho^{-1}\diver_\gamma(\rho\Phi^{-1}d\Phi)=\Phi^{-1}\bigl(\tau_\gamma(\Phi)+d\Phi(\nabla_\gamma\ell)\bigr)$ for $\ell=\log\rho$. Congruence covariance is immediate from the formulas.

\end{proof}

\subsection{Division and Euclidean lifts}\label{subsec:division-lifts}

The singular operators in \eqref{eq:dictionary-operators}
 are ordinary Euclidean Laplacians on radial functions or first angular modes. For functions $\varphi$ and $u$ of $(x,y)$, define the lifts
\[
 (\mathcal L_j\varphi)(Z,{Z'}):=\varphi(|Z|,|{Z'}|),\ \ Z,{Z'}\in\mathbb R^{j+1},\quad
 (\mathcal L_\perp u)(Z,{Z'}):=u(|Z|,|{Z'}|)\,\frac{Z}{|Z|}\otimes\frac{Z'}{|{Z'}|},\ \ Z,{Z'}\in\mathbb R^2.
\]
The function $\mathcal L_j\varphi$ is invariant under $O(j+1)\times O(j+1)$, and $\mathcal L_\perp u$ takes values in $\mathbb R^2\otimes\mathbb R^2$. The components of $Z/|Z|$ and ${Z'}/|{Z'}|$ are first spherical harmonics, of eigenvalue $-1$ on the unit circles, and they produce the terms $-x^{-2}-y^{-2}$ of $D_\perp$. 
At a rod, omit the factor ${Z'}$: $(\mathcal L_j\varphi)(Z,y):=\varphi(|Z|,y)$ and $(\mathcal L_\perp u)(Z,y):=u(|Z|,y)\,Z/|Z|$. At a rod there is one radial block: omit the singular terms $jy^{-1}\partial_y$ and $y^{-2}$, but retain $\partial_y^2$. In the energy formulas, replace $x^jy^j$ by $x^j$, $xy$ by $x$, and the squared sphere-volume factors by one factor.

\begin{lemma}\label{lem:division-lifts}

Let $0<\alpha<1$. All statements concern functions on a rod chart $\{x\geqslant0\}$ (resp. a corner chart $\{x,y\geqslant0\}$) of $\overline\Omega$, and a statement in the variable $x$ holds with $y$ as a parameter, with constants independent of $y$, and conversely.
\begin{enumerate}[label=\textup{(\alph*)}]
\item \emph{Division.} If $r,k,l\geqslant0$ are integers and $u\in C^{r+k+l,\alpha}$ has $\partial_x^iu=0$ on $\{x=0\}$ for $i<k$ and $\partial_y^iu=0$ on $\{y=0\}$ for $i<l$, then $u$ can be written as
\[
 u=x^ky^lw,\qquad\text{with}\qquad \|w\|_{C^{r,\alpha}}\leqslant C\|u\|_{C^{r+k+l,\alpha}}.
\] 
In particular, for even $F$ and odd $u$ in the $x$ variable, we have
\begin{gather}
 \frac{\partial_xF}x=\int_0^1\partial_x^2F(tx,y)\,dt,\label{eq:divided-even}\\
 \Bigl\|\frac{\partial_xF}x\Bigr\|_{C^{r,\alpha}}\leqslant C\|F\|_{C^{r+2,\alpha}},\label{eq:divided-even-norm}\\
 \frac{u}x=\int_0^1\partial_xu(tx,y)\,dt,\qquad\text{and}\qquad
 \frac{u_x}x-\frac{u}{x^2}=\int_0^1t\,u_{xx}(tx,y)\,dt.\label{eq:divided-odd}
\end{gather}
At $x=0$, \eqref{eq:divided-even} gives $x^{-1}\partial_xF=\partial_x^2F$. Consequently, we have
\[
 \Bigl\|\frac{u}x\Bigr\|_{C^{r,\alpha}}\leqslant C\|u\|_{C^{r+1,\alpha}},\qquad\text{and}\qquad
 \Bigl\|\frac{u_x}x-\frac{u}{x^2}\Bigr\|_{C^{r,\alpha}}\leqslant C\|u\|_{C^{r+2,\alpha}}.
\]
\item \emph{Euclidean lifts.} For $j\in\{1,2\}$, even $\varphi$ and $u$ odd in each variable, the lifts $\mathcal L_j\varphi$ and $\mathcal L_\perp u$ above satisfy, for every integer $r\geqslant0$,
\begin{equation}\label{eq:lift-norms}
\begin{aligned}
 C^{-1}\|\varphi\|_{C^{r,\alpha}_{\mathrm{ev}}}\leqslant \|\mathcal L_j\varphi\|_{C^{r,\alpha}}&\leqslant C\|\varphi\|_{C^{r,\alpha}_{\mathrm{ev}}},&\qquad
 \Delta\mathcal L_j\varphi&=\mathcal L_j(D_j\varphi),\\
C^{-1}\|u\|_{C^{r,\alpha}_{\mathrm{odd}}} \leqslant \|\mathcal L_\perp u\|_{C^{r,\alpha}}&\leqslant C\|u\|_{C^{r,\alpha}_{\mathrm{odd}}},&
 \Delta\mathcal L_\perp u&=\mathcal L_\perp(D_\perp u).
\end{aligned}
\end{equation}
For $r\geqslant2$ the Laplacian identities hold pointwise. 
 \item \emph{Energy and weak equations.} For a corner chart $U$, put $\widetilde U_j:=\{(Z,Z')\in\mathbb R^{j+1}\times\mathbb R^{j+1}:(|Z|,|Z'|)\in U\}$. Polar integration gives
\begin{equation}\label{eq:lift-energy}
\begin{aligned}
 \|\mathcal L_j\varphi\|_{H^1(\widetilde U_j)}^2&=|\mathbb S^j|^2\!\int_U x^jy^j\bigl(\varphi^2+|d\varphi|_0^2\bigr)\,dx\,dy,\\
 \|\mathcal L_\perp u\|_{H^1(\widetilde U_1)}^2&=(2\pi)^2\!\int_U xy\bigl(|du|_0^2+(1+x^{-2}+y^{-2})u^2\bigr)\,dx\,dy.
\end{aligned}
\end{equation}

Thus the lifts $\mathcal L_j$ and $\mathcal L_\perp$ extend, isometrically up to the factors $|\mathbb S^j|$ and $2\pi$, to the completions of the smooth even, respectively odd, profiles on the quadrant for the weighted energies in \eqref{eq:lift-energy}, with values in $H^1$. For smooth even $\kappa,c$, writing tildes for their radial lifts {to $\mathbb R^{j+1}\times\mathbb R^{j+1}$},
\begin{equation}\label{eq:weak-lift}
 (\Delta+\langle d\widetilde\kappa,d\,\cdot\,\rangle-\widetilde c)\,\mathcal L_j\varphi
 =\mathcal L_j\bigl(D_j\varphi+\langle d\kappa,d\varphi\rangle_0-c\varphi\bigr)
\end{equation}
holds also weakly: tests on the quadrant use the measure $x^jy^je^\kappa\,dx\,dy$, and tests after lifting use $e^{\widetilde\kappa}dZ\,d{Z'}$. The identity $\Delta\mathcal L_\perp u=\mathcal L_\perp(D_\perp u)$ has the analogous weak interpretation in the energy norm in \eqref{eq:lift-energy}.
\end{enumerate}
\end{lemma}

\begin{proof}
(a) Each division by $x$ is the formula $u/x=\int_0^1\partial_xu(tx,y)\,dt$, and $\|v(t\,\cdot,\cdot)\|_{C^{r,\alpha}}\leqslant\|v\|_{C^{r,\alpha}}$ for $0\leqslant t\leqslant1$ and similarly in $y$. For even $F$, apply it to the odd function $\partial_xF$. For the last identity of \eqref{eq:divided-odd}, $f(t):=tx\,\partial_xu(tx,y)-u(tx,y)$ satisfies $f(0)=0$ and $f'(t)=tx^2\,\partial_x^2u(tx,y)$.

(b) Consider first one radial variable $x=|Z|$, with $z$ denoting the remaining coordinates, and put $n:=Z/|Z|$ for $Z\neq0$. A bounded extension from a fixed signed coordinate rectangle, followed by averaging with the prescribed parity and an even cutoff, reduces the estimates to compactly supported profiles, which we first assume smooth. Their lifts are then smooth by \cite{Whi43} and part~(a).

Scalar lifting preserves $C^{0,\alpha}$, since $\bigl||Z|-|\widetilde Z|\bigr|\leqslant|Z-\widetilde Z|$. If moreover $v(0,z)=0$ and $H$ is smooth on the unit sphere, then, extending the product by $0$ at $Z=0$, we have
\[
 \|H(n)\,v(|Z|,z)\|_{C^{0,\alpha}}\leqslant C\|H\|_{C^1}\|v\|_{C^{0,\alpha}}:
\]
for two points off the axis, at distance $d$ and with smaller radius $\varsigma$, the change of the angular factor contributes at most $C[v]_\alpha\varsigma^\alpha\min\{1,d/\varsigma\}\leqslant C[v]_\alpha d^\alpha$. This proves the estimates of order $0$. For order $1$, use
\[
 \partial_{Z_i}\varphi(|Z|,z)=n_i\,\varphi_x(|Z|,z),\qquad\text{and}\qquad
 \partial_{Z_i}\bigl(n_a\,u(|Z|,z)\bigr)=\delta_{ia}v(|Z|,z)+n_in_a\,w(|Z|,z),
\]
where $v:=u/x=\int_0^1u_x(tx,\cdot)\,dt$ and $w:=u_x-v$. Here $\varphi_x$ and $w$ vanish at $x=0$, and $\|v\|_{C^{0,\alpha}}+\|w\|_{C^{0,\alpha}}\leqslant C\|u\|_{C^{1,\alpha}}$. Tangential derivatives commute with the lifts and preserve parity, so the estimates of order $0$ apply to them.

Write $\mathcal L$ for either lift and $D$ for its one-block operator, $\partial_x^2+jx^{-1}\partial_x+\Delta_z$, respectively $\partial_x^2+x^{-1}\partial_x-x^{-2}+\Delta_z$. The polar formula for the Laplacian and $\Delta_{\mathbb S^1}n_a=-n_a$ give $\Delta\mathcal Lf=\mathcal L(Df)$. Part~(a) shows that $D$ preserves parity and $\|Df\|_{C^{r-2,\alpha}}\leqslant C\|f\|_{C^{r,\alpha}}$. For $r\geqslant2$, induction on $r$ and the interior Schauder estimate \cite[Theorem~6.2]{GT01} on unit balls give
\[
 \|\mathcal Lf\|_{C^{r,\alpha}}\leqslant C\bigl(\|\mathcal L(Df)\|_{C^{r-2,\alpha}}+\|\mathcal Lf\|_{C^0}\bigr)\leqslant C\|f\|_{C^{r,\alpha}}.
\]
A parity-preserving mollification removes the smoothness assumption, by compactness in $C^{r,\alpha'}$, $\alpha'<\alpha$, and lower semicontinuity of the H\"older seminorm. The corner case follows by applying the one-block result in each variable, and restriction to the signed coordinate planes gives the reverse bounds. The Laplacian identities pass to the limit for $r\geqslant2$.

(c) The identities \eqref{eq:lift-energy} follow from $dZ\,d{Z'}=x^jy^j\,dx\,dy\,d\sigma_{{\mathbb S^j}}\,d\sigma_{{\mathbb S^j}}$ and $|\nabla\mathcal L_\perp u|^2=|du|_0^2+(x^{-2}+y^{-2})u^2$. Let $P$ be the average over $O(j+1)\times O(j+1)$ for the scalar lift, and over the action of $O(2)\times O(2)$ for the tensor lift. A smooth averaged test function is the lift of an even, respectively odd, profile. For the tensor lift, it has the form $Z_aZ'_b\varphi_0(|Z|,|Z'|)$ with $\varphi_0$ smooth and even. By \eqref{eq:lift-energy}, the weak pairing with such a lift is the weighted pairing on the quadrant. Since the lifted operator $\widetilde L$ and forcing $\widetilde F$ are equivariant, for the lift $\Xi$ of a weak solution and every compactly supported smooth test function $\vartheta$, we have
\[
 \langle\widetilde L\Xi-\widetilde F,\vartheta\rangle=\langle\widetilde L\Xi-\widetilde F,P\vartheta\rangle=0,
\]
and \eqref{eq:lift-energy} extends this to the energy completions. 
\end{proof}

\subsection{Regularized equations at rods and corners}\label{subsec:regularized-equations}

We use the coefficients $\widehat\ell,\widehat s,\widehat\omega,\widehat a$ of \eqref{eq:regularized-corner-matrix} and verify the local formulas used in Lemma~\ref{lem:reduced-map} and in Corollary~\ref{cor:fixed-set-regularity}.

\paragraph{Connection form and corner derivatives.} For $q$ near $0$ in the spaces \eqref{eq:side-domain-norm} and \eqref{eq:corner-domain}, the coefficients \eqref{eq:matrix-chart-differential} have the parities of Remark~\ref{rem:conformal-parity} and satisfy the angle equalities \eqref{eq:angle-equalities}, by  Lemma~\ref{lem:regularized-chart}.  Put $r_c:=|x+iy|$. First, since $\omega=(y/x)\widehat\omega$ and $s=\log(y/x)+\widehat s$, the connection form $A_\Phi:=e^{-s}d\omega$ of the frame \eqref{eq:target-rod-frame} has the regularized expression
\begin{equation}\label{eq:connection-form-local}
 A_\Phi=e^{-\widehat s}\Bigl[\Bigl(\partial_x\widehat\omega-\frac{\widehat\omega}x\Bigr)dx+\Bigl(\partial_y\widehat\omega+\frac{\widehat\omega}y\Bigr)dy\Bigr],
\end{equation}
and $A_\Phi=e^{-\bar s}\bigl[(\partial_x\bar\omega-\bar\omega/x)\,dx+\partial_y\bar\omega\,dy\bigr]$ at a rod. The formula $\widehat\omega/x=\int_0^1\partial_x\widehat\omega(tx,y)\,dt$ of Lemma~\ref{lem:division-lifts}(a), and its analogue in $y$, then give, for $k\geqslant2$,
\begin{equation}\label{eq:connection-regularity}
 (\widehat s,\widehat\omega)\in C^{k,\alpha}_{\mathrm{ev}}\times C^{k,\alpha}_{\mathrm{odd}}
 \ \Longrightarrow\
 A_\Phi\in C^{k-1,\alpha},\qquad\text{and}\qquad |A_\Phi|_0=O(r_c),
\end{equation}
and $|A_\Phi|_0=O(x)$ at a rod: the coefficient of $dx$ in \eqref{eq:connection-form-local} is odd in $y$, the coefficient of $dy$ is odd in $x$, and both are $C^{k-1,\alpha}$. No division by $xy$ is needed. Second, differentiating \eqref{eq:angle-equalities} once and then twice along the two axes, and using that $\widehat\ell$ and $\widehat s$ are even, gives the corner derivatives
\begin{equation}\label{eq:corner-jets}
 d\widehat a(c)=0,\qquad
 2\partial_x^2\widehat a(c)=\partial_x^2(\widehat\ell+\widehat s)(c),\qquad\text{and}\qquad
 2\partial_y^2\widehat a(c)=\partial_y^2(\widehat\ell-\widehat s)(c).
\end{equation}

\begin{lemma}\label{lem:regularized-equations}

In a corner chart, with the connection form \eqref{eq:connection-form-local}, substituting $\rho=xye^{\widehat\ell}$, $s=\log(y/x)+\widehat s$ and $\omega=(y/x)\widehat\omega$ in \eqref{eq:conic-map} gives
\begin{equation}\label{cmp:nonlinear-first-residuals}
\begin{aligned}
 e^{2c_\gamma}\mathcal E_1&=D_2\widehat\ell+|d\widehat\ell|_0^2+6e^{2\widehat a},\\
 e^{2c_\gamma}\mathcal H_s&=D_1\widehat s+\langle d\widehat\ell,d\widehat s\rangle_0-\frac{\partial_x\widehat\ell}x+\frac{\partial_y\widehat\ell}y+|A_\Phi|_0^2,\\
 e^{2c_\gamma}\mathcal H_\perp&=e^{-\widehat s}D_\perp\widehat\omega+\langle d\widehat\ell-2d\widehat s,A_\Phi\rangle_0,\\
 e^{2c_\gamma}\mathcal E_3&=\Delta_0\widehat a-\frac{1}{2}\Bigl(\frac{\partial_x(\widehat\ell+\widehat s)}x+\frac{\partial_y(\widehat\ell-\widehat s)}y\Bigr)
   +\frac{1}{4}\bigl(|d\widehat s|_0^2-|d\widehat\ell|_0^2+|A_\Phi|_0^2\bigr).
\end{aligned}
\end{equation}
The right-hand side of the third line is odd in each variable, and the four right-hand sides are smooth maps of $q\in\mathcal U_0$ into $C^{0,\alpha}$, respectively $C^{0,\alpha}_{\mathrm{odd}}$ for the third line. The rod case follows  with the rod convention of Section~ \ref{subsec:division-lifts}.
\end{lemma}

\begin{proof}

For the third line, the singular part of $d\ell-2ds$ is $3\,dx/x-dy/y$, and
\[
 \frac{x}y\Bigl(\Delta_0+\frac{3}x\partial_x-\frac{1}y\partial_y\Bigr)\Bigl(\frac{y}x\widehat\omega\Bigr)=D_\perp\widehat\omega,\qquad\text{and}\qquad A_\Phi=e^{-\widehat s}\bigl[d\widehat\omega+\widehat\omega(-dx/x+dy/y)\bigr].
\]
By \eqref{eq:divided-odd} and Lemma~\ref{lem:division-lifts}(a), the right-hand sides of \eqref{cmp:nonlinear-first-residuals} are smooth maps to $C^{0,\alpha}$. In \eqref{eq:connection-form-local}, the coefficients of $A_\Phi$ have parities (even, odd) and (odd, even), so $|A_\Phi|^2$ is even. The transverse equation $e^{2c_\gamma}\mathcal H_\perp$ is odd in each variable. By Lemma~\ref{lem:matrix-calculus}(iii), $\Pi_{\Phi_*}$ has the same algebraic formula in the rescaled matrices, preserving even diagonal and odd off-diagonal entries. Hence $\mathcal E_2$ takes values in the spaces of \eqref{eq:conic-spaces}. There is no divided derivative of $\widehat a$.

\end{proof}

\subsection{The lifted target equation}\label{subsec:lifted-regularity}

\begin{lemma}\label{lem:lifted-target}
Let $U$ be a rod (resp. corner) chart for the fixed smooth background of Lemma~\ref{lem:shifted-regularity}, and $U'\subset U$ a chart whose closure in the closed half-plane (resp. quadrant) is a compact subset of $U$. Let $\xi$, with $\mathfrak a_2(\xi;U)<+\infty$ in \eqref{cmp:positive-forms}, be a weak solution of $A_2\xi=F_2$ on $U$ in this local energy class, with $F_s\in C^{0,\alpha}_{\mathrm{ev}}(U)$ and $F_\perp\in C^{0,\alpha}_{\mathrm{odd}}(U)$. Then we have
\[
 \|\xi_s\|_{C^{2,\alpha}_{\mathrm{ev}}(U')}+\|\xi_\perp\|_{C^{2,\alpha}_{\mathrm{odd}}(U')}
 \leqslant C\bigl(\|F_s\|_{C^{0,\alpha}_{\mathrm{ev}}(U)}+\|F_\perp\|_{C^{0,\alpha}_{\mathrm{odd}}(U)}+\|\xi\|_{L^2(U,\rho\,dA_\gamma)}\bigr),
\]
where $C$ depends only on $U$, $U'$ and the background.
\end{lemma}

\begin{proof}
Write $\xi=\xi_se_s+\xi_\perp e_\perp$ in the frame \eqref{eq:target-rod-frame}, whose connection form $A_\Phi:=e^{-s}d\omega$ has the regularized expression \eqref{eq:connection-form-local}. Then we have
\[
 \nabla\xi=(d\xi_s+\xi_\perp A_\Phi)\,e_s+(d\xi_\perp-\xi_sA_\Phi)\,e_\perp,\qquad\text{and}\qquad d\Phi=ds\,e_s+A_\Phi\,e_\perp.
\]
The background map equation gives $\operatorname{div}_{\rho\,dA_\gamma}A_\Phi^\sharp=\langle A_\Phi,ds\rangle_\gamma$. With it, \eqref{eq:weighted-Jacobi} takes the symmetric form
\begin{equation}\label{cmp:frame-jacobi}
 A_2=\begin{pmatrix}L_0-2|A_\Phi|_\gamma^2-1&T_0^*\\ T_0&L_0-|ds|_\gamma^2-|A_\Phi|_\gamma^2-1\end{pmatrix},
\end{equation}
in the basis $(\xi_s,\xi_\perp)$, with $L_0:=\Delta_\gamma+\langle d\ell,d\,\cdot\,\rangle_\gamma$ and $T_0:=-2\langle A_\Phi,d\,\cdot\,\rangle_\gamma$, where $T_0^*$ is the formal adjoint of $T_0$ for $\rho\,dA_\gamma$. The zeroth-order coefficients are those of $\mathfrak a_2$ in \eqref{cmp:positive-forms}.

Write $x=|Z|$, $y=|{Z'}|$ for $Z,{Z'}\in\mathbb R^2$, and lift the two components of $\xi$ from the quadrant to a scalar and a tensor field on $\widetilde U:=\{(Z,Z')\in\mathbb R^2\times\mathbb R^2:(|Z|,|Z'|)\in U\}$ by
\[
 \Xi_s:=\mathcal L_1\xi_s,\qquad \Xi_\perp:=\mathcal L_\perp\xi_\perp,\qquad\text{that is}\quad
 (\Xi_\perp)_{ab}(Z,{Z'})=\frac{Z_a}x\frac{{Z'}_b}y\,\xi_\perp(x,y).
\]
For the fixed smooth background, $\mathfrak t:=\widehat\omega/(xy)$ is smooth and even, by Corollary~\ref{cor:fixed-set-regularity} and Lemma~\ref{lem:division-lifts}(a) with $k=l=1$, and we write
\[
 s=\log(y/x)+\widehat s,\qquad P:=e^{-\widehat s}\frac{\partial_x\mathfrak t}x,\qquad Q:=e^{-\widehat s}(2\mathfrak t+y\partial_y\mathfrak t),\qquad\text{and}\qquad A_\Phi=x^2yP\,dx+xQ\,dy,
\]
with $P$ and $Q$ smooth and even.
Direct computations in the variables $(Z,{Z'})$ give
\begin{equation}\label{eq:angular-energy}
\begin{aligned}
 |\nabla\Xi_\perp|^2&=|d\xi_\perp|_0^2+(x^{-2}+y^{-2})\xi_\perp^2,\\
 |ds|_0^2-x^{-2}-y^{-2}&=|d\widehat s|_0^2-2\frac{\partial_x\widehat s}x+2\frac{\partial_y\widehat s}y.
\end{aligned}
\end{equation}
The formula for $\nabla\xi$ above, the identity $|\xi\wedge d\Phi|=|\xi_sA_\Phi-\xi_\perp ds|$, the boundedness of $A_\Phi$ and \eqref{eq:angular-energy} give $\|\Xi\|_{H^1(\widetilde U)}^2\leqslant C\,\mathfrak a_2(\xi;U)$. The right-hand side of the second identity of \eqref{eq:angular-energy} is smooth by evenness and \eqref{eq:divided-even}. In the transverse component of $A_2\xi=F_2$, the operator $D_1-|ds|_0^2$ therefore lifts to the Euclidean Laplacian on $\mathbb R^2\times\mathbb R^2$ plus a smooth potential.

By \eqref{cmp:frame-jacobi}, the off-diagonal entries of $A_2$ are $T_0=-2\langle A_\Phi,d\,\cdot\,\rangle_\gamma$, from $\xi_s$ to $\xi_\perp$, and its formal adjoint $T_0^*$ for $\rho\,dA_\gamma$. After lifting, $e^{2c_\gamma}T_0=-2\langle A_\Phi,d\,\cdot\,\rangle_0$ becomes
\begin{equation}\label{app:lifted-coupling}
\begin{aligned}
 &\bigl[\mathcal L_\perp\bigl(-2\langle A_\Phi,d\xi_s\rangle_0\bigr)\bigr]_{ab}=-2Z_a\bigl[P\,Z'_b\,(Z\cdot\nabla_Z\Xi_s)+Q\,\partial_{Z'_b}\Xi_s\bigr],
\end{aligned}
\end{equation}
whose coefficients are smooth. The measure $\rho\,dA_\gamma$ lifts to a constant multiple of the smooth positive density $e^{\widehat\ell+2c_\gamma}dZ\,dZ'$, so the formal adjoint $T_0^*$ for this measure also has smooth coefficients. The remaining diagonal terms are smooth.
Thus the lifted equation has the form
\[
 \Delta\Xi+B^a\partial_a\Xi+C\Xi=\bigl(\mathcal L_1(e^{2c_\gamma}F_s),\mathcal L_\perp(e^{2c_\gamma}F_\perp)\bigr),\qquad\text{with }B^a,C\ \text{smooth}.
\]
By Lemma~\ref{lem:division-lifts}(c), the lifted system holds weakly across the axes.

By Lemma~\ref{lem:division-lifts}(b), its right-hand side is $C^{0,\alpha}$, and interior elliptic regularity for this system with scalar Laplacian principal part, followed by the Schauder estimate \cite[Chapters~6 and~9]{GT01}, gives the asserted bound.
Restriction to $Z=(x,0)$, ${Z'}=(y,0)$, with signed $x,y$, gives the even extension of $\xi_s$ and the odd extension of $\xi_\perp$.
\end{proof}

\subsection{Weak removability}\label{subsec:weighted-calculus}

\begin{lemma}\label{lem:B-cutoffs}Let $p_0$ be a point of a two-dimensional chart with the metric $\gamma_\beta$, for $\beta$ in a compact subset of $(0,+\infty)$. Let $A$ be bounded and uniformly elliptic, $f\in L^1_{\mathrm{loc}}$ and $u\in H^1_{\mathrm{loc}}$. If $\diver(A\,du)=f$ holds weakly off $p_0$, then it holds weakly across $p_0$. The same conclusion holds if the measure is multiplied by a density bounded above and below by positive constants.
\end{lemma}

\begin{proof}Let $\varrho$ be the radial coordinate of a cone chart centered at $p_0$, which is $R$ for the central chart. The logarithmic cutoffs $\chi_\varepsilon:=\min\bigl\{1,\max\bigl\{0,\lambda_\varepsilon\bigr\}\bigr\}$, with $\lambda_\varepsilon:=\log(\varrho/\varepsilon^2)/\log(1/\varepsilon)$, vanish near $p_0$, converge to $1$ off $p_0$, and satisfy, by polar integration, for every compactly supported smooth $\varphi$,
\[
 \int\varphi^2|d\chi_\varepsilon|^2\leqslant\frac{C_\varphi}{\log(1/\varepsilon)},\qquad\text{and}\qquad
 \Bigl|\int\varphi\langle A\,du,d\chi_\varepsilon\rangle\Bigr|\leqslant C\|du\|_{L^2(\operatorname{supp}\varphi)}\|\varphi\,d\chi_\varepsilon\|_{L^2}\longrightarrow0.
\]
Testing the equation with $\chi_\varepsilon\varphi$ and dominated convergence give the weak equation across $p_0$. The cutoff estimate persists for such a density, and radial cutoffs preserve the finite symmetries.

\end{proof}

\bigskip\noindent\footnotesize
\textsc{MIT, Dept.\ of Math., 77 Massachusetts Avenue, Cambridge, MA 02139-4307}\par
\textit{Email address:} \texttt{ozuch@mit.edu}\par

\begin{thebibliography}{MRS15}


\bibitem[And89]{And89}
M.~T.~Anderson,
\emph{Ricci curvature bounds and Einstein metrics on compact manifolds},
J. Amer. Math. Soc. \textbf{2} (1989), no.~3, 455--490.
\href{https://doi.org/10.1090/S0894-0347-1989-0999661-1}{doi:10.1090/S0894-0347-1989-0999661-1}.

\bibitem[And06]{And06}
M.~T.~Anderson,
\emph{Dehn filling and Einstein metrics in higher dimensions},
J. Differential Geom. \textbf{73} (2006), no.~2, 219--261.
\href{https://doi.org/10.4310/jdg/1146169911}{doi:10.4310/jdg/1146169911}.

\bibitem[And10]{And10}
M.~T.~Anderson,
\emph{A survey of Einstein metrics on $4$-manifolds},
Handbook of Geometric Analysis, No. 3, Adv. Lect. Math. \textbf{14},
International Press, 2010, 1--39.

\bibitem[AL13]{AL13}
M.~F.~Atiyah and C.~LeBrun,
\emph{Curvature, cones, and characteristic numbers},
Math. Proc. Cambridge Philos. Soc. \textbf{155} (2013), 13--37.
\href{https://arxiv.org/abs/1203.6389}{arXiv:1203.6389}.

\bibitem[Aub78]{Aub78}
T.~Aubin,
\emph{\'Equations du type Monge--Amp\`ere sur les vari\'et\'es k\"ahl\'eriennes compactes},
Bull. Sci. Math. (2) \textbf{102} (1978), no.~1, 63--95.

\bibitem[Bam12]{Bam12}
R.~H.~Bamler,
\emph{Construction of Einstein metrics by generalized Dehn filling},
J. Eur. Math. Soc. \textbf{14} (2012), no.~3, 887--909.
\href{https://doi.org/10.4171/JEMS/321}{doi:10.4171/JEMS/321}.

\bibitem[BKN89]{BKN89}
S.~Bando, A.~Kasue, and H.~Nakajima,
\emph{On a construction of coordinates at infinity on manifolds
with fast curvature decay and maximal volume growth},
Invent. Math. \textbf{97} (1989), 313--349.
\href{https://doi.org/10.1007/BF01389045}{doi:10.1007/BF01389045}.

\bibitem[BS99]{BS99}
V.~V.~Batyrev and E.~N.~Selivanova,
\emph{Einstein--K\"ahler metrics on symmetric toric Fano manifolds},
J. Reine Angew. Math. \textbf{512} (1999), 225--236.
\href{https://arxiv.org/abs/math/9901001}{arXiv:math/9901001}.

\bibitem[BM07]{BM07}
Ya.~V.~Bazaikin and I.~V.~Matvienko,
\emph{On the four-dimensional $T^2$-manifolds of positive
Ricci curvature},
Siberian Math. J. \textbf{48} (2007), no.~5, 778--783.
\href{https://doi.org/10.1007/s11202-007-0080-7}
{doi:10.1007/s11202-007-0080-7}.

\bibitem[BKMR21]{BKMR21}
J.~Bertrand, C.~Ketterer, I.~Mondello, and T.~Richard,
\emph{Stratified spaces and synthetic Ricci curvature bounds},
Ann. Inst. Fourier \textbf{71} (2021), no.~1, 123--173.
\href{https://arxiv.org/abs/1804.08870}{arXiv:1804.08870}.

\bibitem[Bes87]{Bes87}
A.~L.~Besse,
\emph{Einstein Manifolds}, Springer, 1987.
\href{https://doi.org/10.1007/978-3-540-74311-8}{doi:10.1007/978-3-540-74311-8}.

\bibitem[Biq08]{Biq08}
O.~Biquard,
\emph{Polycopi\'e on Differential Geometry and Global Analysis},
lecture notes, UPMC, February 2008, Chapter~3.
\url{https://webusers.imj-prg.fr/~olivier.biquard/dgga2007.pdf}.

\bibitem[BG22]{BG22}
O.~Biquard and H.~Guenancia,
\emph{Degenerating K\"ahler--Einstein cones, locally symmetric cusps, and the Tian--Yau metric},
Invent. Math. \textbf{230} (2022), 1101--1163.
\href{https://arxiv.org/abs/2108.13318}{arXiv:2108.13318}.


\bibitem[Biq13]{Biq13}
O.~Biquard,
\emph{D\'esingularisation de m\'etriques d'Einstein. I},
Invent. Math. \textbf{192} (2013), no.~1, 197--252.


\bibitem[Boh98]{Boh98}
C.~B\"ohm,
\emph{Inhomogeneous Einstein metrics on low-dimensional spheres and
other low-dimensional spaces}, Invent. Math. \textbf{134} (1998), 145--176.
\href{https://doi.org/10.1007/s002220050261}{doi:10.1007/s002220050261}.

\bibitem[BH99]{BH99}
M.~R.~Bridson and A.~Haefliger,
\emph{Metric Spaces of Non-Positive Curvature},
Grundlehren der Mathematischen Wissenschaften, vol.~319, Springer, 1999.

\bibitem[CS04]{CS04}
D.~M.~J.~Calderbank and M.~A.~Singer,
\emph{Einstein metrics and complex singularities},
Invent. Math. \textbf{156} (2004), 405--443.
\href{https://arxiv.org/abs/math/0206229}{arXiv:math/0206229}.


\bibitem[CLSZ07]{CLSZ07}
C.~Charmousis, D.~Langlois, D.~Steer, and R.~Zegers,
\emph{Rotating spacetimes with a cosmological constant},
J. High Energy Phys. \textbf{2007}, no.~2, 064.

\bibitem[CC97]{CC97}
J.~Cheeger and T.~H.~Colding,
\emph{On the structure of spaces with Ricci curvature bounded below. I},
J. Differential Geom. \textbf{46} (1997), no.~3, 406--480.

\bibitem[CGT82]{CGT82}
J.~Cheeger, M.~Gromov, and M.~Taylor,
\emph{Finite propagation speed, kernel estimates for functions of the Laplace operator, and the geometry of complete Riemannian manifolds},
J. Differential Geom. \textbf{17} (1982), no.~1, 15--53.

\bibitem[CN15]{CN15}
J.~Cheeger and A.~Naber, \emph{Regularity of Einstein manifolds and the codimension $4$ conjecture}, Ann. of Math. (2) \textbf{182} (2015), no.~3, 1093--1165.

\bibitem[CT06]{CT06}
J.~Cheeger and G.~Tian,
\emph{Curvature and injectivity radius estimates for Einstein $4$-manifolds},
J. Amer. Math. Soc. \textbf{19} (2006), no.~2, 487--525.
\href{https://doi.org/10.1090/S0894-0347-05-00511-4}{doi:10.1090/S0894-0347-05-00511-4}.

\bibitem[CZ17]{CZ17}
Q.~Chen and G.~Zhao,
\emph{A Schwarz lemma for $V$-harmonic maps and their applications},
Bull. Aust. Math. Soc. \textbf{96} (2017), no.~3, 504--512.
\href{https://doi.org/10.1017/S000497271700051X}{doi:10.1017/S000497271700051X}.

\bibitem[CDS15a]{CDS15a}
X.~Chen, S.~Donaldson, and S.~Sun,
\emph{K\"ahler--Einstein metrics on Fano manifolds, I:
approximation of metrics with cone singularities},
J. Amer. Math. Soc. \textbf{28} (2015), no.~1, 183--197.
\href{https://doi.org/10.1090/S0894-0347-2014-00799-2}{doi:10.1090/S0894-0347-2014-00799-2}.

\bibitem[CDS15b]{CDS15b}
X.~Chen, S.~Donaldson, and S.~Sun,
\emph{K\"ahler--Einstein metrics on Fano manifolds, II:
limits with cone angle less than $2\pi$},
J. Amer. Math. Soc. \textbf{28} (2015), no.~1, 199--234.
\href{https://doi.org/10.1090/S0894-0347-2014-00800-6}{doi:10.1090/S0894-0347-2014-00800-6}.

\bibitem[CDS15c]{CDS15c}
X.~Chen, S.~Donaldson, and S.~Sun,
\emph{K\"ahler--Einstein metrics on Fano manifolds, III:
limits as cone angle approaches $2\pi$ and completion of the main proof},
J. Amer. Math. Soc. \textbf{28} (2015), no.~1, 235--278.
\href{https://doi.org/10.1090/S0894-0347-2014-00801-8}{doi:10.1090/S0894-0347-2014-00801-8}.

\bibitem[CLW08]{CLW08}
X.~Chen, C.~LeBrun, and B.~Weber,
\emph{On conformally K\"ahler, Einstein
manifolds}, J. Amer. Math. Soc. \textbf{21} (2008), 1137--1168.

\bibitem[CT10]{CT10}
Y.~Chen and E.~Teo,
\emph{Rod-structure classification of gravitational instantons with $U(1)\times U(1)$ isometry},
Nuclear Phys. B \textbf{838} (2010), no.~1--2, 207--237.

\bibitem[Chr08]{Chr08}
P.~T.~Chru\'sciel,
\emph{Mass and angular-momentum inequalities for axi-symmetric initial data sets. I. Positivity of mass},
Ann. Physics \textbf{323} (2008), no.~10, 2566--2590.

\bibitem[DHPW23]{DHPW23}
X.~Dai, S.~Honda, J.~Pan, and G.~Wei,
\emph{Singular Weyl's law with Ricci curvature bounded below},
Trans. Amer. Math. Soc. Ser. B \textbf{10} (2023), 1212--1253.
\href{https://arxiv.org/abs/2208.13962}{arXiv:2208.13962}.

\bibitem[DJ91]{DJ91}
M.~W.~Davis and T.~Januszkiewicz,
\emph{Convex polytopes, Coxeter orbifolds and torus actions},
Duke Math. J. \textbf{62} (1991), no.~2, 417--451.

\bibitem[Der83]{Der83}
A.~Derdzi\'nski,
\emph{Self-dual K\"ahler manifolds and Einstein manifolds of dimension four},
Compositio Math. \textbf{49} (1983), no.~3, 405--433.


\bibitem[DGSW23]{DGSW23}
O.~J.~C.~Dias, G.~W.~Gibbons, J.~E.~Santos, and B.~Way,
\emph{Static Black Binaries in de Sitter Space},
Phys. Rev. Lett. \textbf{131} (2023), 131401.
\href{https://doi.org/10.1103/PhysRevLett.131.131401}{doi:10.1103/PhysRevLett.131.131401}.

\bibitem[Don08]{Don08}
S.~K.~Donaldson,
\emph{Extremal metrics on toric surfaces: a continuity method},
J. Differential Geom. \textbf{79} (2008), no.~3, 389--432.

\bibitem[Don12]{Don12}
S.~K.~Donaldson,
\emph{K\"ahler metrics with cone singularities along a divisor},
Essays in Mathematics and its Applications, Springer, 2012, 49--79.

\bibitem[ES64]{ES64}
J.~Eells and J.~H.~Sampson,
\emph{Harmonic mappings of Riemannian manifolds},
Amer. J. Math. \textbf{86} (1964), 109--160.

\bibitem[ER02]{ER02}
R.~Emparan and H.~S.~Reall,
\emph{Generalized Weyl solutions},
Phys. Rev. D \textbf{65} (2002), 084025.
\href{https://arxiv.org/abs/hep-th/0110258}{arXiv:hep-th/0110258}.

\bibitem[FP20]{FP20}
J.~Fine and B.~Premoselli,
\emph{Examples of compact Einstein four-manifolds
with negative curvature}, J. Amer. Math. Soc. \textbf{33} (2020), 991--1038.

\bibitem[Fre82]{Fre82}
M.~H.~Freedman,
\emph{The topology of four-dimensional manifolds},
J. Differential Geom. \textbf{17} (1982), 357--453.
\href{https://doi.org/10.4310/jdg/1214437136}{doi:10.4310/jdg/1214437136}.

\bibitem[GH77]{GH77}
G.~W.~Gibbons and S.~W.~Hawking,
\emph{Cosmological event horizons, thermodynamics, and particle creation},
Phys. Rev. D \textbf{15} (1977), no.~10, 2738--2751.

\bibitem[GT01]{GT01}
D.~Gilbarg and N.~S.~Trudinger,
\emph{Elliptic Partial Differential Equations of Second Order},
Classics in Mathematics, Springer, 2001.

\bibitem[Gri85]{Gri85}
P.~Grisvard, \emph{Elliptic Problems in Nonsmooth Domains},
Monographs and Studies in Mathematics, vol.~24, Pitman, 1985;
reprinted in Classics in Applied Mathematics, vol.~69, SIAM, 2011.

\bibitem[GL80]{GL80}
M.~Gromov and H.~B.~Lawson, Jr.,
\emph{The classification of simply connected manifolds of positive
scalar curvature}, Ann. of Math. (2) \textbf{111} (1980), no.~3, 423--434.
\href{https://doi.org/10.2307/1971103}{doi:10.2307/1971103}.

\bibitem[GK73]{GK73}
K.~Grove and H.~Karcher,
\emph{How to conjugate $C^1$-close group actions}, Math. Z. \textbf{132} (1973), 11--20.
\href{https://doi.org/10.1007/BF01214029}{doi:10.1007/BF01214029}.

\bibitem[GZ12]{GZ12}
K.~Grove and W.~Ziller,
\emph{Polar manifolds and actions},
J. Fixed Point Theory Appl. \textbf{11} (2012),
no.~2, 279--313.
\href{https://doi.org/10.1007/s11784-012-0087-y}
{doi:10.1007/s11784-012-0087-y}.

\bibitem[GH25]{GH25}
H.~Guenancia and U.~Hamenst\"adt,
\emph{K\"ahler--Einstein metrics of negative curvature}, preprint (2025).
\href{https://arxiv.org/abs/2503.02838}{arXiv:2503.02838}.

\bibitem[HJ24]{HJ24}
U.~Hamenst\"adt and F.~J\"ackel,
\emph{Negatively curved Einstein metrics on Gromov--Thurston manifolds},
J. Eur. Math. Soc., to appear;
\href{https://arxiv.org/abs/2411.12956}{arXiv:2411.12956} (2024).

\bibitem[Hit74]{Hit74}
N.~Hitchin,
\emph{Compact four-dimensional Einstein manifolds},
J. Differential Geom. \textbf{9} (1974), 435--441.


\bibitem[HO25]{HO}
Y.~Huang and T.~Ozuch,
\emph{Regularity of Einstein $5$-manifolds via $4$-dimensional gap theorems}, preprint (2025).
\href{https://arxiv.org/abs/2512.21317}{arXiv:2512.21317}.


\bibitem[JMR16]{JMR16}
T.~Jeffres, R.~Mazzeo, and Y.~A.~Rubinstein,
\emph{K\"ahler--Einstein metrics with edge singularities},
Ann. of Math. \textbf{183} (2016), 95--176.
\href{https://doi.org/10.4007/annals.2016.183.1.3}{doi:10.4007/annals.2016.183.1.3}.

\bibitem[KRWY22]{KRWY22}
M.~Khuri, M.~Reiris, G.~Weinstein, and S.~Yamada,
\emph{Gravitational solitons and complete Ricci flat Riemannian manifolds of infinite topological type},
Pure Appl. Math. Q. \textbf{20} (2024), no.~4, 1895--1921. \href{https://doi.org/10.4310/PAMQ.2024.v20.n4.a12}{doi:10.4310/PAMQ.2024.v20.n4.a12}; \href{https://arxiv.org/abs/2204.08048}{arXiv:2204.08048}.

\bibitem[KN94]{KN94}
D.~Korotkin and H.~Nicolai,
\emph{A periodic analog of the Schwarzschild solution},
preprint (1994), \href{https://arxiv.org/abs/gr-qc/9403029}{arXiv:gr-qc/9403029}.

\bibitem[Kot18]{Kot18}
F.~Kottler,
\emph{\"Uber die physikalischen Grundlagen der Einsteinschen Gravitationstheorie},
Ann. Phys. (4) \textbf{56} (1918), 401--462.
\href{https://doi.org/10.1002/andp.19183611402}{doi:10.1002/andp.19183611402}.

\bibitem[LeB96]{LeB96}
C.~LeBrun,
\emph{Four-manifolds without Einstein metrics},
Math. Res. Lett. \textbf{3} (1996), no.~2, 133--147.
\href{https://doi.org/10.4310/MRL.1996.v3.n2.a1}{doi:10.4310/MRL.1996.v3.n2.a1}.

\bibitem[LeB01]{LeB01}
C.~LeBrun,
\emph{Ricci curvature, minimal volumes, and Seiberg--Witten theory},
Invent. Math. \textbf{145} (2001), 279--316.

\bibitem[LeB12]{LeB12}
C.~LeBrun,
\emph{On Einstein, Hermitian $4$-manifolds},
J. Differential Geom. \textbf{90} (2012), 277--302.

\bibitem[LeB21]{LeB21}
C.~LeBrun,
\emph{Einstein metrics, conformal curvature, and anti-holomorphic involutions},
Ann. Math. Qu\'e. \textbf{45} (2021), no.~2, 391--405.
\href{https://doi.org/10.1007/s40316-020-00154-2}{doi:10.1007/s40316-020-00154-2}.

\bibitem[LeB26]{LeB26}
C.~LeBrun,
\emph{An overview of Einstein $4$-manifolds}, preprint, 2026.
\href{https://arxiv.org/abs/2609.03033}{arXiv:2609.03033}.

\bibitem[LO26]{LO26}
C.~LeBrun and T.~Ozuch,
\emph{Desingularizations of conformally K\"ahler, Einstein orbifolds},
preprint (2026).
\href{https://arxiv.org/abs/2601.19215}{arXiv:2601.19215}.

\bibitem[Li24]{Li24}
M.~Li,
\emph{Classification results for conformally K\"ahler gravitational instantons}, preprint (2024).
\href{https://arxiv.org/abs/2310.13197v4}{arXiv:2310.13197v4}.

\bibitem[LS25]{LS25}
M.~Li and S.~Sun,
\emph{Gravitational instantons and harmonic maps}, preprint (2025).
\href{https://arxiv.org/abs/2507.15284}{arXiv:2507.15284}.

\bibitem[Lic63]{Lic63}
A.~Lichnerowicz,
\emph{Spineurs harmoniques}, C. R. Acad. Sci. Paris
\textbf{257} (1963), 7--9.

\bibitem[LPP04]{LPP04}
H.~L\"u, D.~N.~Page, and C.~N.~Pope,
\emph{New inhomogeneous Einstein metrics on sphere bundles over
Einstein--K\"ahler manifolds}, Phys. Lett. B \textbf{593} (2004), 218--226.
\href{https://doi.org/10.1016/j.physletb.2004.04.068}{doi:10.1016/j.physletb.2004.04.068}.

\bibitem[Maz91]{Maz91}
R.~Mazzeo, \emph{Elliptic theory of differential edge operators. I}, Comm. Partial Differential Equations \textbf{16} (1991), no.~10, 1615--1664.

\bibitem[MW17]{MW17}
R.~Mazzeo and H.~Weiss,
\emph{Teichm\"uller theory for conic surfaces},
in \emph{Geometry, Analysis and Probability: In Honor of Jean-Michel Bismut},
Progress in Mathematics \textbf{310}, Birkh\"auser, Cham, 2017, 127--164.
\href{https://doi.org/10.1007/978-3-319-49638-2_7}{doi:10.1007/978-3-319-49638-2\_7}.

\bibitem[Mil63]{Mil63}
J.~Milnor,
\emph{Morse Theory},
Annals of Mathematics Studies \textbf{51}, Princeton University Press, 1963.

\bibitem[Mon18]{Mon18}
I.~Mondello,
\emph{An Obata singular theorem for stratified spaces},
Trans. Amer. Math. Soc. \textbf{370} (2018), no.~6, 4147--4175.
\href{https://doi.org/10.1090/tran/7105}{doi:10.1090/tran/7105}.

\bibitem[My87]{My87}
R.~C.~Myers,
\emph{Higher-dimensional black holes in compactified space-times},
Phys. Rev. D \textbf{35} (1987), no.~2, 455--466.
\href{https://doi.org/10.1103/PhysRevD.35.455}{doi:10.1103/PhysRevD.35.455}.

\bibitem[Oba62]{Oba62}
M.~Obata,
\emph{Certain conditions for a Riemannian manifold to be isometric with
a sphere}, J. Math. Soc. Japan \textbf{14} (1962), no.~3, 333--340.
\href{https://doi.org/10.2969/jmsj/01430333}{doi:10.2969/jmsj/01430333}.

\bibitem[OR70]{OR70}
P.~Orlik and F.~Raymond,
\emph{Actions of the torus on $4$-manifolds. I},
Trans. Amer. Math. Soc. \textbf{152} (1970), no.~2, 531--559.
\href{https://doi.org/10.1090/S0002-9947-1970-0268911-3}{doi:10.1090/S0002-9947-1970-0268911-3}.

\bibitem[Ozu21]{Ozu21}
T.~Ozuch,
\emph{Higher order obstructions to the desingularization of Einstein metrics},
Camb. J. Math. \textbf{9} (2021), no.~4, 901--976.
\href{https://arxiv.org/abs/2012.13316}{arXiv:2012.13316}.

\bibitem[Ozu22a]{Ozu22a}
T.~Ozuch,
\emph{Noncollapsed degeneration of Einstein $4$-manifolds, I},
Geom. Topol. \textbf{26} (2022), no.~4, 1483--1528.
\href{https://arxiv.org/abs/1909.12957}{arXiv:1909.12957}.

\bibitem[Ozu22b]{Ozu22b}
T.~Ozuch,
\emph{Noncollapsed degeneration of Einstein $4$-manifolds II},
Geom. Topol. \textbf{26} (2022), no.~4, 1529--1634.


\bibitem[Ozu24]{Ozu24}
T.~Ozuch,
\emph{Integrability of Einstein deformations and desingularizations},
Comm. Pure Appl. Math. \textbf{77} (2024), no.~1, 177--220.
\href{https://arxiv.org/abs/2105.13193}{arXiv:2105.13193}.

\bibitem[Ozu26]{Ozu26}
T.~Ozuch,
\emph{Stability of gravitational instantons with a bounded
Killing vector field},
arXiv:2608.31008, 2026.

\bibitem[PR00]{PR00}
F.~Pacard and T.~Rivi\`ere,
\emph{Linear and nonlinear aspects of vortices: the Ginzburg-Landau model},
Progress in Nonlinear Differential Equations and their Applications \textbf{39}, Birkh\"auser, Boston, 2000.

\bibitem[Pag78]{Pag78}
D.~N.~Page,
\emph{A compact rotating gravitational instanton},
Phys. Lett. B \textbf{79} (1978), 235--238.

\bibitem[PP87]{PP87}
D.~N.~Page and C.~N.~Pope,
\emph{Inhomogeneous Einstein metrics on complex line bundles},
Class. Quantum Grav. \textbf{4} (1987), 213--225.
\href{https://doi.org/10.1088/0264-9381/4/2/005}{doi:10.1088/0264-9381/4/2/005}.

\bibitem[Pan23]{Pan23}
J.~Pan,
\emph{The Grushin hemisphere as a Ricci limit space with curvature $\geqslant1$},
Proc. Amer. Math. Soc. Ser. B \textbf{10} (2023), 71--75.
\href{https://arxiv.org/abs/2211.02747}{arXiv:2211.02747}.

\bibitem[Par86]{Par86}
J.~Parker,
\emph{$4$-dimensional $G$-manifolds with $3$-dimensional orbits},
Pacific J. Math. \textbf{125} (1986), no.~1, 187--204.


\bibitem[PW22]{PW22}
J.~Pan and G.~Wei,
\emph{Examples of Ricci limit spaces with non-integer Hausdorff dimension},
Geom. Funct. Anal. \textbf{32} (2022), 676--685.
\href{https://arxiv.org/abs/2106.03967}{arXiv:2106.03967}.

\bibitem[RS99]{RS99}
I.~Rivin and J.-M.~Schlenker,
\emph{The Schl\"afli formula in Einstein manifolds with boundary},
Electron. Res. Announc. Amer. Math. Soc. \textbf{5} (1999), 18--23.
\href{https://doi.org/10.1090/S1079-6762-99-00057-8}{doi:10.1090/S1079-6762-99-00057-8}.

\bibitem[RS00]{RS00}
I.~Rivin and J.-M.~Schlenker,
\emph{On the Schl\"afli differential formula}, preprint (2000).
\href{https://arxiv.org/abs/math/0001176}{arXiv:math/0001176}.


\bibitem[Tho69]{Tho69}
J.~A.~Thorpe,
\emph{Some remarks on the Gauss--Bonnet integral},
J. Math. Mech. \textbf{18} (1969), 779--786.

\bibitem[Tia90]{Tia90}
G.~Tian,
\emph{On Calabi's conjecture for complex surfaces with positive first
Chern class}, Invent. Math. \textbf{101} (1990), 101--172.
\href{https://doi.org/10.1007/BF01231499}{doi:10.1007/BF01231499}.

\bibitem[Tia15]{Tia15}
G.~Tian,
\emph{K-stability and K\"ahler--Einstein metrics},
Comm. Pure Appl. Math. \textbf{68} (2015), no.~7, 1085--1156;
corrigendum, \textbf{68} (2015), no.~11, 2082--2083.
\href{https://arxiv.org/abs/1211.4669}{arXiv:1211.4669}.

\bibitem[WZ90]{WZ90}
M.~Y.~Wang and W.~Ziller,
\emph{Einstein metrics on principal torus bundles},
J. Differential Geom. \textbf{31} (1990), no.~1, 215--248.
\href{https://doi.org/10.4310/jdg/1214444095}{doi:10.4310/jdg/1214444095}.

\bibitem[Web19]{Web19}
B.~Weber,
\emph{Energy and asymptotics of Ricci-flat $4$-manifolds with a Killing vector field},
Proc. Amer. Math. Soc. \textbf{147} (2019), no.~7, 3117--3130.
\href{https://doi.org/10.1090/proc/14014}{doi:10.1090/proc/14014}; revised version \href{https://arxiv.org/abs/1308.4997}{arXiv:1308.4997v2}.


\bibitem[Wei95]{Wei95}
G.~Weinstein,
\emph{On the Dirichlet problem for harmonic maps with prescribed singularities},
Duke Math. J. \textbf{77} (1995), no.~1, 135--165.
\href{https://arxiv.org/abs/dg-ga/9408005}{arXiv:dg-ga/9408005}.

\bibitem[Whi43]{Whi43}
H.~Whitney,
\emph{Differentiable even functions}, Duke Math. J. \textbf{10} (1943), 159--160.

\bibitem[Wis03]{Wis03}
T.~Wiseman,
\emph{Static axisymmetric vacuum solutions and
non-uniform black strings}, Class. Quantum Grav. \textbf{20}
(2003), no.~6, 1137--1176.
\href{https://doi.org/10.1088/0264-9381/20/6/308}{doi:10.1088/0264-9381/20/6/308}.

\bibitem[Yau78]{Yau78}
S.-T.~Yau,
\emph{On the Ricci curvature of a compact K\"ahler manifold and the complex Monge--Amp\`ere equation.~I},
Comm. Pure Appl. Math. \textbf{31} (1978), no.~3, 339--411.
\href{https://doi.org/10.1002/cpa.3160310304}{doi:10.1002/cpa.3160310304}.

\end{thebibliography}
\end{document}